\documentclass[hidelinks,onefignum,onetabnum]{siamart220329}

\usepackage{lipsum}
\usepackage{amsfonts}
\usepackage{graphicx}
\usepackage{epstopdf}
\usepackage{algorithmic}
\usepackage{amssymb,hyperref}
\usepackage[T1]{fontenc}
\usepackage[utf8]{inputenc}

\ifpdf
  \DeclareGraphicsExtensions{.eps,.pdf,.png,.jpg}
\else
  \DeclareGraphicsExtensions{.eps}
\fi

\newsiamremark{remark}{Remark}
\newsiamremark{hypothesis}{Hypothesis}
\crefname{hypothesis}{Hypothesis}{Hypotheses}
\newsiamthm{claim}{Claim}
\newcommand{\Kappa}{\mathrm{K}}
\newcommand{\assign}{:=}
\newcommand{\asterisk}{\mathord{*}}
\newcommand{\backassign}{=:}
\newcommand{\infixand}{\text{ and }}
\newcommand{\extend}{\displaystyle}
\newcommand{\mathi}{\mathrm{i}}
\newcommand{\nobracket}{}
\newcommand{\nosymbol}{}
\newcommand{\ontop}[2]{\genfrac{}{}{0pt}{}{#1}{#2}}

\newcommand{\tmcolor}[2]{{\color{#1}{#2}}}

\newcommand{\tmop}[1]{\ensuremath{\operatorname{#1}}}

\newcommand{\tmscript}[1]{\text{\scriptsize{$#1$}}}

\newcommand{\tmtextbf}[1]{\text{{\bfseries{#1}}}}
\newcommand{\tmtextit}[1]{\text{{\itshape{#1}}}}

\newcommand{\tmtexttt}[1]{\text{{\ttfamily{#1}}}}
\newcommand{\tmtextup}[1]{\text{{\upshape{#1}}}}
\newcommand{\tmxspace}{\hspace{1em}}
\newcommand{\dis}{\displaystyle}
\headers{Higeh Order Numerical Methods }{Diego Armentano and Jean-Claude Yakoubsohn
    }

\title{High Order Numerical Methods To 
    Approximate The Singular Value Decomposition\thanks{Submitted to the editors DATE.
}}

\author{Diego Armentano
\thanks{Instituto de Estadística, Facultad de Ciencias Económicas y de Administración, Universidad de
la República, Montevideo, Uruguay (\email{diego.armentano@fcea.edu.uy})}
  \and
  Jean-Claude Yakoubsohn
  \thanks{Institut de Math{\'e}matiques de Toulouse,
  Universit{\'e} Paul Sabatier,
  118, route de Narbonne,
  31062 Toulouse C{\'e}dex,
  France
 (\email{jean-claude.yakoubsohn@math.univ-toulouse.fr})}
  }
\usepackage{amsopn}

\ifpdf
\hypersetup{
  pdftitle={High Order Numerical Methods To \\
    Approximate The Singular Value Decomposition}{Diego rmentano and Jean-Claude Yakoubsohn},
  pdfauthor={DiegoArmentano and Jean-Claude Yakoubsohn}
}
\fi

\begin{document}

\maketitle

\begin{abstract}
In this paper, we present a class of high order methods to approximate the
  singular value decomposition of a given complex matrix (SVD). To the best of
  our knowledge, only methods up to order three appear in the the literature.
  A first part is dedicated to defline and analyse this class of method in the
  regular case, i.e., when the singular values are pairwise distinct. The
  construction is based on a perturbation analysis of a suitable
  system of
   associated to the SVD (SVD
  system). More precisely, for an integer $p$ be given, we define a sequence
  which converges with an order $p + 1$ towards the left-right singular
  vectors and the singular values if the initial approximation of the SVD
  system satisfies a condition which
  depends on three quantities : the norm of initial
approximation of the SVD system, the greatest singular
  value and the greatest inverse of the modulus of the difference
  between the singular values. From a numerical computational
  point of view, this furnishes a very efficient simple test to prove and
  certifiy the existence of a SVD in neighborhood of the initial
  approximation. We generalize these result in the case of clusters of
  singular values. We  show also how to use the result of regular case to
  detect the clusters of singular values and to define a notion of deflation
  of the SVD.  Moreover numerical experiments confirm the theoretical
  results.
\end{abstract}



\section{Introduction}

\subsection{Notations and main goal}

\ Let us consider an $m \times n$ complex matrix $M \in \mathbb{C}^{m \times
n}$ \ where we can assume $m \geqslant n$ without loss of generalty. The
terminology ``diagonal'' for a matrix of $\mathbb{C}^{m \times n}$ is
understood \ if it is of the form $\left(\begin{array}{c}
  \tmop{diag} (\sigma_1, \ldots, \sigma_n)\\
  0
\end{array}\right)$ and design by $\mathbb{D}^{m \times n}$ the set of such
type matrices and also $\mathbb{E}^{m \times \ell}_{n \times q} =\mathbb{C}^{m
\times \ell} \times \mathbb{C}^{n \times q} \times \mathbb{D}^{\ell \times
q}$. For $\ell \geqslant 1$, we denote the identity matrix in
$\mathbb{C}^{\ell \times \ell}$ by $I_{\ell}$ and for $W \in \mathbb{C}^{m
\times \ell}$ we define $E_{\ell} (W) = W^{\ast} W - I_{\ell}$. The variety of
Stiefel matrices is $\tmop{St}_{m, \ell} = \{ W \in \mathbb{C}^{m \times \ell}
: E_{\ell} (W) = 0 \}$. For each $\ell$, $1 \leqslant \ell \leqslant m$ and
$q$, $1 \leqslant q \leqslant n$, we know that there exists two Stiefel
matrices $U \in \tmop{St}_{m, \ell}$, $V \in \tmop{St}_{n, q}$, and a diagonal
matrix $\Sigma \in \mathbb{D}^{\ell \times q}_{\geqslant 0}$ be such that
\begin{align}
  \begin{array}{lll}
    f (U, V, \Sigma) & = \left( \begin{array}{c}
      E_{\ell} (U)\\
      E_q (V)\\
      U^{\ast} M V - \Sigma
    \end{array} \right) =
  \end{array} & 0.  \label{syst-svd}
\end{align}
When $\ell = m$ and $q = n$, the triplet $(U, V, \Sigma)$ is the classical
singular value decompsition (SVD) of the matrix $M$. If $\ell < m$ or $q < n$
this abbreviated version of the SVD is referred as the thin SVD. The problem
of computing a numerical thin SVD of $M$ is to approximate the triplet $(U, V,
\Sigma)$ by a sequence $(U_i, V_i, \Sigma_i,)_{i \geqslant 0}$ such that the
quantities $f (U_i, V_i, \Sigma_i)_{i \geqslant 0}$ converge to $0$. We name
\tmtextit{SVD sequence} a such type sequence $(U_i, \Sigma_i, V_i)_{i
\geqslant 0}$.

In the context of this paper we will say that a sequence $(T_i)_{i \geqslant
0}$ of a normed space with a norm $\| . \|$ converges to $T_{\infty}$ with an
order $p + 1 \geqslant 2$ if there exists a positive constant $c$ be such that
$\| T_i - T_{\infty} \| \leqslant c 2^{- (p + 1)^i + 1}$. We then say that the
numerical method which defines the sequence $(T_i)_{i \geqslant 0}$ is of
order $p + 1$. If $p = 1$ (respectively \ $p = 2$) we say that the method is
quadratic (respectively cubic). Finally we say that a method associated to a
map $H$ is of order $p$ if there exists a sequence $x_{k + 1} = H (x_k)$, $k
\geqslant 0$, which converges at the order $p$. Moreover we shall consider the
matrix norm $\| A \| = \max (\| A \|_1, \| A^{\ast} \|_1)$ where
\begin{align*}
  &\| A \|_1  \assign  \max_{1 \leqslant i \leqslant m} \dis\sum_{j = 1}^n |M_{i,
  j} | .
\end{align*}
Fundamental quantities occur throughout this study. From a triplet $(U, V,
\Sigma) \in \mathbb{E}^{m \times \ell}_{n \times q}$ we introduce :
\begin{enumerate}
  \item $\Delta = U^{\ast} M V - \Sigma .$
  \item $\kappa (\Sigma) = \max \left( 1, \max_{1 \leqslant i \leqslant q}
  \dfrac{1}{| \sigma_i |}, \hspace{0.2em} \max_{i \neq j} \left( \dfrac{1}{|
  \nobracket \sigma_i - \sigma_j | \nobracket} + \dfrac{1}{| \nobracket
  \sigma_i + \sigma_j | \nobracket} \right) \hspace{0.2em} \right)$ where the
  $\sigma_i$'s constitute the diagonal of $\Sigma$.
  \item $K (\Sigma) = \max \left( 1, \hspace{0.2em} \max_i \sigma_i
  \hspace{0.2em} \right)$.
\end{enumerate}
Throughout the text $p$ is a given integer greater or equal to one. The goal
of this paper is the construction and the convergence analysis of a class of
methods of order $p + 1$. The classical methods to compute the SVD are linear
or quadratic : to best of our knowledge, there is no mention of any study in
the literature on this subject of a method of order greater than three. These
methods only use matrix addition and multiplication : there is no linear
system to solve nor matrix to invert.

\subsection{Construction of a quadratic method}

We begin by explain how to construct a quadratic method to approximate the
SVD. Let us given $U, V, \Sigma$ and denote $\Delta = U^{\ast} M V - \Sigma$.
The first step is to consider\tmcolor{black}{ multiplicative perturbations}
such type $U \Omega$, $V \Lambda$ and $S$ \ of $U$, $V$, $\Sigma$ respectively
and also \ $U_2 = U_1 (I_{\ell} + X)$ and $V_2 = V_1 (I_q + Y)$
\tmcolor{black}{multiplicative perturbations} of $U_1 = U (I_{\ell} + \Omega)$
and $V_1 = V (I_q + \Lambda)$ respectively. Expanding the quantities $E_{\ell}
(U_1)$, $E_q (V_1)$ and $\Delta_2 \assign U_2^{\ast} M V_2 - \Sigma - S$, we
get
\begin{align}
  & E_{\ell} (U_1)  =  E_{\ell} (U) + \Omega + \Omega^{\ast} + \Omega^{\ast}
  E_{\ell} (U ) + E_{\ell} (U )  \Omega + \Omega^{\ast} \Omega + \Omega^{\ast}
  E_{\ell} (U) \Omega,  \label{linear-part-1}\\
  &   \tmop{idem} \tmop{for} E_q (V_1) \nonumber \\
  &\Delta_2  = 
    \Delta_1 - S  + X^{\ast} \Sigma + \Sigma Y + X^{\ast} \Delta_1 + \Delta_1
    {Y + X^{}}^{\ast} (\Delta_1 + \Sigma) Y.
 \label{linear-part-2}
\end{align}
where $\Delta_1 = U_1^{\ast} M V_1 - \Sigma$. Denoting $\varepsilon = \max (\|
E_{\ell} (U) \|, | | E_q (V) | |, \| \Delta  \|)$, the second step is to
determine two Hermitian matrices $\Omega$, $\Lambda$, a diagonal matrix $S$,
and two skew Hermitian matrices $X$, $Y$ in order to get
\begin{align}
  &\max (| | E_{\ell} (U_2) | |, | | E_q (V_2) | |, | | \Delta_2 | |) 
  \leqslant  O (\varepsilon^2) .  \label{eq-max}
\end{align}
This occurs with $\Omega = - E_{\ell} (U) / 2$, \ $\Lambda = - E_q (V) / 2$
and $(X, Y, S)$ a solution of the equation $\Delta_1 - S + X^{\ast} \Sigma +
\Sigma Y = 0$. We will give in section \ref{sec-svd-diag} explicit formulas to
solve this the linear equation where a solution is given by $S = \tmop{diag}
(\Delta_1)$ and $X$, $Y$ that are two skew Hermitian matrices. In fact a
straighforward calculation shows that
\begin{align}
  &E_{\ell} (U_1)  =  - (3 I_{\ell} + 2 \Omega) \Omega^2  \label{eq-emu1}\\
&\tmop{idem}
  \tmop{for} E_q (V_1) \nonumber \\
 &   \Delta_1  =  \Delta + \Omega  (\Delta + \Sigma) + (\Delta + \Sigma) \Omega
  + \Omega  (\Delta + \Sigma) \Omega  \label{eq-delta1}\\
 & \Delta_2 =  - X  \Delta_1 + \Delta_1 Y - X^{} (\Delta_1 + \Sigma) Y
  \qquad \tmop{since} X^{\ast} = - X  \label{eq-delta2}\\
  &E_{\ell} (U_2) =  (I_{\ell} - X) E_{\ell} (U_1) (I_{\ell} + X) +
  (I_{\ell} - X) (I_{\ell} + X) - I_{\ell}    \label{eq-emu2}
\\
&  \tmop{idem} \tmop{for} E_q
  (V_2) \nonumber.\end{align}
The formula $(\nobracket$\ref{eq-emu1}-\ref{eq-delta1}$\nobracket)$ imply $\|
E_{\ell} (U_1) \| \leqslant O (\varepsilon^2)$ and $\| \Delta_1 \| \leqslant O
(\varepsilon )$. Similarly we have $\| E_q (V_1) \| \leqslant O
(\varepsilon^2)$. Moreover we will prove that $\| X \|, \| Y \| \leqslant O
(\varepsilon)$ in section \ref{sec-svd-diag}. Plugging these estimates in the
formulas $(\nobracket$\ref{eq-delta2}-\ref{eq-emu2}) we find that the
inequality $\left( \ref{eq-max} \right) $ holds. From the point of view of the
complexity this step is the key point of the methods presented here since this
requires no matrix inversion. These ingredients pave the way for the
construction of a quadradic method. The third step is to introduce the map
\begin{align*}
&  H_1 (U, V, \Sigma)  =  \left( \begin{array}{c}
    U  (I_{\ell} + \Omega )  (I_{\ell} + X)\\
    V (I_q + \Lambda )  (I_q + Y)\\
    \Sigma + S 
  \end{array} \right)
\end{align*}
where $\Omega  = - \dfrac{1}{2} E_{\ell} (U)$, $\Lambda  = - \dfrac{1}{2} E_q
(V)$, $S \in \mathbb{D}^{m \times n}$ is a diagonal matrix and $X$, $Y$are
skew Hermitian matrices\quad be such that $\Delta_1 - S - X \Sigma + \Sigma Y
= 0$. The behaviour of the sequence $\left( U_i, V_i {, \Sigma_i}  \right)_{i
\geqslant 0}$ defined by $(U_{i + 1}, V_{i + 1}, \Sigma_{i + 1}) = H_1 (U_i,
V_i, \Sigma_i)$, $i \geqslant 0$ is given by Theorem \ref{th-svd-main}.

\begin{remark}
  The Newton's method is based on the cancellation of the affine part of a
  Taylor expansion closed to a root of the function. Here we remark that only
  the cancellation of a part of the affine part is enough to build a numerical
  quadratic method. For instance in the expression $\left( \ref{linear-part-1}
  \right)$, we cancel $E_{\ell} (U) + \Omega + \Omega^{\ast}$ rather than
  $E_{\ell} (U) + \Omega + \Omega^{\ast} + \Omega^{\ast} E_{\ell} (U ) +
  E_{\ell} (U )  \Omega$. In the same way $\Delta_1 - S  + X^{\ast} \Sigma +
  \Sigma Y$ is cancelled rather than $\Delta_1 - S  + X^{\ast} \Sigma + \Sigma
  Y + X^{\ast} \Delta_1 + \Delta_1 Y$ in the expression $\left(
  \ref{linear-part-2} \right)$.
\end{remark}

\subsection{Construction of a method of order $p + 1$}

We explain the main ideas that allow to generalize the previous method with
the care to improve the condition of convergence. Taking in account the
formulas $\left( \ref{eq-emu1} \right.$-$\left. \ref{eq-emu2} \right)$ we
notice that to generalize the previous construction we need the following
tools. We first require a method of order $p + 1$ to approximate the variety
of Stiefel matrices. This is realized in considering a multiplicative
perturbation $U s_p (\Omega)$ of $U$ where \ $s_p (u)$ is an univariate
polynomial of degree $p$ in order that $U_1 = U (I_{\ell} + s_p (\Omega))$
satisfies $E_{\ell} (U_1) = O (E_{\ell} (U)^{p + 1})$. This is motivated by
$\left( \ref{eq-emu1} \right)$. Next we introduce a multiplicative
perturbation $U_1 c_p (U_1)$ where $c_p (u)$ is an univariate polynomial of
degree $p$ such that $(1 + c_p (- u)) (1 + c_p (u)) - 1 = O (u^{p + 1})$. This
is motivated by $\left( \ref{eq-emu2} \right)$ where appears the expression
$(I_{\ell} - X) (I_{\ell} + X) - I_{\ell}$. The polynomials $s_p (u)$ and $c_p
(u)$ as well as the matrices $\Omega$ and $X$ are defined respectively below
and their properties will be precisely studied in sections \ref{sec-unitary}
and \ref{sec-proof-main-th}. Under these previous conditions a we will prove
in Section 3 that a perturbation such type $U_2 = U  (I_{\ell} + s_p (\Omega))
(I_{\ell} + c_p (X))$ satisfies $E_{\ell} (U_2) = O (E_{\ell} (U)^{p + 1})$.
Finally the third tool is to determine $X$, $Y$, and $S$ in order to get the
condition $| | \Delta_{p + 1} | | = O (| | \Delta | |^{p + 1})$ where
$\Delta_{p + 1} = U_2^{\ast} M V_2 - \Sigma - S$.

To introduce the map on which is based the method of order $p + 1$ we define
the following quantities:
\begin{enumerate}
  \item Let $s_p (u)$ the truncated polynomial of degree $p$ of the series
  expansion of $- 1 + (1 + u^2)^{- 1 / 2}$.
  
  \item Let $c_p (u)$ the truncated polynomial of degree $p$ of the series
  expansion of $(1 + u^2)^{1 / 2} + u - 1$.
\end{enumerate}
With these preliminaries we introduce the map $H_p$ :
\begin{align}
  &(U, V, \Sigma) \in \mathbb{E}^{m \times n}  \rightarrow & H_p (U, V,
  \Sigma) = \left(\begin{array}{c}
    U (I_{\ell} + \Omega )  (I_{\ell} + \Theta )\\
    V (I_q + \Lambda ) (I_q + \Psi )\\
    \Sigma  + S
  \end{array}\right) \in \mathbb{E}^{m \times n}  \label{map-Hp}
\end{align}
where :
\begin{enumerate}
  \item $\Omega  = s_p (E_{\ell} (U) )$ and $\Lambda  = s_p (E_q (V) ) .$
  
  \item $\Theta = c_p (X)$ and $\Psi = c_p (Y)$ where $X$ and $Y$ are defined
  below.
  
  \item $S = S_1 + \cdots + S_p \in \mathbb{D}^{m \times n}$, $X = X_1 +
  \cdots + X_p$ and $Y = Y_1 + \cdots + Y_p$ with each $X_k$, $Y_k$ are skew
  Hermitian matrices in $\mathbb{C}^{\ell \times \ell}$ and $\mathbb{C}^{q
  \times q}$ respectively. Moreover each triplet $(S_k, X_k, Y_k)$ are
  solutions of the following linear systems :
  \begin{align}
    \Delta_k - S_k - X_k \Sigma + \Sigma Y_k & = 0, \qquad 1 \leqslant k
    \leqslant p  \label{eq-SkXkYk}
  \end{align}
  where the $\Delta_k$'s for $2 \leqslant k \leqslant p + 1$, are defined as
  \begin{align}   
      &\Delta_1  =  (I_{\ell} + \Omega) (\Delta + \Sigma) (I_q + \Lambda) -
    \Sigma, \qquad S_1 = \tmop{diag} (\Delta_1) \nonumber\\
    &
    \Theta_k  =  c_p (X_1 + \cdots + X_k), \hspace{1.5em} \Psi_k = c_p (Y_1
    + \cdots + Y_k), \hspace{1em} 1 \leqslant k \leqslant p,
     \nonumber\\
    &\Delta_k  =  (I_{\ell} + \Theta_{k - 1}^{\ast}) (\Delta_1 + \Sigma) (I_q
    + \Psi_{k - 1}) - \Sigma - \dis\sum_{j = 1}^{k - 1} S_j,
        \label{def-Deltai}
\\
    &S_k  =  \tmop{diag} (\Delta_k), \hspace{0.5em} 2 \leqslant k \leqslant p.
    \nonumber
  \end{align}
\end{enumerate}
We will see in section \ref{sec-proof-main-th} that the formulas $\left(
\ref{eq-SkXkYk} \right)$ cancel respectively the linear parts of each
$\Delta_k$. We will show that \ $| | \Delta_{p + 1} | | = O (| | \Delta_1 |
|^{p + 1})$.

\subsection{Main result}

Then we state the folowing result which precisely shows the method associated
to the map $H_p$ is of order $p + 1$.

\begin{theorem}
  \label{th-svd-main} Let $p \geqslant 1$. From $(U_0, V_0, \Sigma_0)$, let us
  define the sequence
  \[ (U_{i + 1}, V_{i + 1}, \Sigma_{i + 1}) = H_p (U_i, V_i, \Sigma_i), \quad
     i \geqslant 0. \]
  We denote $\Delta = U_0^{\ast} M V_0 - \Sigma_0$, $K = K (\Sigma_0)$ and
  $\kappa = \kappa (\Sigma_0)$. We consider the constants defined in Table
  \ref{table_constants} :
  
  \begin{table}[h]
    $$\begin{array}{|c|c|c|c|}
      \hline
      & p = 1 & p = 2 & p \geqslant 3\\
      \hline
      a & 2 & 4 / 3 & 4 / 3\\
      \hline
      u_0 & 0.0289 & 0.046 & 0.0297\\
      \hline
      \gamma_1 & 6.1 & 9.41 & 10.2\\
      \hline
      \sigma & 1.67 & 2.1 & 2.62\\
      \hline
    \end{array}$$
    \caption{\label{table_constants}}
  \end{table}
  If
  \begin{align}
  &  \max ((\kappa K)^a  | | E_{\ell} (U_0) | |, (\kappa K)^a, | | E_q (V_0) |
    |, \kappa^a K^{a - 1}  | | \Delta_0 | |) = \varepsilon  \leqslant  u_0 
    \label{test-svd-main}
  \end{align}
  then the sequence $(U_i, V_i, \Sigma_i)_{i \geqslant 0}$ \ converges to a
  solution $(U_{\infty}, V_{\infty}, \Sigma_{\infty})$ of system
  $(\nobracket$\ref{syst-svd}$\nobracket)$ with an order of convergence equal
  to $p + 1$. More precisely we have for $i \geqslant 0$ :
  \begin{align*}
    &\|U_i - U_{\infty} \|  \leqslant  \gamma_1  \sqrt{\ell} 2^{- (p + 1)^i +
    1} \varepsilon\\
    &\|V_i - F_{\infty} \|  \leqslant  \gamma_1  \sqrt{q} 2^{- (p + 1)^i + 1}
    \varepsilon\\
   & \| \Sigma_i - \Sigma_{\infty} \|  \leqslant  \sigma \times 2^{- (p +
    1)^i + 1} \varepsilon .
  \end{align*}
\end{theorem}

\subsection{Arithmetic Complexity}

The computation of $H_p (U, V, \Sigma)$ only requires matrix additions and
multiplications without resolution of linear systems. This is possible since
there are explicit formulas for the equations $\left( \ref{eq-SkXkYk}
\right)$. Table \ref{table_complexity} gives the number of addition and
multiplications to evaluate $H_p (U, V, \Sigma)$ where $L_k:=\Delta_k - S_k - X_k
    \Sigma + \Sigma Y_k$.

\begin{table}[h]
  $$\begin{array}{|c|c|c|c|c|c|c|}
    \hline
    & E_{\ell} (U) & s_p (E_{\ell} (U)) & c_p (X) & L_k & S_k & \Delta_k\\
    \hline
    \begin{array}{c}
      \tmop{matrix}\\
      \tmop{additions}
    \end{array} & 1 & p & p^2 &  & p & \\
    \hline
    \begin{array}{c}
      \tmop{matrix}\\
      \tmop{multiplications}
    \end{array} & 1 & p & p^2 &  &  & 2 p + 2\\
    \hline
    \tmop{additions} &  &  &  & 10 n p &  & (m + 4 n) p\\
    \hline
    \tmop{multiplications} &  &  &  & (m - n + 8) n p &  & (m + n) m n p\\
    \hline
  \end{array}$$
  \caption{\label{table_complexity}}
\end{table}

This implies $2 (p + 1) (m^2 + n^2) + (m + 14 n) p  $ \ additions and \ $2 (p
+ 1) (m^3 + n^3) + (m^2 + m n + m - n + 8) n p $ multiplications.

\subsection{Outline of this paper}

In section \ref{sec-related-works} we give a short overview on the
computational methods for the SVD and we discuss about the method of
Davies-Smith to update the SVD. We exhibit the links with the method
associated to the map $H_2$. We also state a result on Davies-Smith method
which will be proved in section \ref{sec-davies}. \ In section
\ref{sec-unitary} we study the approximation of the unitary group by high
order methods. We will use the polynomial $s_p (u)$ to define the sequence
$U_{i + 1} = U_i (I_{\ell} + s_p (E_{\ell} (U_i)))$, $i \geqslant 0$, from a
matrix $U_0$ closed to the unitary group. The result is that under condition
$\| E_{\ell} (\nobracket U_0) | | < 1 / 4$ the sequence $(U_i)_{i \geqslant
0}$ converges to the polar projection of $U_0$. In section \ref{sec-svd-diag}
we show how to explicitely solve the equation $\Delta - S - X \Sigma + \Sigma
Y = 0$. We also state a condition-like result that shows the quantity $\kappa$
is the condition number of this resolution. In fact we will prove that : $\| X
\|, \| Y \| \leqslant \kappa \| \Delta \|$. This bound plays a signifiant role
in the convergence analysis. The section \ref{sec-proof-main-th} is devoted to
the convergence analysis. We introduce the notion of $p$-map for the SVD. This
is convenient to states in Theorem \ref{general-result} that the method
associated to a $p$-map is of order $p + 1$. Then Theorem \ref{th-svd-main}
derives from Theorem \ref{general-result}. The proof is done in sections
\ref{proof-p=1}, \ref{proof-p=2} and \ref{proof-p=3} for $p = 1$, $p = 2$, and
$p = 3$ respectively. In section \ref{sec-clusters}, we study the case of
clusters of singular values and we show how to use the condition
$(\nobracket$\ref{test-svd-main}$\nobracket)$ to separate clusters of singular
values. We introduce a notion of deflation for the SVD : the idea is to
compute a thin SVD with one singular value per cluster. Finally we illustrate
this by numerical experiments in section \ref{sec-numerical-experiments}.

\section{Related works and discussion}\label{sec-related-works}

\subsection{Short overview on the SVD and the \ methods to compute it}

``The practical and theoretical importance of the SVD is hard to
overestimate''.This sentence from Golub and Van Loan {\cite{Golub-Loan13}} \
perfectly sums up the role of SVD in science and more particularly in the
world of computation. The SVD was discovered by Belrami in 1873 and Jordan in
1874, see the historical survey of Stewart {\cite{stewart93}} that traces the
contributions of Sylvester, Schmidt and Weyl, the first precursors of the SVD.
A recent overview of numerical methods for the SVD can be found in the Hanbook
of Linear Algebra {\cite{hogben2013}} mainly in chapters 58 and 59. On the
aspects developments on modern computers, Dongarra and all
{\cite{dongarra2018}} give a survey of algorithms and their implementations
for dense and tall matrices with comparison of performances of most
bidiagonalization and Jacobi type methods. From a numerical linear algebra
point of view, the SVD is at the center of the significant problems. Let us
mention a few : the generalized inverse of a matrix
{\cite{israel-greville03}}, the best subspace problem {\cite{golub-loan80}},
the orthogonal Procrustes problem {\cite{elden-park}}, the linear least square
problem {\cite{Golub-Loan13}}, the low rank approximation
problem{\cite{Golub-Loan13}}. Finally, a very stimulating article of Martin
and Porter {\cite{martin-porter}} describes the vitality of SVD in all areas
by showing surprising examples.

There are two classes of methods to compute the SVD : bidiagonalizations
methods and Jacobi methods. Since the time of precursors, Golub and Kahan in
1965 {\cite{golubkahan65}} for bidiagonalization with QR iteration \ and
Kogbeliantz in 1955 {\cite{kog55}} for Jacobi two-sided method, many various
evolutions and ameliorations have been proposed. In our context $(m \geqslant
n)$, the bidiagonalzation methods reduce first the complex matrix under the
form $M = U M' V^{\ast}$ where $U$, $V$are unitary and $M'$ real and upper
bidiagonal {\cite{do-so-ha}}. Next the SVD is computed roughly by QR iteration
with notable improvements as implicit zero-shift QR {\cite{demmel-kahan}} and
differential qd algorihms {\cite{fernando-parlett}}. In this vein of
bidiagonalization methods, other alternatives to QR iteration have been
developped. Let us mention the divide and conquer methods {\cite{gu95}},
{\cite{gates2018}}, {\cite{li2014}}, the bisection and inverse iteration
methods {\cite{ipsen97}}, {\cite{hogben2013}} in chapter 55 and methods based
on multiple relatively robust representation {\cite{dhillon2004}},
{\cite{willems2006}}. The Jacobi methods consist to successively apply
rotations now called Givens rotations on the left and right of the original
matrix in order to eliminate a pair of elements at each steps. Wilkinson
{\cite{wilkinson62}} proves that the method is ultimately quadratic for the
eigenvalue problem. After Kogbetliantz, the properties of two-sided Jacobi
method applying two different rotations has been studied a lot : global
convergence {\cite{fernando1989}}, {\cite{forsythe1960}}, quadratic
convergence at the end of the algorithm {\cite{paige1986}}, {\cite{bai1988}},
behaviour in presence of clusters {\cite{charlier1987}}, reliability and
accuracy {\cite{drmacveselic108}}, {\cite{drmacveselic208}}, {\cite{hari2009}},
{\cite{matejas2010}}, {\cite{matejas2015}}. Let us also mention main
improvements for the one-sided Jacobi method due to several forms of
preconditionning {\cite{drmacveselic108}}, {\cite{drmacveselic208}} and
{\cite{drmac2017}} which uses a preconditionner QR to get high accuracy for
the SVD. Finally the simultaneous use of block Jacobi methods and
preconditionning \ improve convergence {\cite{becka15}}, {\cite{oksa19}} and
computing time {\cite{dongarra2018}}.

Other ways have been investigated related to classical topics studied in the
field of numerical analysis. For instance, Chatelin {\cite{chatelin1984}}
studies the Newton method for the eigenproblem. This requires a resolution of
a Sylvester equation. Since the resolution of Sylvester is expensive, several
variants of Newton method are proposed but the quadratic convergence is lost.
There is also the purpose of Edelman et al. {\cite{edelman1998}} which
explores the geometry of Grassmann and Stiefel \ manifolds \ in the context of
numerical algorithms and propose Newton method in this context. It also
requires to solve a Sylvester equation \ to get numerical results. These ideas
also have been developped by Absil et al. {\cite{absil2009}} in the context of
the optimization on manifolds. Finally let us mention \ differential \ point
of view developped by Chu {\cite{chu1986}} where an O.D.E. is derived for the
SVD in the context of bidiagonal matrices. The methods mentioned above have a
most quadratic order of convergence.

To end this short review, let us mention methods based on the computation of the polar decomposition to deduce the SVD by spectral factorization of a Hermitian matrix. This idea has been introduced by Highm and Papadimitriou~\cite{HP1994}. More recently in 2014, Higham and Nakatsukasva improve this method to get the QDWH-SVD algorithm ~\cite{HN2012}, ~\cite{HN2014}. For a more complete history of the SVD we can refer to the chapter 4 section 3 of the very documented book {\em A Journey through the History of Numerical Linear Algebra}~\cite{BMRZ2022}.

\subsection{The Davies-Smith method}The method of Davies and Smith
{\cite{daviessmith04}} to update the singular decomposition of matrices in
$\mathbb{R}^{m \times n}$ is probably the closest study to our. In our
framework of notations, it consists to \ define the map
\begin{align}
  &(U, V, \Sigma)  \rightarrow  \tmop{DS}  (U, V, \Sigma) =
  \left(\begin{array}{c}
    U \left( I_{\ell} + X + \frac{1}{2} X_1^2 \right) {\backassign U \Gamma_1}
    \\
    V \left( I_q + Y + \frac{1}{2} Y_1^2 \right) \backassign V \Kappa_1\\
    \Sigma   + S \backassign \Sigma_1
  \end{array}\right)  \label{DSF}
\end{align}
with $S = S_1 + S_2$, $X = X_1 + X_2$, $Y = Y_1 + Y_2$ where \ the $S_i$'s, $i
= 1, 2$, \ are \ diagonal matrices, the $X_i$'s and $Y_i$'s are skew Hermitian
matrices that verify
\begin{align}
  &X_1 \Sigma   - \Sigma   Y_1 + S_1  =  \Delta_1 \assign \Delta = U^{\ast}
  MV - \Sigma   \label{X1DS}\\
  &X_2 \Sigma   - \Sigma   Y_2 + S_2  =  \Delta_2 \assign - \frac{1}{2} X_1 
  (\Delta  + S_1 ) + \frac{1}{2} (\Delta + S_1 ) Y_1 .  \label{X2DS}
\end{align}
This gives an approximation at the order three of the SVD in the regular case
under the condition that the quantity $\| \Delta + \Sigma \|$ is small enough.
More precisely Davies and Smith states that if the condition $\kappa^3
\varepsilon^3 \leqslant \tmop{tol}$ where $\tmop{tol}$ is a given tolerance
then ${U \Gamma_1}   \Sigma K_1^{\ast} V_1^{\ast}$is an approximation of the
SVD of $M$, such that :
\begin{enumerate}
  \item $\| E_{\ell} (U \Gamma_1) \|, \| E_q (\nobracket V K_1)  \leqslant 2
  (\kappa \varepsilon)^3 + O (\kappa^4 \varepsilon^4)$.
  
  \item $\dfrac{1}{\| M \|} \| \Gamma_1^{\ast} U^{\ast} M V K_1 - \Sigma_1 \| 
  \leqslant \dfrac{28}{3} (\kappa \varepsilon)^3 + O (\kappa^4 \varepsilon^4)$
\end{enumerate}
where the considered norm is that of Frobenius. Thanks to the map $H_p$
defined in the introduction with $p = 2$ , we improve the previous method and
its analysis on several points.
\begin{enumerate}
  \item The norm of $E_{\ell} (U (I_{\ell} + \Omega ) (I_q + \Theta ))$ is in
  $O (\varepsilon^3)$, see Theorem \ref{th-DS-revisited} below, while the norm
  of $E_{\ell} (U \Gamma_1)$ depends on the norm of $E_{\ell} (U)$. In fact
  \begin{align*}
    E_{\ell} (U \Gamma_1) & = \Gamma_1^{\ast} E_{\ell} (U) \Gamma_1 +
    E_{\ell} (\Gamma_1) .
  \end{align*}
  For this reason, Davies and Smith suggest to use a Givens type method after
  their update of the SVD to iterate the method.
  
  \item Note that $\Theta_2 = X_1 + X_2 + \dfrac{1}{2}  (X_1 + X_2)^2$ is
  computed with the same arithmetic complexity as $\Gamma_1$. There is a gain
  in the error analysis.
  
  \item The analysis of the map $H_2$ takes in account all the terms of the
  series expansion of $H_2 (U, V, \Sigma)$ with respect $U, V, \Sigma$. In
  this way, the Theorem \ref{th-DS-revisited} show that $\kappa^{5 / 4} K^{2 /
  5} \varepsilon$ (and not $\kappa \varepsilon$) is the quantity on which the
  method Davies Smith rests. \tmcolor{black}{This shows that the quantity $K$
  is not negligible in the error analysis.}
  
  \item The tolerance $\tmop{tol}$ in the method associated to the map $H_p$
  is determined by imposing a condition of contraction which is not the case
  in the Davies-Smith method, see the algorithm 2.3 of {\cite{daviessmith04}}.
\end{enumerate}
We defined a Davies-Smith revisited method introducing the map
\begin{align}
&  (U, V, \Sigma)  \rightarrow  \overline{\tmop{DS}}  (U, V, \Sigma) =
  \left(\begin{array}{c}
    U (I_{\ell} + \Theta_2)\\
    V (I_q + \Psi_2)\\
    \Sigma   + S \backassign \Sigma_1
  \end{array}\right)  \label{DSF-revisited}
\end{align}
with $S = S_1 + S_2$, $X = X_1 + X_2$, $Y = Y_1 + Y_2$ where \ the $S_i$'s, $i
= 1, 2$, \ are \ diagonal matrices, the $X_i$'s and $Y_i$'s are skew Hermitian
matrices defined by $\left( \ref{X1DS} \right.$-$\left. \ref{X2DS} \right)$.
The following result specifies the behaviour of $\tmop{DS} (U, V, \Sigma)$ and
$\overline{\tmop{DS}} (U, V, \Sigma)$.

\begin{theorem}
  \label{th-DS-revisited}
  
Let us consider $M$, $U$, $V$, $\Sigma$ as in the
  introduction, $\Delta  = U^{\ast} M V - \Sigma$ and $\varepsilon_1 = \|
  \Delta  \|$. Let $\kappa = \kappa (\Sigma)$ and $K = K (\Sigma)$.
  \begin{enumerate}
    \item Let us assume that $\kappa^{5 / 4} {K^{2 / 5}}  \varepsilon_1
    \leqslant \varepsilon \leqslant 0.1.$ Then \ the triplet $(U_1, V_1,
    \Sigma_1) = \tmop{DS}  (U, V, \Sigma)$ defined by $\left( \ref{DSF}
    \right)$ satisfies
    \begin{align}
      &\| \Delta_1 \| \assign \| U_1^{\ast} M V_1 - \Sigma_1 \|  \leqslant 
      (8 + 18 \varepsilon + 33 \varepsilon^2) \varepsilon ^3 . 
      \label{bnd-DS1}
    \end{align}
    \item Let us assume that $\kappa^{6 / 5} {K^{3 / 10}}  \varepsilon_1
    \leqslant \varepsilon \leqslant 0.1.$ Then \ the triplet $(\bar{U}_1,
    \bar{V}_1, \bar{\Sigma}_1) = \overline{\tmop{DS}}  (U, V, \Sigma)$ defined
    by $\left( \ref{DSF-revisited} \right)$ satisfies
    \begin{align}
&      \| \bar{\Delta}_1 \| \assign \| \bar{U}_1^{\ast} M \bar{V}_1 -
      \bar{\Sigma}_1 \|  \leqslant  (6 + 21 \varepsilon + 54 \varepsilon^2)
      \varepsilon ^3 .  \label{bnd-DS2}
    \end{align}
    
  \end{enumerate}
\end{theorem}

Since $\kappa^{6 / 5} {K^{3 / 10}}  < \kappa^{5 / 4} {K^{2 /
5}} $, the \ condition to update the singular value decomposition is better
with the Davies Smith method revisited than the Davies Smith method.

\section{Approximation of Stiefel matrices}\label{sec-unitary}

The Stieffel manifold $\tmop{St}_{m, \ell}$ generalizes the Unitary group. An
important tool is the polar decomposition $U_0 = \pi (U_0) H$ of rectangular
matrix $U_0$ where the polar projection $\pi (U_0)$ is a Stiefel matrix and
$H$ is Hermitian positive semidefinite {\cite{horn2012}}. It is also well
known that $\pi (U_0)$ is indeed the closest element in $\tmop{St}_{m, l}$ to
$U_0$ for every unitarily norm~{\cite[Theorem~1]{FH55}}. Since we are doing
approximate computations, the Stiefel matrices in an SVD are not given
exactly, so we may wish to estimate the distance between an approximate
Stiefel \ matrix and the closest actual Stiefel matrix. This is related to the
following problem: given an approximately Stiefel $m \times \ell$ matrix $U$,
find a good approximation $U + \dot{U}$ for its projection on the manifold
$\tmop{St}_{m, \ell}$. We define a class of high order iterative methods for
this problem \ and provide a detailed analysis of its \ convergence, see also
~{\cite{Kov70,BB71,Higham89}}. The theorem \ref{unit-proj-tm} establishes that
our method converges towards the polar projection of the matrix $U_0 \in
\mathbb{C}^{m \times \ell}$ if $U_0$ is sufficiently close to the Stiefel
manifold. In this case the matrix $H$ is positive definite and can uniquely be
written as the exponential of another Hermitian matrix. 

\subsection{A class of high order iterative methods}

We wish to compute $\dot{U}$ using an appropriate Newton iteration. Since the
normal space in $U$ of Stiefel manifol is composed of $U \Omega$'s where
e$\Omega$ is an Hermitian matrix,it turns out that it is more convenient to
write $U + \dot{U} = U (I_{\ell} + \Omega)$. The following lemma gives the
expression $\Omega$ so that $U + \dot{U} \in \tmop{St}_{m, \ell}$ it is the
polar projection of $U$.

\begin{lemma}
  \label{polar-lem}Let $U \in \mathbb{C}^{m \times \ell}$ such that the
  spectral radius of $E_{\ell} (U)$ is strictly less than $1$. Then
  \begin{align}
&    \Omega = - I_{\ell} + (I_{\ell} + E_{\ell} (U))^{- 1 / 2}  \Rightarrow 
    E_{\ell} (U + U \Omega) = 0.  \label{EmUDelta}
  \end{align}
  Hence $U (I_{\ell} + E_{\ell} (U))^{- 1 / 2} \in \tmop{St}_{m, \ell}$ is the
  polar projection of $U$.
\end{lemma}

\begin{proof}
  If the spectral radius of $E_{\ell} (U)$ is strictly less than $1$ then the
  matrix $(I_{\ell} + E_{\ell} (U))^{1 / 2}$ exists and $\Omega = - I_{\ell} +
  (I_{\ell} + E_{\ell} (U))^{- 1 / 2}$ is Hermitian positive definite matrix.
  With $E_{\ell} (U) = U^{\asterisk} U - I_{\ell}$ and $\dot{U} = U \Omega$,
  we have
  \begin{align*}
    E_{\ell}  (U + U \Omega) & =  (I_{\ell} + \Omega^{\ast})  (I_{\ell} +
    E_{\ell} (U))  (I_{\ell} + \Omega) - I_{\ell}\\
    & =  E_{\ell} (U) + 2 \Omega + \Omega E_{\ell} (U) + E_{\ell} (U) \Omega
    + \Omega^2 + \Omega E_{\ell} (U) \Omega .
  \end{align*}
  A straighforward calculation implies \ $E_{\ell}  (U + U \Omega) = 0$. Then
  $U = U (I_{\ell} + \Omega)  (I_{\ell} + \Omega)^{- 1}$. Hence $U (I_{\ell} +
  E_{\ell} (U))^{- 1 / 2} \in \tmop{St}_{m, \ell}$ is the polar projection of
  $U$.
\end{proof}

Consequently an high order approximation of $\Omega = - I_{\ell} + (I_{\ell} +
E_{\ell} (U))^{- 1 / 2}$ will permit to define an high order method to
numerically compute the polar projection. Evidently $\Omega$ commutes with
$U$. The approximation of $\Omega$ can be obtained as follows. Let us consider
the Taylor serie \ of $- 1 + (1 + u)^{- 1 / 2}$ at $u = 0$ :
\begin{align*}
  s (u) & = \dis\sum_{k \geqslant 1}  (- 1)^k \frac{1}{4^k} \left(
  \begin{array}{c}
    2 k\\
    k
  \end{array} \right) u^k = - \frac{1}{2} u + \frac{3}{8} u^2 - \frac{5}{16}
  u^3 + \cdots \quad
\end{align*}
For $p \geqslant 1$ we introduce $\dis s_p (u) = \dis\sum_{k = 1}^p (- 1)^k t_k u^k$
and $r_p (u) = s (u) - s_p (u)$. The quantities
\begin{align}
  &\Omega_p =  s_p (E_{\ell} (U)), \quad R_p = r_p (E_{\ell} (U)) 
  \label{Deltap}
\end{align}
commute with $U^{\ast} U$. We have $\Omega_p = \Omega - R_p$ and $E_{\ell} (U
+ U \Omega) = 0$. A straightforward calculation shows that
\begin{align}
  E_{\ell}  (U + U \Omega_p) & =  (U ^{\ast} + \Omega_p U^{\ast} - R_p
  U^{\ast}) (U + U \Omega_p - U R_p) - I_{\ell} \nonumber\\
  & =  E (U + U \Omega) - 2 (I_{\ell} + \Omega) U^{\ast} U R_p + R_p^2
  U^{\ast} U \nonumber\\
  & =  (I_{\ell} + E_{\ell} (U)) R_p  (- 2 I_{\ell} - 2 \Omega + R_p) \qquad
  \tmop{since} U^{\ast} U = I_{\ell} + E_{\ell} (U)  \label{EUUpt}
\end{align}
We are thus lead to the iteration that we will further study below:
\begin{align}
&  U_{i + 1}  =  U_i  (I_{\ell} + s_p (E_{\ell} (\nobracket U_i) ), \qquad i
  \geqslant 0.  \label{Newton}
\end{align}
Theorem \ref{unit-proj-tm} below shows the convergence of the sequence
(\ref{Newton}) towards the polar projection of $U_0$ with a $p$ order of
convergence under the universal condition $\|E (U_0)\| < 1 / 4$.

\subsection{Error analysis}

\begin{proposition}
  \label{unit-proj-prop}Let $p \geqslant 1$. Let $U$ be an $m \times \ell$
  matrix with $\varepsilon \assign \|E_{\ell} (U)\| < 1$ and $\Omega_p = s_p
  (E_{\ell} (U))$. Let $U_1 = U (I_{\ell} + \Omega)$ \ and write
  $\varepsilon_1 \assign \|E_{\ell} (U_1)\|$. Then $\| \Omega_p \| \leqslant |
  s_p (\varepsilon) | \leqslant - 1 + (1 - \varepsilon)^{- 1 / 2}$ and
  \begin{align}
    \varepsilon_1 & \leqslant  \varepsilon^{p + 1} .  \label{ineq}
  \end{align}
\end{proposition}

\begin{proof}
  Let $\Omega_p = s_p (E_{\ell} (U))$. We have
  \begin{align*}
    \| \Omega_p \| & \leqslant  | s_p (\varepsilon) |\\
    & \leqslant  - 1 + (1 - \varepsilon)^{- 1 / 2} .
  \end{align*}
  Since $\Omega$ is Hermitian which commutes with $U$ we have
  \begin{align*}
    E_{\ell} (U_1) & =  (I_{\ell} + \Omega_p) U^{\ast} U (I_{\ell} +
    \Omega_p) - I_{\ell}\\
    & =  (I_{\ell} + \Omega_p)^2 E_{\ell} (U) + \Omega_p^2 + 2 \Omega_p\\
    & =  (I_{\ell} + E_{\ell} (U)) (\Omega_p^2 + 2 \Omega_p) + E_{\ell} (U).
  \end{align*}
  Then using Lemma \ref{lem-hp} below in sub-section, it follows easily that
  \begin{align*}
    E_{\ell} (U_1) & =  \left( \dis\sum_{k = 0}^p \alpha_k E_{\ell} (U)^k \right)
    E_{\ell} (U)^{p + 1}
  \end{align*}
  where $\dis \dis\sum_{k = 0}^p | \nobracket \alpha_k | \nobracket \leqslant 1$. Hence
  $\varepsilon_1 \leqslant \varepsilon^{p + 1}$.
\end{proof}

Proposition \ref{unit-proj-prop} permits to analyse the behaviour of the
sequence $(U_i)_{i \geqslant 0}$ deftined by $\left( \ref{Newton} \right)$.

\begin{theorem}
  \label{unit-proj-tm}let $p \geqslant 1$. Let $U_0 \in \mathbb{C}^{m \times
  \ell}$ be such that $\|E (U_0)\| \leqslant \varepsilon < 1 / 2$. Then the
  sequence defined by
  \begin{align}
    U_{i + 1} & = U_i (I_{\ell} + s_p (E (\nobracket U_i)) \qquad i
    \geqslant 0,  \label{seq-Ui}
  \end{align}
  converges to a Stiefel matrix $U_{\infty} \in \tmop{St}_{m, \ell}$. More
  precisely, for all $i \geqslant 0$, we have
  \begin{align}
    \|U_i - U_{\infty} \|  & \leqslant   \sqrt{\ell} 
    \frac{2^{- (p + 1)^i + 1} 2 \varepsilon}{1 - 2 \varepsilon } . 
    \label{UipiU}
  \end{align}
  Moreover if $\varepsilon < 1 / 4$ then this sequence converges to the polar
  projection $\pi (U_0) \in \tmop{St}_{m, \ell}$ of $U_0$.
\end{theorem}

\begin{proof}
  The Newton sequence~(\ref{seq-Ui}) defined from $U_0 = U$ gives
  \begin{align*}
    U_{i + 1} & = U_0  (I_{\ell} + \Omega_{0, p}) \cdots (I_{\ell} +
    \Omega_{i, p})
  \end{align*}
  with $\Omega_{i, p} = s_p (E_{\ell} (U_i) )$. An obvious induction using
  Proposition~\ref{unit-proj-prop} yields $\|E_{\ell} (U_i) \| \leqslant 2^{-
  (p + 1)^i + 1} \varepsilon$. In fact we have
  \begin{align*}
    | | E_{\ell} (U_{i + 1}) | | & \leqslant | | E_{\ell} (U_i) | |^{p + 1}
    \hspace{3em} \tmop{from} \tmop{Proposition} \ref{unit-proj-prop}\\
    & \leqslant 2^{- (p + 1)^{i + 1} + p + 1} \varepsilon^{p + 1}\\
    & \leqslant (2 \varepsilon)^p 2^{- (p + 1)^{i + 1} + 1} \varepsilon \\
    & \leqslant 2^{- (p + 1)^{i + 1} + 1} \varepsilon  \qquad \tmop{since}
    \quad \varepsilon < 1 / 2.
  \end{align*}
  We are using Lemma \ref{conv-Ui} to conclude. We have $\| \Omega_{k, p} \|
  \leqslant - 1 + (1 - 2^{- (p + 1)^k + 1} \varepsilon)^{- 1 / 2}$. Since
  $\varepsilon \leqslant 1 / 2$ then \ $- 1 + (1 - 2^{- (p + 1)^k + 1}
  \varepsilon)^{- 1 / 2} \leqslant 2^{- (p + 1)^k + 1} \varepsilon$.
  Considering  $u_0 = \varepsilon$, $\alpha_1 = 1$ and $\alpha_2 = 0$, the
  assumptions of Lemma \ref{conv-Ui} below are satisfied. Hence the sequence
  $(U_i)_{i \geqslant 0}$ converges to a matrix $U_{\infty}$ which is an unitary
  matrix since
  the sequence $(\nobracket E_{\ell} (U_i)_{i \geqslant 0}$ converges towards
  $0$. We then have
  \begin{align*}
    \| U_i - U_{\infty} \| & \leqslant \sqrt{\ell} \frac{2 (\alpha_1 +
    \alpha_2 + \alpha_1 \alpha_2 u_0)}{1 - 2 (\alpha_1 + \alpha_2 + \alpha_1
    \alpha_2 u_0) u_0} 2^{- (p + 1)^i + 1} \alpha_{_0} \varepsilon\\
    & \leqslant \sqrt{\ell}  \frac{2^{- (p + 1)^i + 1} 2 \varepsilon}{1 - 2
    \varepsilon } .
  \end{align*}
  We denote $Z_0 = \prod_{j \geqslant 0} (I_{\ell} + \Omega_{j, p})$. We have
  $U_{\infty} = U_0 Z_0$. From \ Lemma \ref{conv-Ui} $Z_0$ is invertible with
  $\| Z_0 \| \leqslant 2 \varepsilon .$ By induction on $i$, it can also be
  checked that all the $\Omega_{i, p}$'s commute. Whence $Z_0$ and $Z_0^{- 1}$
  are actually Hermitian matrices. If $\varepsilon < 1 / 4$ we have ${\|
  Z_0^{- 1} - I_{\ell} \|} \leqslant \| Z_0^{- 1} \|  \| I_{\ell} - Z_0 \|
  \leqslant 2 \varepsilon / (1 - 2 \varepsilon) < 1$. Then the logarithm $\log
  Z_0^{- 1}$ is well defined. We conclude that $Z_0^{- 1}$ is the exponential
  of a Hermitian matrix, whence it is positive-definite. Since $U_0 =
  U_{\infty} Z_0^{- 1}$, we conclude that $U_{\infty} = \pi (U_0)$ the polar
  projection of $U_0$ from the polar decomposition theorem.
\end{proof}
\subsection{Technical Lemmas}

This following Lemma is used in the proof of Proposition \ref{unit-proj-prop}.

\begin{lemma}
  \label{lem-hp}Let $p \geqslant 1$. We have
  \begin{align*}
    (u + 1) (s_p (u)^2 + 2 s_p (u)) + u & = \left( \dis\sum_{k = 0}^p \alpha_k
    u^k \right) u^{p + 1}
  \end{align*}
  where $\dis \sum_{k = 0}^p | \alpha_k | \leqslant 1$.
\end{lemma}

\begin{proof}
  Let $t_i = (- 1)^i \dfrac{1}{4^i} \left( \begin{array}{c}
    2 i\\
    i
  \end{array} \right)$ for $i \geqslant 0$. The convolution of sequence
  binomial $t_i$ with itself is the sequence with general terms (-1)$^i$. In
  fact it is sufficient to square $(1 + u)^{- 1 / 2}$:
  \begin{align*}
    \frac{1}{1 + u} = \dis\sum_{k \geqslant 0} (- 1)^k u^k & = \dis\sum_{k \geqslant
    0} \left( \dis\sum_{i + j = k} t_i t_j \right) u^k .
  \end{align*}
  We proceed by induction. When $p = 1$ the lemma holds since
  \begin{align*}
    (u + 1) (h_1 (u)^2 + 2 h_1 (u)) + u & = (u + 1) \left( \frac{u^2}{4} - u
    \right) + u\\
    & = \left( - \frac{3}{4} + \frac{1}{4} u \right) u^2
  \end{align*}
  and $\dfrac{1}{4} + \dfrac{3}{4} = 1$. Let us suppose that the lemma holds
  for an indice $p \geqslant 1$ be given. We first remark that $\alpha_0 = - 2
  t_{p + 1}$. In fact since $\alpha_0$ is the coefficient of $u^{p + 1}$ in
  $(u + 1) (s_p (u)^2 + 2 s_p (u)) + u$. Then
  \begin{align*}
    \alpha_0 & = \dis\sum_{\tmscript{\begin{array}{c}
      i + j = p\\
      1 \leqslant i, j \leqslant p
    \end{array}}} t_i t_j + \dis\sum_{\tmscript{\begin{array}{c}
      i + j = p + 1\\
      1 \leqslant i, j \leqslant p
    \end{array}}} t_i t_j + 2 t_p\\
    & = (- 1)^p - 2 t_0 t_p + (- 1)^{p + 1} - 2 t_0 t_{p + 1} + 2 t_p\\
    & = - 2 t_{p + 1} .
  \end{align*}
  Next, writing $h_{p + 1} (u) = s_p (u) + t_{p + 1} u^{p + 1}$ we get by
  straightforward calculations :
  \begin{align*}
    (u + 1) (s_p (u)^2 &+ 2 s_p (u)) + u   \\
      &\hspace{-2cm} =\left( \dis\sum_{k = 0}^p \alpha_k
    u^k \right) u^{p + 1}+ (u + 1) (2 t_{p + 1} s_p (u) u^{p + 1} + t_{p + 1}^2 u^{2 (p + 1)}
    + 2 t_{p + 1} u^{p + 1})\\
    &\hspace{-2cm} = (\alpha_1 + 2 t_{p + 1} (t_1 + 1)) u^{p + 2} + \dis\sum_{k = 2}^p
    (\alpha_k + 2 t_{p + 1} (t_k + t_{k - 1})) u^{p + k + 1}\\
      &\hspace{-1.5cm} + t_{p + 1} (2 t_p + t_{p + 1}) u^{2 (p + 1)} + t_{p + 1}^2 u^{2 p +
    3}\\
    & \hspace{-2cm}\assign  \left( \dis\sum_{k = 0}^{p + 1} \beta_k u^k \right) u^{p + 2}
  \end{align*}
  Let us prove that $\dis\sum_{k = 0}^{p + 1} | \nobracket \beta_k | \nobracket
  \leqslant 1$. In fact since $t_1 = - 1 / 2$ and $\dis\sum_{k = 1}^p | \nobracket
  \alpha_k | \nobracket = 1 - 2 | \nobracket t_{p + 1} | \nobracket$ it
  follows:
  \begin{align*}
    \sum_{k = 0}^{p + 1} | \nobracket \beta_k | \nobracket & \leqslant
    \sum_{k = 1}^p | \nobracket \alpha_k | \nobracket + | t_{p + 1} | + 2 |
    \nobracket t_{p + 1} | \nobracket \sum_{k = 2}^p (| t_{k - 1} | - | t_k |)
    + | \nobracket t_{p + 1} | \nobracket (2 | t_p | - | t_{p + 1} |) + t_{p +
    1}^2\\
    & \leqslant 1 - 2 | \nobracket t_{p + 1} | \nobracket + | t_{p + 1} | +
    2 | \nobracket t_{p + 1} | \nobracket (| \nobracket t_1 | \nobracket - |
    t_p | ) + | \nobracket t_{p + 1} | \nobracket (2 | t_p | - | t_{p + 1} |)
    + t_{p + 1}^2\\
    & \leqslant 1.
  \end{align*}
  The Lemma is proved.
\end{proof}

The \ following Lemma $\ref{u2j}$ is used in the proof of Lemma \ref{conv-Ui}.

\begin{lemma}
  \label{u2j}
  \begin{enumerate}
    \item Let $0 \leqslant u < 1$. We have $\prod_{j \geqslant 0} (1 + u^{2^j}
    ) = \dfrac{1}{1 - u} .$
    
    \item Let $p \geqslant 1$ and $0 \leqslant \varepsilon < 1.$ We have for
    $i \geqslant 0$,
    \begin{align}
      \prod_{j \geqslant 0} (1 + 2^{- (p + 1)^{j + i} + 1} \varepsilon) &
      \leqslant  1 + 2^{- (p + 1)^i + 1} 2 \varepsilon  \label{ineq-lem-eps}
    \end{align}
    \item \label{lem-eps}Let $p \geqslant 1$ and $0 \leqslant \varepsilon
    \leqslant 1 / 2.$ We have for $i \geqslant 0$,
    \begin{align}
      \prod_{j \geqslant 0} (1 - 2^{- (p + 1)^{j + i} + 1} \varepsilon)^{- 1 /
      2} & \leqslant 1 + 2^{- (p + 1)^i + 1} 2 \varepsilon 
      \label{ineq-lem-eps1}
    \end{align}
  \end{enumerate}
  
\end{lemma}

\begin{proof}
  For the item 1 we prove by induction that $\prod_{j = 0}^k (1 + u^{2^j} ) =
  \cfrac{1 - u^{2^{k + 1}}}{1 - u}$. This holds when $k = 0$. Next, assuming
  the property for $k$ be given we have
  \begin{align*}
    \prod_{j = 0}^{k + 1} (1 + u^{2^j} ) & = \cfrac{1 - u^{2^{k + 1}}}{1 -
    u} (1 + u^{2^{k + 1}})\\
    & = \frac{1 - u^{2^{k + 2}}}{1 - u} .
  \end{align*}
  Item 1 is proved. The item 2 follows from
  \begin{align*}
    \prod_{j \geqslant 0} (1 + 2^{- (p + 1)^{j + i} + 1} \varepsilon) &
    \leqslant  \prod_{j \geqslant 0} (1 + (2^{- (p + 1)^i})^{2^j} 2
    \varepsilon) \\
    & \leqslant 1 + \left( \prod_{j \geqslant 0} (1 + (2^{- (p +
    1)^i})^{2^j} ) - 1 \right) 2 \varepsilon\\
    & \leqslant 1 + \left( \frac{1}{1 - 2^{- (p + 1)^i}} - 1 \right) 2
    \varepsilon \qquad \tmop{from} \tmop{item} 1.\\
    & \leqslant 1 + 2^{- (p + 1)^i} 4 \varepsilon .
  \end{align*}
  Since $\varepsilon \leqslant 1 / 2$ we have \ $(1 - u )^{- 1 / 2} \leqslant
  1 + u$, item 3 follows from :
  \begin{align*}
    \prod_{j \geqslant 0} (1 - 2^{- (p + 1)^{j + i} + 1} \varepsilon)^{- 1 /
    2} & \leqslant \prod_{j \geqslant 0} (1 + 2^{- (p + 1)^{i + j} + 1}
    \varepsilon) \\
    & \leqslant 1 + 2^{- (p + 1)^i + 1} 2 \varepsilon \quad \tmop{from}
    \tmop{item} 2.
  \end{align*}
  
\end{proof}

The Lemma \ref{conv-Ui} is used in Theorems \ref{unit-proj-tm} and
\ref{general-result}.

\begin{lemma}
  \label{conv-Ui}Let $\varepsilon$, $u_0$, and \ $\alpha_i$, $i = 1, 2$, be
  real numbers such that $\varepsilon \leqslant u_0$ and $2 (\alpha_1 +
  \alpha_2 + \alpha_1 \alpha_2 u_0) u_0 < 1$. Let us consider a sequence of
  matrices defined by
  \begin{align*}
    U_{i + 1} & = U_i (I_{\ell} + \Omega_i)  (I_l + \Theta_i), \qquad i
    \geqslant 0,
  \end{align*}
  where the norms of the $\Omega_i$'s and the $\Theta_i$'s satisfy
  \begin{align*}
&    \| \Omega_i \| \leqslant \alpha_1 2^{- (p + 1)^i + 1} \varepsilon \qquad
    \tmop{and} \qquad \| \Theta_i \| \leqslant \alpha_2 2^{- (p + 1)^i + 1}
    \varepsilon .
  \end{align*}
  Then the sequence $(U_i)_{i \geqslant 0}$ converges to a matrix
  $U_{\infty}$. If $U_{\infty}$ is an unitary matrix then each $U_i$ is
  invertible and we have
  \begin{align*}
    \|U_i - U_{\infty} \| & \leqslant \sqrt{\ell} \frac{2 (\alpha_1 +
    \alpha_2 + \alpha_1 \alpha_2 u_0)}{1 - 2 (\alpha_1 + \alpha_2 + \alpha_1
    \alpha_2 u_0) u_0} 2^{- (p + 1)^i + 1} \varepsilon .
  \end{align*}
  Moreover each $N_i = \prod_{j \geqslant 0}  (I_{\ell} + \Omega_{i + j}) 
  (I_{\ell} + \Theta_{i + j})$ is invertible and satisfies
  \begin{align*}
    \| N_i - I_{\ell} \| & \leqslant 1 - 2 (\alpha_1 + \alpha_2 + \alpha_1
    \alpha_2 u_0) u_0 .
  \end{align*}
  \begin{align*}
    &  & 
  \end{align*}
\end{lemma}

\begin{proof}
  We remark that $U_i = U_0 \prod_{j = 0}^{i - 1} (I_{\ell} + \Omega_j) 
  (I_{\ell} + \Theta_j)$. Let $N_i = \prod_{j \geqslant 0}  (I_{\ell} +
  \Omega_{i + j})  (I_{\ell} + \Theta_{i + j})$. Let us consider $U_{\infty} =
  U_0 N_0$. From assumption we know that $\| \Omega_j \| \leqslant \alpha_1
  2^{- (p + 1)^j + 1} \varepsilon$ and $\| \Theta_k \| \leqslant \alpha_2 2^{-
  (p + 1)^j + 1} \varepsilon$. Taking in account that $\varepsilon \leqslant
  u_0$, it follows
  
  \begin{align*}
    (1 + \| \Omega_{i + j} \|) (1 + \| \Theta_{i + j} \|) & \leqslant 1 +
    (\alpha_1 + \alpha_2 + \alpha_1 \alpha_2 u_0) \times 2^{- (p + 1)^{i + j}
    + 1} \varepsilon .
  \end{align*}
  
  The matrix $N_i - I_{\ell}$ is written an infinite sum of homogeneous
  polynomials of degree $k \geqslant 1$:
  \begin{align*}
    N_i - I_{\ell} & = \sum_{k \geqslant 1} P_k (\Omega_i, \ldots, \Omega_{i
    + j}, \ldots \Theta_i, \ldots, \Theta_{i + j}, \ldots)
  \end{align*}
  Consequently for $i \geqslant 0$ we have :
  
  \begin{align*}
    \|N_i - I_{\ell} \| & \leqslant \sum_{k \geqslant 1} P_k (\| \Omega_i \|,
    \ldots \| \Omega_{i + j} \|, \ldots, \| \Theta_i \|, \ldots, \| \Theta_{i
    + J} \|, \ldots)\\
    & \leqslant \prod_{j \geqslant 0} (1 + \| \Omega_{i + j} \|) (1 + \|
    \Theta_{i + j} \|) - 1\\
    & \leqslant \prod_{j \geqslant 0} (1 + (\alpha_1 + \alpha_2 + \alpha_1
    \alpha_2 u_0) \times 2^{- (p + 1)^{i + j} + 1} \varepsilon) - 1\\
    & \leqslant 2 (\alpha_1 + \alpha_2 + \alpha_1 \alpha_2 u_0) 2^{- (p +)^i
    + 1} \varepsilon \quad \tmop{from} \tmop{Lemma} \ref{lem-p1j}\\
    & \leqslant 2 (\alpha_1 + \alpha_2 + \alpha_1 \alpha_2 u_0) u_0 \quad
    \tmop{since} \quad \varepsilon \leqslant u_0
  \end{align*}
  
  Since $2 (\alpha_1 + \alpha_2 + \alpha_1 \alpha_2 u_0) u_0 < 1$ it follows
  that each $N_i$ is invertible. Since $U_{\infty} = U_0 N_0$ it is easy to
  see
  \begin{align*}
    \| U_{\infty} \| & \leqslant  \| U_0 \|  (1 + 2 (\alpha_1 + \alpha_2 +
    \alpha_1 \alpha_2 u_0) \varepsilon) .
  \end{align*}
  We have $U_i = U _{\infty} N_i^{- 1}$. We deduce that
  \begin{align*}
    \|U_i - U_{\infty} \| & \leqslant \| U_{\infty} N_i^{- 1} (I_{\ell} -
    N_i) \|\\
    & \leqslant \| U_{\infty} \| \frac{1}{1 - 2 (\alpha_1 + \alpha_2 +
    \alpha_1 \alpha_2 u_0) u_0} 2 (\alpha_1 + \alpha_2 + \alpha_1 \alpha_2
    u_0) 2^{- (p + 1)^i + 1} \varepsilon .
  \end{align*}
  If $U_{\infty}$ is an unitary matrix then each $U_i$ is invertible and $\|
  U_{\infty} \| \leqslant \sqrt{\ell}$. The result is proved.
\end{proof}

\begin{lemma}
  \label{lem-perturb-polar}From $U_0 \in \mathbb{C}^{m \times \ell}$ be given,
  let us define the sequence for $i \geqslant 0$, $U_{i + 1} = U_i (I_{\ell} +
  \Omega_{i, p})$ with $\Omega_{i, p} = s_p (E_{\ell} (U_i))$. Let
  $\varepsilon = \| E_{\ell} (U_0) \|$. Then we have
  \begin{align*}
    \| \Omega_{i, p} \| & \leqslant (- 1 + (1 - \varepsilon)^{- 1 / 2})
    \varepsilon^{(p + 1)^i - 1}
  \end{align*}
\end{lemma}

\begin{proof}
  From Proposition \ref{unit-proj-prop} we know that $\| E_{\ell} (U_i) \|
  \leqslant \varepsilon^{(p + 1)^i}$. Since $s_p (u) \leqslant - 1 + (1 -
  u)^{- 1 / 2}$ we can write \ $\| \Omega_{i, p} \| \leqslant - 1 + (1 -
  \varepsilon^{(p + 1)^i})^{- 1 / 2}$. The function $u \rightarrow
  \dfrac{1}{u} (- 1 + (1 - u)^{- 1 / 2})$ is defined and is increasing on $[0,
  1]$. We then find that
  \begin{align*}
    \| \Omega_{i, p} \| & \leqslant \frac{1}{\varepsilon} (- 1 + (1 -
    \varepsilon)^{- 1 / 2}) \varepsilon^{(p + 1)^i} .
  \end{align*}
  We are done.
\end{proof}
\section{SVD for perturbed diagonal matrices}\label{sec-svd-diag}
\subsection{Solving the equation $\Delta - S - X \Sigma + \Sigma Y = 0$}

The following proposition shows how to explicitly solve this linear equation
under these constraints without inverting a matrix.

\begin{proposition}
  \label{xij=0} Let $\Sigma = \tmop{diag} (\sigma_1, \ldots \sigma_q) \in
  \mathbb{D}^{\ell \times q}$ and $\Delta = (\delta_{i, j}) \in
  \mathbb{C}^{\ell \times q}$. Consider the diagonal matrix $S \in
  \mathbb{D}^{\ell \times q}$ and the two skew Hermitian matrices $\smash{X} =
  (x_{i, j}) \in \mathbb{C}^{\ell \times \ell}$ and $Y = (y_{i, j}) \in
  \mathbb{C}^{q \times q}$ that are dend the tfined by the following formulas
  :
  \begin{itemize}
    \item For $1 \leqslant i \leqslant q$, we take
    \begin{align}
      S_{i, i} & = \tmop{Re} \delta_{i, i}  \label{pert-it1}\\
      x_{i, i} & = - y_{i, i} \hspace{1.2em} = \frac{\tmop{Im} \delta_{i,
      i}}{2 \sigma_i} \mathi  \label{pert-it2}
    \end{align}
    \item For $1 \leqslant i < j \leqslant q$, we take
    \begin{align}
      x_{i, j} & = \frac{1}{2} \left( \frac{\delta_{i, j} +
      \overline{\delta_{j, i}}}{\sigma_j - \sigma_i} + \frac{\delta_{i, j} -
      \overline{\delta_{j, i}}}{\sigma_j + \sigma_i} \right) 
      \label{pert-it4}\\
      y_{i, j} & = \frac{1}{2} \left( \frac{\delta_{i, j} +
      \overline{\delta_{j, i}}}{\sigma_j - \sigma_i} - \frac{\delta_{i, j} -
      \overline{\delta_{j, i}}}{\sigma_j + \sigma_i} \right)  \label{pert-it5}
    \end{align}
    \item For $q + 1 \leqslant i \leqslant \ell$ and $1 \leqslant j \leqslant
    q$, we take
    \begin{align}
      x_{i, j} & = \frac{1}{\sigma_j} \delta_{i, j} .  \label{pert-it8}
    \end{align}
    \item For $q + 1 \leqslant i \leqslant \ell$ and $q + 1 \leqslant j
    \leqslant \ell$, we take
    \begin{align}
      x_{i, j} & = 0.  \label{pert-it9}
    \end{align}
  \end{itemize}
  Then we have
  \begin{align}
    \Delta - S - X \Sigma + \Sigma Y & = 0  \label{obj-eq}
  \end{align}
\end{proposition}

\begin{proof}
  Since $X$ and $Y$ are skew Hermitian matrices, we have $\tmop{diag}
  (\tmop{Re} (X \Sigma - \Sigma Y)) = 0$. In view of~(\ref{pert-it1}), we thus
  get
  \begin{align*}
    \tmop{diag} (\tmop{Re} \Delta) & = \tmop{diag} \tmop{Re} (X \Sigma -
    \Sigma Y + S) .
  \end{align*}
  By skew symmetry, for the equation
  \[ X \Sigma - \Sigma Y = \tmop{diag} (\tmop{Re} \Delta) = \Delta - S \]
  holds, it is sufficient to have
  \begin{align}
    \sigma_i x_{i, i} - \sigma_i y_{i, i} & = \extend{\mathi \tmop{Im}
    \delta_{i, i}},\qquad\mathord{1 \leqslant i \leqslant q} .
    \label{eq-XptYpt1}\\
    \left( \begin{array}{cc}
      \sigma_i x_{i, i} & \sigma_j x_{i, j}\\
      - \sigma_i  \overline{x_{i, j}} & \sigma_j x_{j, j}
    \end{array} \right) & - \left( \begin{array}{cc}
      \sigma_i y_{i, i} & \sigma_i y_{i, j}\\
      - \sigma_j  \overline{y_{i, j}} & \sigma_j y_{j, j}
    \end{array} \right)\label{eq-XptYpt2} \\
    & = \extend{\left( \begin{array}{cc}
      \mathi \tmop{Im} \delta_{i, i} & \delta_{i, j}\\
      \delta_{j, i} & \mathi \tmop{Im} \delta_{j, j}
    \end{array} \right)},\quad \mathord{1 \leqslant i < j
    \leqslant q} \nonumber \\
    \sigma_j x_{i, j} & = \delta_{i, j}, \qquad \mathord{q + 1
    \leqslant i \leqslant \ell, \quad 1 \leqslant j \leqslant q} . 
    \label{eq-XptYpt3}
  \end{align}
  The formulas~(\ref{pert-it2}) clearly imply~(\ref{eq-XptYpt1}). The $x_{i,
  j}$ from~(\ref{pert-it4}) clearly satisfy~(\ref{eq-XptYpt3}) as well. For $1
  \leqslant i < j \leqslant q$, the formulas~(\ref{eq-XptYpt2}) can be
  rewritten as
  \begin{align*}
    \left( \begin{array}{cc}
      \sigma_j & - \sigma_i\\
      - \sigma_i & \sigma_j
    \end{array} \right) \left( \begin{array}{c}
      \tmop{Re} x_{i, j}\\
      \tmop{Re} y_{i, j}
    \end{array} \right) & = \left( \begin{array}{c}
      \tmop{Re} \delta_{i, j}\\
      \tmop{Re} \delta_{j, i}
    \end{array} \right)\\
    \left( \begin{array}{cc}
      \sigma_j & - \sigma_i\\
      \sigma_i & - \sigma_j
    \end{array} \right) \left( \begin{array}{c}
      \tmop{Im} x_{i, j}\\
      \tmop{Im} y_{i, j}
    \end{array} \right) & = \left( \begin{array}{c}
      \tmop{Im} \delta_{i, j}\\
      \tmop{Im} \delta_{j, i}
    \end{array} \right) .
  \end{align*}
  Since $\sigma_i > \sigma_j$, the formulas~(\ref{pert-it4}--\ref{pert-it5})
  indeed provide us with a solution. The entries $x_{i, j}$ with $q + 1
  \leqslant i, j \leqslant \ell$ do not affect the product $X \Sigma$, so they
  can be chosen as in~(\ref{pert-it9}). In view of the skew symmetry
  constraints $x_{j, i} = - \overline{x_{i, j}}$ and $y_{j, i} = -
  \overline{y_{i, j}}$, we notice that the matrices~$X$ and~$Y$ are completely
  defined.
\end{proof}

\begin{definition}
  \label{def-condition} Let $\Sigma = \tmop{diag} (\sigma_1, \ldots \sigma_q)
  \in \mathbb{D}^{\ell \times q} $ and $\Delta \in \mathbb{C}^{\ell \times
  q}$. We name condition number of equation $X \Sigma - \Sigma Y = \Delta - S$
  the quantity
  \begin{align}
    \kappa = \kappa (\Sigma) & = \max \left( 1, \max_{1 \leqslant i
    \leqslant q} \frac{1}{\sigma_i}, \max_{1 \leqslant i < j \leqslant q} 
    \dfrac{1}{\sigma_i - \sigma_j} + \dfrac{1}{\sigma_i + \sigma_j}  \right) 
    \label{def-kappa}
  \end{align}
\end{definition}

\subsection{Error analysis}

\begin{proposition}
  \label{fund-pert-prop} Under the notations and assumptions of Proposition
  \ref{xij=0}, assume that $X, Y$ and $S$ are computed using
  \tmtextup{(\ref{pert-it1}--\ref{pert-it5})}. Given $\varepsilon$ with $\|
  \Delta \|  \leqslant \varepsilon$, the matrices $X$, $Y$ and $S$ solutions
  of $\Delta - S - X \Sigma + \Sigma Y = 0$ satisfy
  \begin{align}
    \|S\| & \leqslant \varepsilon  \label{dSigma-bnd}\\
    \|X\|, {\|Y\| }  & \leqslant \kappa \varepsilon  \label{dU-bnd}
  \end{align}
\end{proposition}

\begin{proof}
  From the formula~(\ref{pert-it1}) we clearly have $\|S\| \leqslant \| \Delta
  \| \leqslant \varepsilon$.
  
  Since $\Sigma \in \mathbb{D}^{\ell \times q}$ we know that $\sigma_i >
  \sigma_j$ for $i < j$. It follows
  \begin{align*}
    | x_{i, j} | & \leqslant \dfrac{| \delta_{i, j} |}{2}  \left(
    \dfrac{1}{\sigma_i - \sigma_j} + \dfrac{1}{\sigma_i + \sigma j} \right) +
    \dfrac{| \overline{\delta_{i, j}} |}{2}  \left( \dfrac{1}{\sigma_i -
    \sigma_j} + \dfrac{1}{\sigma_i + \sigma_j} \right)\\
    & \leqslant \kappa | \delta_{i, j} | \qquad \tmop{since} \quad |
    \delta_{i, j} | = | \overline{\delta_{i, j}} | .
  \end{align*}
  We also have \ $| x_{i, i} | \leqslant \dfrac{| \delta_{i, i} |}{\sigma_i}$
  and for $q + 1 \leqslant i \leqslant \ell$ and $1 \leqslant j \leqslant q$,
  \ $| x_{i, i} | \leqslant \dfrac{| \delta_{i, i} |}{\sigma_j}$. Combined
  with the fact that $\| \Delta \|  \leqslant \varepsilon$, we get $\|X\|
  \leqslant \kappa \varepsilon$. In the same way we also have $\|Y\| 
  \leqslant \kappa \varepsilon$.
\end{proof}

\section{Convergence analysis : a general result}\label{sec-proof-main-th}

\begin{definition}
  \label{def-mapH} Let an integer $p \geqslant 1$. Let $\delta = 1$ if $p$ is
  odd and $\delta = 2$ if $p$ is even. Let us consider the map
  \begin{align}
    (U, V, \Sigma) \in \mathbb{E}^{m \times \ell}_{n \times q} & \rightarrow &
    H (U, V, \Sigma) = \left(\begin{array}{c}
      U (I_{\ell} + \Omega) (I_{\ell} + \Theta)\\
      V (I_q + \Lambda) (I_q + \Psi)\\
      \Sigma  + S
    \end{array}\right) \in \mathbb{E}^{m \times \ell}_{n \times q} 
    \label{map-H}
  \end{align}
  where $\Omega, \Lambda$ are Hermitian matrices, $S$ a diagonal matrix and
  $\Theta, \Psi$ are skew Hermitian matrices. Let $\Delta = U^{\ast} M V -
  \Sigma$ and $\Delta_1 = (I_{\ell} + \Theta^{\ast}) (I_{\ell} + \Omega)
  U^{\ast} M V (I_q + \Lambda) (I_q + \Psi) - \Sigma - S$. We said that \ $H$
  is a $p$-map if \ there exists quantities $a \geqslant 1$, $b \geqslant 0$,
  $\tau$, $\zeta_1$, $\zeta_2$, $\alpha_1$, $\alpha_2$, $\alpha_0$, $\alpha$,
  $\varepsilon$ be such that for all \ $(U, V, \Sigma)$ satisfying \ $\max
  \left( \kappa^a K^b {\| \Delta }  \|, \kappa^a K^{b + 1} \| E_{\ell} (U )
  \|,  \kappa^a K^{b + 1} \| E_q (V ) \|  \right) \leqslant \varepsilon$ \ we
  have :
  \begin{align}
&    | | E_{\ell} (U  (I_{\ell} + \Omega)) | | \leqslant \| E_{\ell} (U) \|^{p
    + 1} \infixand \| E_q (V (I_q + \Lambda)) \| \leqslant \| E_q (V) \|^{p +
    1} \label{gen-H1}\\
    &    \nonumber\\
    & \kappa^a K^b  \| \Delta_1 \| \leqslant \tau \| \Delta
    \|^{p + 1} \infixand \kappa^a K^b  \| S  \| \leqslant \alpha \| \Delta \| 
      \label{gen-H2}\\
     &  \nonumber\\
   & \begin{array}{l}
      \| I_{\ell} + \Theta \|^2, \quad \| I_q + \Psi \|^2 \leqslant \zeta_1\\
    \\
      \| (I_{\ell} + \Theta^{\ast}) (I_{\ell} + \Theta) - I_{\ell} \|, \quad
      \| (I_q + \Psi^{\ast}) (I_q + \Psi) - I_q \| \leqslant \dis\frac{1}{\kappa^a
      K^{b + 1}} \zeta_2  \varepsilon^{p + \delta}
    \end{array}  \label{gen-H3}\\
      &  \nonumber\\
&    \| \Omega  \|, \| \Lambda  \| \leqslant \alpha_1 \| \Delta \| \infixand \|
    \Theta  \|, \| \Psi  \| \leqslant \alpha_2 \alpha_0 \varepsilon .
    \label{gen-OiLiTiPi}
  \end{align}
\end{definition}

We are proving that the theorems cited in the introduction result from the
following

satement.

\begin{theorem}
  \label{general-result} Let an integer $p \geqslant 1$ and three reals $a
  \geqslant 1$, $b, \varepsilon \geqslant 0$. Let $\delta = 1$ if $p$ is odd
  and $\delta = 2$ if $p$ is even. Let us consider a $p$-map $H$ as in $\left(
  \ref{map-H} \right)$. Let us consider a triplet $(U_0, V_0, \Sigma_0)$ \ and
  define the sequence for $i \geqslant 0$, $(U_{i + 1}, V_{i + 1}, \Sigma_{i +
  1}) = H (U_i, V_i, \Sigma_i)$. Let $\Delta_i = U_i^{\ast} M V_i - \Sigma$,
  $K_i \assign K (\Sigma_i)$ and $\kappa_i = \kappa (\Sigma_i)$ with $K = K_0$
  and $\kappa = \kappa_0$. \ Let us suppose
  \begin{align}
&    \max \left( \kappa^a K^b {\| \Delta_0}  \|, \kappa^a K^{b + 1} \| E_{\ell}
    (U_0) \|,  \kappa^a K^{b + 1} \| E_q (V_0) \|  \right)  \leqslant
    \varepsilon  \label{h-eps}\\
&    \frac{(1 + \alpha \varepsilon)^b}{(1 - 2 \alpha \varepsilon)^a} (2
    \varepsilon)^p \tau  \leqslant 1.  \label{tau-zetai-1}\\
 &   (2 \varepsilon)^p  \frac{(1 + \alpha \varepsilon)^{b + 1}}{(1 - 2 \alpha
    \varepsilon)^a }  (\zeta_1 + \zeta_2 \varepsilon^{\delta - 1} )  
    \leqslant & 1.  \label{tau-zetai-2}\\
&    1 - 8 \alpha \varepsilon  >  0  \label{8ae}
  \end{align}
  where the quantities $\alpha$, $\tau$, $\zeta_1$ and $\zeta_2$ are as in
  Definition \ref{def-mapH}. Then the sequence $(U_i, V_i, \Sigma_i)_{i
  \geqslant 0}$ converge to an SVD of $M$ and we have
  \begin{align}
&    \max \left( \kappa_i^a K_i^b {\| \Delta_i}  \|_, \kappa_i^a K_i^{b + 1} \|
    E_{\ell} (U_i) \|, \kappa_i^a K_i^{b + 1} \| E_q (V_i) \|  \right)
    \leqslant \varepsilon_i  \leqslant 2^{- (p + 1)^i + 1} \varepsilon 
    \label{gen-cl1}\\
&    \| \Sigma_i - \Sigma_0 \|  \leqslant (2 - 2^{2 - (p + 1)^i}) 
    \frac{\alpha c}{\kappa} \varepsilon  \label{gen-cl2}
  \end{align}
  where $c (1 - 4 \alpha \varepsilon) = 1$. The inequality
  $(\nobracket$\ref{gen-cl2}$\nobracket)$ implies $K - 2 \alpha c \varepsilon
  \leqslant K_i \leqslant K + 2 \alpha c \varepsilon$ and \ $\dfrac{\kappa}{c}
  \leqslant \kappa_i \leqslant \dfrac{\kappa}{1 - 4 \alpha c \varepsilon}$.
  Morever if there exist positive constant $u_0$ such that \ $\varepsilon
  \leqslant u_0$ and $2 (\alpha_1 + \alpha_2 + \alpha_1 \alpha_2 u_0) u_0 <
  1$, then by denoting $\gamma = 2 (\alpha_1 + \alpha_2 + \alpha_1 \alpha_2
  u_0)$ and $\sigma = 0.82 \times \alpha$ we have
  \begin{align}
    \|U_i - U_{\infty} \| & \leqslant 2^{- (p + 1)^i + 1} \sqrt{m} 
    \frac{\gamma}{1 - \gamma u_0} \varepsilon  \label{gen-Ui}\\
    \|V_i - V_{\infty} \| & \leqslant 2^{- (p + 1)^i + 1} \sqrt{n} 
    \frac{\gamma}{1 - \gamma u_0} \varepsilon  \label{gen-Vi}\\
    \| \Sigma_i - \Sigma_{\infty} \| & \leqslant 2^{- (p + 1)^i + 1} \sigma
    \varepsilon  \label{gen-Sigmai}
  \end{align}
\end{theorem}

\begin{proof}
  Let us denote for each~$i \geqslant 0$, $U_{i, 1} = U_i (I_{\ell} +
  \Omega_i)$ and $U_{i + 1} = U_{i, 1} (I_{\ell} + \Theta_i)$ with similar
  notations for $V_{i, 1}$ and $V_{i + 1}$. Let $\Delta_i + \Sigma_i =
  U_i^{\ast} M V_i$, \ $\Sigma_{i + 1} = \Sigma_i + S_i$ and also
  \[ \begin{array}{ccccccl}
       \varepsilon_0 & =& \varepsilon &  & \varepsilon_i & =& \max
       (\kappa_i^a K_i^b \| \Delta_i \|, \kappa_i^a K_i^{b + 1} \| E_{\ell}
       (U_i) \|, \kappa_i^a K_i^{b + 1} \| E_q (V_i) \| )\\
       \kappa_0 & =& \kappa & \qquad & \kappa_i & = &\kappa (\Sigma_i)\\
       K_0 & =& K &  & K_i & = &K (\Sigma_i)
     \end{array} \]
  We proceed by induction to prove $\left( \ref{gen-cl1} \right.$-$\left.
  \ref{gen-cl2} \right)$. The property evidently hold for $i = 0$. By assuming
  this for a given $i$, let us prove it for $i + 1$. We first prove that $\|
  \Sigma_{i + 1} - \Sigma_0 \| \leqslant (2 - 2^{2 - (p + 1)^{i + 1}})
  \dfrac{\alpha c}{\kappa} \varepsilon$ under the assumption $\| \Sigma_i -
  \Sigma_0 \| \leqslant (2 - 2^{2 - (p + 1)^i} ) \cfrac{\alpha c}{\kappa}
  \varepsilon$ with $c = 1 + 4 \alpha c \varepsilon$. From Lemma
  \ref{point-estimates-kappai-Ki} we have $K - 2 \alpha c \varepsilon
  \leqslant K_i \leqslant K + 2 \alpha c \varepsilon$ and $\dfrac{\kappa}{c}
  \leqslant \kappa_i \leqslant \dfrac{\kappa}{1 - 4 \alpha c \varepsilon} =
  \dfrac{1 - 4 \alpha \varepsilon}{1 - 8 \alpha \varepsilon} \kappa$. Using
  these bounds and assumption $(\nobracket$\ref{gen-H2}$\nobracket)$it follows
  that
  \begin{align}
    \| \Sigma_{i + 1} - \Sigma_i \| = \| S_i \| & \leqslant
    \frac{1}{\kappa_i^a K_i^b} \alpha \varepsilon_i \nonumber\\
    & \leqslant  \frac{c}{\kappa} 2^{- (p + 1)^i + 1} \alpha \varepsilon
    \quad \tmop{since} a \geqslant 1 \quad K \geqslant 1 \infixand \kappa_i
    \geqslant \frac{\kappa}{c} .  \label{deltai-bound-gen}
  \end{align}
  By applying the bound $\left( \ref{deltai-bound-gen} \right)$ we get
  \begin{align*}
    \| \Sigma_{i + 1} - \Sigma_0 \| & \leqslant \|S_i \| + \| \Sigma_i -
    \Sigma_0 \|\\
    & \leqslant 2^{1 - (p + 1)^i} \frac{1}{\kappa} \alpha c \varepsilon +
    (2 - 2^{2 - (p + 1)^i})  \frac{1}{\kappa} \alpha c \varepsilon\\
    & \leqslant (2 - 2^{1 - (p + 1)^i} (2 - 1) ) \frac{\alpha c}{\kappa}
    \varepsilon\\
    & \leqslant (2 - 2^{- (p + 1)^i} ) \frac{\alpha c}{\kappa} \varepsilon
    .
  \end{align*}
  But it is easy to see that \ $p \geqslant 1$ implies \ \ \ $2^{1 - (p +
  1)^i} \geqslant 2^{2 - (p + 1)^{i + 1}}$. Hence
  \begin{align*}
    \| \Sigma_{i + 1} - \Sigma_0 \| & \leqslant (2 - 2^{2 - (p + 1)^{i +
    1}})  \frac{\alpha c}{\kappa}  \hspace{0.17em} \varepsilon .
  \end{align*}
  Then inequality $(\nobracket$\ref{gen-cl2}$\nobracket)$ holds for all $i$.
  From \ $(\nobracket$\ref{gen-H2}$\nobracket)$ we have $\| \Sigma_{i + 1} -
  \Sigma_i \| = \| S_i \| \leqslant \dfrac{\alpha}{\kappa_i} \varepsilon_i$.
  We then deduce
  \begin{align}
    & K_i - \frac{\alpha}{\kappa_i} \varepsilon_i \leqslant K_{i + 1}
    \leqslant  \| \Sigma_i \| + \| \Sigma_{i + 1} - \Sigma_i \| \leqslant K_i
    + \dfrac{\alpha}{\kappa_i} \varepsilon_i .  \label{Ki1Ki}
  \end{align}
  As in the proof of Lemma \ \ref{point-estimates-kappai-Ki} we can obtain
  \begin{align}
&    \frac{\kappa_i}{1 + 2 \alpha \varepsilon} \leqslant  \kappa_{i + 1}
    \leqslant  \frac{\kappa_i}{1 - 2 \alpha \varepsilon}  \label{bnd-ki1}
  \end{align}
  \ We now prove that $\kappa_{i + 1}^a K^b_{i + 1} \| \Delta_{i + 1} \|
  \leqslant 2^{- 2^{i + 1} + 1} \varepsilon$. Using both the assumption
  $\left( \ref{gen-H2} \right)$ and $\left( \ref{Ki1Ki} \right.$-$\left.
  \ref{bnd-ki1} \right)$ it follows
  \begin{align*}
    \kappa_{i + 1}^a K^b_{i + 1}  \| \Delta_{i + 1} \| & \leqslant  \frac{(1
    + \alpha \varepsilon)^b}{(1 - 2 \alpha \varepsilon)^a} \kappa_i^a K_i^b
    \tau \| \Delta_i \|^{p + 1} \quad\\
    & \leqslant \frac{(1 + \alpha \varepsilon)^b}{(1 - 2 \alpha
    \varepsilon)^a} \tau \varepsilon_i^{p + 1}\\
    & \leqslant \frac{(1 + \alpha \varepsilon)^b}{(1 - 2 \alpha
    \varepsilon)^a} (2 \varepsilon)^p \tau 2^{- (p + 1)^{i + 1} + 1}
    \varepsilon\\
    & \leqslant 2^{- (p + 1)^{i + 1} + 1} \varepsilon \quad \tmop{since}
    \quad \frac{(1 + \alpha \varepsilon)^b}{(1 - 2 \alpha \varepsilon)^a} (2
    \varepsilon)^p \tau \leqslant 1 \quad \tmop{from} \quad \left(
    \ref{tau-zetai-1} \right) .
  \end{align*}
  We now can bound $ \| E_{\ell} (U_{i + 1}) \|$. We have
  \begin{align}
    \| E_{\ell} (U_{i + 1}) \| & \leqslant \| \nobracket (I_{\ell} +
    \Theta^{\ast}_i) {U_{i, 1}} ^{\ast} {U_{i, 1}}  (I_{\ell} + \Theta_i) \|
    \nobracket \nonumber\\
    & \leqslant \| (I_{\ell} + \Theta^{\ast}_i) E_{\ell} (U_{i, 1})
    (I_{\ell} + \Theta_i) + (I_{\ell} + \Theta^{\ast}_i) (I_{\ell} + \Theta_i)
    - I_{\ell} \| \nonumber\\
    & \leqslant (1 + \| \Theta_i \|)^2 \| E_{\ell} (U_{i, 1}) \| + \|
    (I_{\ell} + \Theta^{\ast}_i) (I_{\ell} + \Theta_i) - I_{\ell} \| . 
    \label{Ki1EmUi1}
  \end{align}
  From assumption $\left( \ref{gen-H1} \right) $we know $\| E_{\ell} (U_{i,
  1}) \| \leqslant \| E_{\ell} (U_i) \|^{p + 1} \leqslant \dfrac{1}{\kappa_i^a
  K_i^{b + 1}} \varepsilon_i^{p + 1}$. It follows using both assumption
  $(\nobracket$\ref{gen-H3}$\nobracket)$, $\left( \ref{kappa-bound-inf}
  \right.$-$\left. \ref{Ki1Ki} \right)$ \ that
  \begin{align*}
  \kappa_{i + 1}^a K^{b + 1}_{i + 1} \| E_{\ell} (U_{i + 1}) \| & \leqslant
     \frac{(1 + \alpha \varepsilon)^{b + 1}}{(1 - 2 \alpha \varepsilon)^a} 
    (\zeta_1 \varepsilon_i^{p + 1} + \zeta_2 \varepsilon_i^{p + \delta})\\
    & \leqslant \frac{(1 + \alpha \varepsilon)^{b + 1}}{(1 - 2 \alpha
    \varepsilon)^a }  (2 \varepsilon)^p (\zeta_1 + \zeta_2 \varepsilon^{\delta
    - 1}) 2^{- (p + 1)^{i + 1} + 1} \varepsilon\\
    & \leqslant 2^{- (p + 1)^{i + 1} + 1} \varepsilon \qquad\\
      & \tmop{since} \quad \frac{(1 + \alpha \varepsilon)^{b + 1}}{(1 - 2
    \alpha \varepsilon)^a } (2 \varepsilon)^p (\zeta_1 + \zeta_2
    \varepsilon^{\delta - 1} ) \leqslant 1 \quad \tmop{from} \left(
    \ref{tau-zetai-2} \right) .
  \end{align*}
  Hence $\kappa_{i + 1}^a K^{b + 1}_{i + 1} \nobracket \| E_{\ell} (U_{i + 1})
  \nobracket \| \leqslant 2^{- (p + 1)^{i + 1} + 1} \varepsilon$. In the same
  way $\kappa_{i + 1}^a K^{b + 1}_{i + 1} \nobracket \| E_q (V_{i + 1})
  \nobracket \|$ $\le 2^{- 2^{i + 1} + 1} \varepsilon$. Hence we have shown that
  $\varepsilon_{i + 1} \leqslant 2^{- 2^{i + 1} + 1} \varepsilon$. This
  completes the proof of (\ref{gen-cl1}--\ref{gen-cl2}).
  
  By applying Lemma \ref{conv-Ui} we conclude that the sequences $(U_i)_{i
  \geqslant 0}$ and $(V_i)_{i \geqslant 0}$ converges respectively towards
  $U_{\infty}$ and $V_{\infty}$ which are two unitary matrices since
  $\nobracket \| E_{\ell} (U_i) \nobracket \|, \| E_q (V_i) \nobracket
  \leqslant 2^{- 2^i + 1} \varepsilon$. Hence the bounds $\left( \ref{gen-Ui}
  \right.$-$\left. \ref{gen-Vi} \right)$ hold. Finally the bound $\left(
  \ref{gen-Sigmai} \right)$ follows from
  \begin{align*}
    \| \Sigma_{i + j} - \Sigma_i \| & \leqslant \sum_{k = i}^{i + j - 1} \|
    \Sigma_{k + 1} - \Sigma_k \|\\
    & \leqslant \sum_{k \geqslant i} 2^{- (p + 1)^k + 1} \alpha
    \varepsilon\\
    & \leqslant \left( \sum_{k \geqslant 0} 2^{- (p + 1)^k}  \right) 2^{-
    (p + 1)^i + 1} \alpha \varepsilon\\
    & \leqslant 2^{- (p + 1)^i + 1} \times 0.82 \alpha \varepsilon \quad
    \tmop{since} \quad \sum_{k \geqslant 0} 2^{- (p + 1)^k} \leqslant \sum_{k
    \geqslant 3} 2^{- 2^k} \leqslant 0.82.
  \end{align*}
  Hence the sequence $(\Sigma_i)_{i \geqslant 0}$ admits a limit
  $\Sigma_{\infty}$. The triplet $(U_{\infty}, V_{\infty}, \Sigma_{\infty})$
  is a solution of SVD system $\left( \ref{syst-svd} \right)$. The theorem is
  proved.
\end{proof}

\begin{lemma}
  \label{point-estimates-kappai-Ki} Using the notations and asumptions of the
  proof of Theorem \ref{general-result} we have with $c = 1 + 4 \alpha c
  \varepsilon$ :
  \begin{align*}
    K - 2 \alpha c \varepsilon \leqslant K_i & \leqslant K + 2 \alpha c
    \varepsilon\\
    \dfrac{\kappa}{c} \leqslant \kappa_i & \leqslant \dfrac{\kappa}{1 - 4
    \alpha c \varepsilon}
  \end{align*}
\end{lemma}

\begin{proof}
  Let us \ prove that $K_i \leqslant K + 2 \alpha \varepsilon$. We have
  \begin{align*}
    K_i \assign \| \Sigma_i \| & \leqslant \| \Sigma_0 \| + \| \Sigma_i -
    \Sigma_0 \|\\
    & \leqslant K + (2 - 2^{- (p + 1)^i + 1})  \frac{\alpha c}{\kappa}
    \varepsilon\\
    & \leqslant K + 2 \alpha c \varepsilon \quad \tmop{since} \quad \kappa
    \geqslant 1.
  \end{align*}
  In the same way $K_i \geqslant K - 2 \alpha c \varepsilon$. We have also
  $\kappa_i \leqslant \dfrac{\kappa}{1 - 4 \alpha c \varepsilon}$. In fact, if
  \ $\sigma_{i, j}$'s be the diagonal values of $\Sigma_i^{\nosymbol}$, the
  Weyl's bound {\cite{Weyl1912}} implies that
  \begin{align}
    & | \sigma_{i, j} - \sigma_{0, j} | \leqslant  \| \Sigma_i - \Sigma_0 \|
    \leqslant 2 \frac{\alpha c}{\kappa} \varepsilon \hspace{3em} 1 \leqslant j
    \leqslant n,  \label{sij-s0j}
  \end{align}
  and
  \begin{align*}
&    K - 2 \frac{\alpha c}{\kappa} \varepsilon  \leqslant \sigma_{i, j}
    \leqslant K + 2 \frac{\alpha c}{\kappa} \varepsilon \qquad 1 \leqslant j
    \leqslant n.
  \end{align*}
  Hence, since $\kappa, K \geqslant 1$ we get
  \begin{align}
&    \frac{\kappa}{1 + 2 \alpha c \varepsilon} \leqslant  \sigma_{i, j}^{- 1}
    \leqslant  \frac{\kappa}{1 - 2 \alpha c \varepsilon}  \label{sigmaij-1}
  \end{align}
  Moreover for $1 \leqslant j < k \leqslant n$, we have :
  \begin{align}
    \begin{array}{lll}
      | \nobracket \sigma_{i, k} \pm \sigma_{i, j} | \nobracket & \geqslant 
      | \sigma_{0, k} \pm \sigma_{0, j} | - | \sigma_{i, k} - \sigma_{0, k} |
      - | \sigma_{i, j} - \sigma_{0, j} |\\
      & \geqslant  \dis| \sigma_{0, k} \pm \sigma_{0, j} | \left( 1 -
      \frac{1}{\kappa | \sigma_{0, k} \pm \sigma_{0, j} |} 4 \alpha c
      \varepsilon \right) \quad \tmop{from} \quad \left( \ref{sij-s0j}
      \right)\\
      & \geqslant \dis | \sigma_{0, k} \pm \sigma_{0, j} |  (1 - 4 \alpha c
      \varepsilon) = | \sigma_{0, k} \pm \sigma_{0, j} | \frac{1 - 8 \alpha
      \varepsilon}{1 - 4 \alpha \varepsilon} > 0 \quad\\
      &  \qquad \tmop{since} \kappa | \sigma_{0, k} \pm \sigma_{0,
      j} | \geqslant 1 \infixand \left( \ref{8ae} \right)
    \end{array} &  &  \label{bnd-sigmajk}
  \end{align}
  Taking in account the definition of $\kappa$ and the inequalities $\left(
  \ref{sigmaij-1} \right), \left( \ref{bnd-sigmajk} \right)$, we then get
  \begin{align*}
    \kappa_i& = \max \left( 1, \max_j \frac{1}{\sigma_{i, j}}, \hspace{0.2em}
    \max_{k \neq j} \left( \dfrac{1}{| \nobracket \sigma_{i, k} - \sigma_{i,
    j} | \nobracket} + \dfrac{1}{| \nobracket \sigma_{i, k} + \sigma_{i, j} |
    \nobracket} \right) \right) \\
    & \leqslant \kappa \max \left( \frac{1}{1 -
    2 \alpha c \varepsilon}, \frac{1}{1 - 4 \alpha c \varepsilon} \right)\\
    & \leqslant \dfrac{\kappa}{1 - 4 \alpha c \varepsilon} = \frac{1 - 4
    \alpha \varepsilon}{1 - 8 \alpha \varepsilon} .
  \end{align*}
  \ In the same way we have
  \begin{align*}
    | \nobracket \sigma_{i, k} \pm \sigma_{i, j} | \nobracket & \leqslant |
    \sigma_{0, k} \pm \sigma_{0, j} | + | \sigma_{i, k} - \sigma_{0, k} | + |
    \sigma_{i, j} - \sigma_{0, j} |\\
    & \leqslant | \sigma_{0, k} \pm \sigma_{0, j} |  (1 + 4 \alpha c
    \varepsilon) = | \sigma_{0, k} \pm \sigma_{0, j} | c.
  \end{align*}
  We deduce that
  \begin{align}
    \kappa_i & \geqslant  \frac{\kappa}{c} = (1 - 4 \alpha \varepsilon)
    \kappa .  \label{kappa-bound-inf}
  \end{align}
  The Lemma is proved.
\end{proof}
\section{Proof of Theorem \ref{th-svd-main} : case $p = 1$}\label{proof-p=1}
\qquad\\
\textcolor{red}{To help the verification of certain parts of the proof, it is possible to download a Maple file at}
\\
\textcolor{blue}{\url{https://perso.math.univ-toulouse.fr/yak/files/2023/09/high-order-methods-for-svd-aided-proofs-with-Maple-09-23.mw}}
\\
\textcolor{red}{After clicking on the link above, a Maple code is displayed in your internet browser.
Then choose the option to display the source code of the page.
Next copy and paste the text into an editor. Finally save for example under {\em svd-armentano-yak.mw}.
Open Maple to run this file.}
\\
Let
\begin{align*}
 s = \left( 1 + \frac{1}{2} \varepsilon  \right)^2 + 1 + \dfrac{1}{4}
  \varepsilon  , &\quad \tau = (3 + s \varepsilon) s^2, &\quad a = 2, \quad b = 1,
  \quad u_0 = 0.0289.
\end{align*}
It consists to verify the assumptions of Theorem \ref{general-result}.
Remember that $\left( \ref{h-eps} \right)$ is satisfied from assumption since
\begin{align*}
&  \max \left( {\kappa }^a {K^{b + 1}}   \| E_{\ell} (U) \|, \kappa^a   {K^{b +
  1}}   \left\| E_q (V) | | {, \kappa^a}  K^b   \right\| \Delta | |  \right) 
  \leqslant  \varepsilon
\end{align*}
where $U$, $V$, $\Delta$ stand for $U_0$, $V_0$, $\Delta_0$ respectively. The
item $\left( \ref{gen-H1} \right)$ follows of Proposition \ref{unit-proj-prop}
since $\Omega  = - \dfrac{1}{2} E_{\ell} (U)$ and $\Lambda  = - \dfrac{1}{2}
E_q (V)$. Let us prove the item $\left( \ref{gen-H2} \right)$. To do that we
denote $\Delta_{0, 1} = (I_{\ell} + \Omega ) (\Delta  + \Sigma ) (I_q +
\Lambda ) - \Sigma $ and $\varepsilon_{0, 1} = \| \Delta_{0, 1} \|$. From
Proposition \ref{unit-proj-prop} and \ $| | E_{\ell} (U ) | |, \| E_q (V ) \|
\leqslant \dfrac{\varepsilon }{{{\kappa^a} }  {K^{b + 1} } }$ we know that $\|
\Omega  \|, \| \Lambda  \| \leqslant \cfrac{1}{{{2 \kappa^a} }  {K^{b + 1}}  }
\varepsilon $ . We then apply Proposition \ref{part1} with $w = \dfrac{1}{2}$
to get
\begin{align}
  \varepsilon_{0, 1} & \leqslant \left( \left( 1 + \frac{1}{2} \varepsilon 
  \right)^2 + 1 + \dfrac{1}{4} \varepsilon  \right)
  \frac{\varepsilon}{{{\kappa^a} }  K^b} \nonumber\\
  & \leqslant  \frac{s \varepsilon}{\kappa^a K^b}  .  \label{sei-proofp=1}
\end{align}
From Lemma \ref{fund-pert-prop} we have $\| X \|, \| Y \| \leqslant \kappa 
\varepsilon_{0, 1}$. We deduce that the quantity
\begin{align*}
  \Delta_1 & = (I_{\ell} - X ) (\Delta_{0, 1} + \Sigma ) (I_q + Y ) - \Sigma
  - S \\
  & = - X  \Delta_{0, 1} + \Delta_{0, 1} Y  - X  \Delta_{0, 1} Y  - X 
  \Sigma  Y  \quad \tmop{since} \quad \Delta_{0, 1} - S  - X \Sigma  +
  \Sigma  Y  = 0,
\end{align*}
can be bounded by
\begin{align*}
  \| \Delta_1 \| & \leqslant 2 \kappa  \varepsilon_{0, 1}^2 + \kappa^2 
  \varepsilon_{0, 1}^3 + \kappa ^2 K  \varepsilon_{0, 1}^2\\
  & \leqslant \left( \frac{2}{\kappa^3 K^2} + \frac{s \varepsilon}{\kappa^4
  K^3} + \frac{1}{\kappa^2 K} \right) s^2 \varepsilon^2\hspace{0.5em}
   \tmop{since}
  \hspace{0.5em}\kappa, K \geqslant 1 \infixand \varepsilon_{0, 1} \leqslant \frac{s
  \varepsilon}{{\kappa }^2 K } \tmop{from} \left( \ref{sei-proofp=1} \right)
  .\\
  & \leqslant \frac{1}{\kappa^2 K} (3 + s \varepsilon) s^2 \varepsilon^2 =
  \frac{1}{\kappa^2 K} \tau \varepsilon^2 \quad
\end{align*}
On the other hand \ $S = \tmop{diag} (\Delta_{0, 1})$. It follows $\| S \|
\leqslant \varepsilon_{0, 1} \leqslant \dfrac{s \varepsilon}{{\kappa^2}  K }$.
The quantity $\alpha$ of Definition \ref{def-mapH} is equal to $s$. This
allows to prove the assumption $\left( \ref{tau-zetai-1} \right)$ that is
\begin{align*}
  2 \varepsilon  \frac{1 + s \varepsilon}{(1 - 2 s \varepsilon)^2} \tau &
  \leqslant  2 \frac{1 + s \varepsilon}{(1 - 2 s \varepsilon)^2}  (3 + s
  \varepsilon) s^2 \varepsilon\\
  & \leqslant 1 \quad \tmop{since} \quad \varepsilon \leqslant u_0 =
  0.0289.
\end{align*}
We now prove the item $(\nobracket$\ref{gen-H3}$\nobracket)$. We have
\begin{align*}
  \| I_{\ell} + \Theta  \|^2 & \leqslant (1 + \| X \|)^2\\
  \| (I_{\ell} - X) (I_{\ell} + X ) - I_{\ell} \| & = \| X \|^2 .
\end{align*}
Using Lemma \ref{epCp1} we know that $\| X \| \leqslant \kappa 
\varepsilon_{0, 1} \leqslant \dfrac{s \varepsilon}{ \kappa^{a - 1} K^b}$. We
deduce that
\begin{align*}
  (1 + \| X \| )^2 & \leqslant (1 + s \varepsilon)^2 = \zeta_1\\
  \| (I_{\ell} - X) (I_{\ell} + X ) - I_{\ell} \| & \leqslant \frac{\zeta_2
  \varepsilon^2}{\kappa^{2 a - 2} K^{2 b}} \quad \tmop{where} \zeta_2 = s^2
  .\\
  & \leqslant \frac{1}{\kappa^a K^{b + 1}} \zeta_2 \varepsilon^2 \quad
  \tmop{since} a = 2 \infixand b = 1.
\end{align*}
This allows to prove the assumption $\left( \ref{tau-zetai-2} \right)$ that is
\begin{align*}
  (2 \varepsilon)  \frac{(1 + s \varepsilon)^2}{(1 - 2 s \varepsilon)^2}&
  (\zeta_1 + \zeta_2 \varepsilon^{\delta - 1})
  \\ & \leqslant 2 \frac{(1 + s
  \varepsilon)^2}{(1 - 2 s \varepsilon)^2}  ((1 + s \varepsilon)^2 + s^2)
  \varepsilon \hspace{0.5em} \tmop{since} 
  \hspace{0.5em} p = 1 \tmop{implies}\hspace{0.5em} \delta = 1\\
  & \leqslant 0.443 \quad \leqslant 1 \quad \tmop{since} \quad u \leqslant
  u_0 .
\end{align*}
Finally $1 - 8 s \varepsilon \geqslant 0.46 > 0$. This proves the item
$\left( \ref{8ae} \right)$.

We now verify the assumption $(\nobracket$\ref{gen-OiLiTiPi}$\nobracket)$. We
have seen that $\| \Omega \|, \| \Lambda \|  \leqslant \dfrac{1}{2}
\varepsilon$. Hence $\alpha_1 = \dfrac{1}{2}$. On the other hand one has
$\Theta = X$ and $\Psi_i = Y$. From $\| X \|, \| Y \| \leqslant s \varepsilon
\leqslant 2.04 2 \varepsilon$ since $u \leqslant u_0$, we can take  $\alpha_2
= 2.042$. Since $\gamma u_0 = 2 (\alpha_1 + \alpha_2 + \alpha_1 \alpha_2 u_0)
u_0 < 0.15$ then the bounds $\left( \ref{gen-Ui} \right.$-$\left.
\ref{gen-Sigmai} \right)$ of Theorem \ref{general-result} hold with
\begin{align*}
  \gamma & = 5.14\\
  \dfrac{\gamma}{1 - \gamma u_0} & \leqslant 6.1\\
  \sigma = 0.82 s & \leqslant 1.67.
\end{align*}
The Theorem \ref{th-svd-main} is proved in the case $p = 1$.$\Box$

\begin{proposition}
  \label{part1}Let $\varepsilon \geqslant 0$ and $a, b > 0$. Let $\Delta_1 =
  (I_{\ell} + \Omega) (\Delta + \Sigma) (I_q + \Lambda) - \Sigma$ with
  $\Omega^{\ast} = \Omega$. Let us suppose $\| \Delta \| \leqslant
  \dfrac{\varepsilon}{\kappa^a K^b}$ and $\| \Omega \|, \| \Lambda \|
  \leqslant \dfrac{w \varepsilon}{\kappa^a K^{b + 1}}$ with $\kappa = \kappa
  (\Sigma)$ and $K = K (\Sigma)$. We have
  \[ \| \Delta_1 \| \leqslant \left( {(1 + w \varepsilon)^2}  + 2 w + w^2
     \varepsilon \right)  \frac{\varepsilon}{\kappa^a K^b} . \]
\end{proposition}

\begin{proof}
  We have $\Omega^{\ast } = \Omega$. A straightforward calculation shows that
  
  \begin{align*}
    \Delta_1 & = (I_{\ell} + \Omega) \Delta (I_q + \Lambda) + (I_{\ell} +
    \Omega) \Sigma (I_q + \Lambda) - \Sigma\\
    & = (I_{\ell} + \Omega) \Delta (I_q + \Lambda) + \Omega \Sigma + \Sigma
    \Lambda + \Omega \Sigma \Lambda .
  \end{align*}
  
  Bounding $\| \Delta_1 \|$ we get
  \begin{align*}
    \| \Delta_1 \| & \leqslant \left( 1 + \frac{w \varepsilon}{\kappa^a K^{b
    + 1}} \right)^2 \frac{\varepsilon}{\kappa^a K^b} + 2 \frac{w
    \varepsilon}{\kappa^a K^b} + \left( \frac{w \varepsilon}{\kappa^a K^{b +
    1}} \right)^2 K\\
    & \leqslant \left( {(1 + w \varepsilon)^2}  + 2 w + w^2 \varepsilon
    \right)  \frac{\varepsilon}{\kappa^a K^b} \quad \tmop{since} \quad \kappa,
    K \geqslant 1.
  \end{align*}
  The proposition is proved.
\end{proof}
\section{Proof of Theorem \ref{th-svd-main} : case $p = 2$}\label{proof-p=2}

Let us introduce some constants and quantities.
\begin{align}
&  \begin{array}{lll}
    w = \dfrac{1}{2} \left( 1 + \dfrac{3}{4} \varepsilon \right), & s = (1 + w
    \varepsilon )^2 + 2 w + w^2 \varepsilon, & \\
  \dis  a = \frac{4}{3}, \qquad \dis b = \frac{1}{3}, & u_0 = 0.046. & 
  \end{array}   \label{cstes-thDS-2}
\end{align}
We also introduce
\begin{align}
  \tau_1 & = 2 + 2 \varepsilon  + \frac{5}{4} \varepsilon^2 + \frac{1}{4}
  \varepsilon^3 \nonumber\\
  \tau_2 & = 3 + \frac{1}{2} (11 + 2 \tau_1) \varepsilon + \frac{1}{2} (8 +
  7 \tau_1)  \varepsilon^2 + \frac{1}{2} (2 + 7 \tau_1 + \tau_1^2)
  \varepsilon^3 \\
&  \quad+ \frac{1}{2} (3 + 2 \tau_1) \tau_1 \varepsilon^4 + \tau_1^2
  \varepsilon^5 + \frac{1}{4} \tau_1^3 \varepsilon^6 \nonumber\\
  \tau & = \tau_1 \tau_2  \label{taup=2}\\
  \alpha & = (1 + \tau_1 (s \varepsilon) s \varepsilon) s \nonumber
\end{align}
Let us verify the assumptions of Theorem \ref{general-result}. The item
$\left( \ref{gen-H1} \right)$ follows of Proposition \ref{unit-proj-prop}
since $\Omega  = s_2 (E_{\ell} (U ))$ and $\Lambda  = s_2 (E_q (V ))$. Let us
prove the item $\left( \ref{gen-H2} \right)$. We first bound $\| \Delta_1 \| $
where $\Delta_1 = U_1^{\ast} M V - \Sigma_1$. We use the $\Delta_{0, i}$, $1
\leqslant i \leqslant 3$, the quantities defined by the formulas
(\ref{eq-SkXkYk}-\ref{def-Deltai}). By definition of the map $H_2$, we have
$\Delta_1 = \Delta_{0, 3}$. We introduce the quantities $\varepsilon_{0, i} =
\left\| {\Delta_{0, i}}  \right\|$. From Proposition \ref{unit-proj-prop} in
the case $p = 2$ and assumption $| | E_{\ell} (U ) | |, \| E_q (V ) \|
\leqslant \dfrac{\varepsilon }{\kappa^a K^{b + 1} }$ we know that $\| \Omega 
\|, \| \Lambda  \| \leqslant \cfrac{w}{\kappa^a K^{b + 1} } \varepsilon $ with
$w = \dfrac{1}{2} \left( 1 + \dfrac{3}{4} \varepsilon \right)$. We then apply
Proposition \ref{part1} to get
\begin{align}
  \varepsilon_{0, 1} & \leqslant ((1 + w \varepsilon )^2 + 2 w + w^2
  \varepsilon )  \frac{\varepsilon}{\kappa^a K^b }  \nonumber\\
  & \leqslant  \frac{s \varepsilon}{\kappa^a K^b } \qquad \tmop{from} \quad
  \left( \ref{cstes-thDS-2} \right) .  \label{sei-DS}
\end{align}
From Proposition \ref{prop-Deltap1-bnd-p=2} we can write

\begin{align*}
  \| \Delta_1 \| & = \| \Delta_{0, 3} \| \leqslant \frac{1}{\kappa^{4 / 3}
  K^{1 / 3}} \tau (s \varepsilon) s^3 \varepsilon^3 .
\end{align*}

We now bound the norm of $S = S_1 + S_2$. We have always from Proposition
\ref{prop-Deltap1-bnd-p=2}
\begin{align}
  \| S \| \leqslant \| \Delta_{0, 1} \| + \| \Delta_{0, 2} \| & \leqslant
  \frac{1}{\kappa^{4 / 3} K^{1 / 3}} (1 + \tau_1  (s \varepsilon) s
  \varepsilon) s \varepsilon = \frac{1}{\kappa^{4 / 3} K^{1 / 3}} \alpha
  \varepsilon .  \label{SDS}
\end{align}
A numerical computation shows that the inequality $(2 \varepsilon)^2 
\dfrac{(1 + \alpha \varepsilon)^{1 / 3}}{(1 - 2 \alpha \varepsilon)^{4 / 3}}
\tau (s \varepsilon) s^3 \leqslant 1$ is verified for all $u \leqslant u_0$.
Then the \ assumption $\left( \ref{tau-zetai-1} \right)$ holds.

We now prove the item $(\nobracket$\ref{gen-H3}$\nobracket)$. We have
\begin{align*}
  \| I_{\ell} + \Theta  \|^2 & \leqslant (1 + | | c_2 (X )  | |)^2\\
  \| (I_{\ell} + \Theta ^{\ast}) (I_{\ell} + \Theta ) - I_{\ell} \| 
&  \leqslant  (1 + c_2 (- \| X  \|)) (1 + c_2 (\| X \|)) - 1
\end{align*}
From the bound $\left( \ref{SDS} \right)$ we deduce that $\| X \| \leqslant \|
X_1 \| + \| X_2 \| \leqslant \dfrac{\kappa x}{\kappa^{4 / 3} K^{1 / 3}} =
\dfrac{x}{\kappa^{1 / 3} K^{1 / 3}}$ with $x = \alpha \varepsilon$. On the
other hand $c_2 (u) = u + \dfrac{1}{2} u^2$ and $(1 + c_2 (- u)) (1 + c_2 (u))
- 1 = \dfrac{u^4}{4}$. \ It follows \ :
\begin{align*}
  \| I_{\ell} + \Theta  \|^2 & \leqslant \left( 1 + x + \frac{1}{2} x^2
  \right)^2 = \zeta_1\\
  \| (I_{\ell} + \Theta ^{\ast}) (I_{\ell} + \Theta ) - I_{\ell} \| &
  \leqslant \frac{1}{4 \kappa^{4 / 3} K^{4 / 3}} (\alpha \varepsilon)^4 =
  \frac{1}{\kappa^{4 / 3} K^{4 / 3}} \zeta_2 \varepsilon^4 \quad \tmop{where}
  \quad \zeta_2 = \frac{1}{4} \alpha^4 \varepsilon^4 .
\end{align*}
We now prove a part of assumption $\left( \ref{tau-zetai-2} \right)$ that is
$(2 \varepsilon)^2 \dfrac{(1 + \alpha \varepsilon)^{4 / 3}}{(1 - 2 \alpha
\varepsilon)^{4 / 3}} (\zeta_1 + \zeta_2 \varepsilon) \leqslant 1.$ We \ have
\begin{align*}
  (2 \varepsilon)^2 \dfrac{(1 + \alpha \varepsilon)^{4 / 3}}{(1 - 2 \alpha
  \varepsilon)^{4 / 3}} (\zeta_1 + \zeta_2 \varepsilon) & \leqslant 0.025
  \quad \quad \tmop{since} \quad u \leqslant u_0 .
\end{align*}
This proves the item $\left( \ref{tau-zetai-2} \right)$. The item \ref{8ae}
holds since $1 - 8 \alpha \varepsilon \geqslant 0.05 > 0$ when $\varepsilon
\leqslant u_0$.

Let us prove the assumption $\left( \ref{gen-OiLiTiPi} \right)$. Using
$\varepsilon \leqslant u_0$ we have $\| \Omega \|, \| \Lambda \| \leqslant w
\varepsilon \leqslant \alpha_1 \varepsilon$ with $\alpha_1 = 0.52$ and $\|
\Theta \|, \| \Psi \| \leqslant (1 + x / 2) \alpha \varepsilon \leqslant
\alpha_2 \varepsilon$ with $\alpha_2 = 2.7$ Moreover
\begin{align*}
  2 (\alpha_1 + \alpha_2 + \alpha_1 \alpha_2 u_0) u_0 & \leqslant 0.304 < 1
\end{align*}
Then the bounds $\left( \ref{gen-Ui} \right.$-$\left. \ref{gen-Sigmai}
\right)$ of Theorem \ref{general-result} hold with
\begin{align*}
  \gamma & = 6.56\\
  \dfrac{\gamma}{1 - \gamma u_0} & \leqslant 9.41\\
  \sigma = 0.82 \alpha & \leqslant 2.1.
\end{align*}
The Theorem \ref{th-svd-main} is proved for \ $p = 2$. $\Box$

\begin{proposition}
  \label{prop-Deltap1-bnd-p=2} Let $p = 2$, \ $\varepsilon \geqslant 0$. Let
  us consider $\Delta_1 = U_1^{\ast} M V_1 - \Sigma$ such that $\| \Delta_1 \|
  = \varepsilon_1 \leqslant \dfrac{\varepsilon}{\kappa^{4 / 3} K^{1 / 3}}$
  where $\kappa = \kappa (\Sigma)$ and $K = K (\Sigma)$. Let us consider
  $\tau_1 \assign \tau_1 (\varepsilon)$ and $\tau \assign \tau (\varepsilon)$
  as in $\left( \ref{taup=2} \right)$ Then we have
  \begin{align*}
    \| \Delta_2 \| & \leqslant \frac{1}{\kappa^{4 / 3} K^{1 / 3}} \tau_1
    \varepsilon^2,\\
    \tau_3 \assign \| \Delta_3 \| & \leqslant \frac{1}{\kappa^{4 / 3} K^{1 /
    3}} \tau  \varepsilon^3,
  \end{align*}
  where $\Delta_2 = (I_{\ell} + \Theta_1^{\ast})  (\Delta_1 + \Sigma)  (I_q +
  \Psi_1) - \Sigma - S_1$ and $\Delta_3 = (I_{\ell} + \Theta_2^{\ast}) 
  (\Delta_1 + \Sigma)  (I_q + \Psi_2) - \Sigma - S_1 - S_2$ with $\Theta_2$
  and $\Psi_2$ are defined by the formulas~$(\nobracket$\ref{def-Deltai}) for
  $p = 2$.
\end{proposition}

\begin{proof}
  We denote $e_2 (X ) = X^2 / 2$, $\Theta_1 = X_1 + e_2 (X_1)$ and $\Psi_1 =
  Y_1 + e_2 (Y_1)$. Remember $\Delta_1 + \Sigma = U^{\ast} \Sigma V$ and
  $\Delta_2 = (I_{\ell} + \Theta_1^{\ast})  (\Delta_1 + \Sigma)  (I_q +
  \Psi_1) - \Sigma - S_1$. Expanding $\Delta_2$ we find
  \begin{align}
    \Delta_2 & = \Delta_1 - S_1 - X_1 \Sigma + \Sigma Y_1 - X_1 \Sigma Y_1 +
    \frac{1}{2} X_1^2 \Sigma + \Sigma \frac{1}{2} Y_1^2 + \frac{1}{4} X_1^2
    \Sigma Y_1^2 \nonumber\\
    &   \qquad + \frac{1}{2} X_1^2 \Sigma Y_1 - \frac{1}{2} X_1  \Sigma Y_1^2 - X_1
    \Delta_1 + \Delta_1 Y_1 - X_1 \Delta_1 Y_1 + \frac{1}{2} X_1^2 \Delta_1 +
    \frac{1}{2} \Delta_1 Y_1^2 \nonumber\\
    &\qquad   + \frac{1}{4} X_1^2 \Delta_1 Y_1^2 + \frac{1}{2} X_1^2 \Delta_1 Y_1
    - \frac{1}{2} X_1  \Delta_1 Y_1^2 \nonumber\\
    & = \frac{1}{2} (X_1 (- \Sigma Y_1 + X_1 \Sigma) + (- X_1 \Sigma +
    \Sigma Y_1) Y_1) + \frac{1}{4} X_1^2 \Sigma Y_1^2\nonumber\\
    &\qquad + \frac{1}{2} X_1 (X_1
    \Sigma - \Sigma Y_1) Y_1 
       - X_1 \Delta_1 + \Delta_1 Y_1 - X_1 \Delta_1 Y_1 + \frac{1}{2} X_1^2
    \Delta_1 + \frac{1}{2} \Delta_1 Y_1^2 \nonumber
    \\
    &\qquad e+ \frac{1}{4} X_1^2 \Delta_1 Y_1^2 +
    \frac{1}{2} X_1^2 \Delta_1 Y_1 + \frac{1}{2} X_1  \Delta_1 Y_1^2
    \nonumber\\
    & = \frac{1}{2} (X_1 (- \Delta_1 - S_1) + (S_1 + \Delta_1) Y_1) +
    \frac{1}{4} X_1^2 \Sigma Y_1^2 + \frac{1}{2} X_1 (- \Delta_1 - S 1) Y_1
     \label{Delta2-DS}\\
    & \qquad + \frac{1}{2} X_1^2 \Delta_1 + \frac{1}{2} \Delta_1 Y_1^2 +
    \frac{1}{4} X_1^2 \Delta_1 Y_1^2 + \frac{1}{2} X_1^2 \Delta_1 Y_1 -
    \frac{1}{2} X_1  \Delta_1 Y_1^2 \nonumber.
  \end{align}
  We know that $\| \Delta_1 \| \leqslant \varepsilon_1$. From the formula
  $(\nobracket$\ref{Delta2-DS}$\nobracket)$ we deduce
  \begin{align}
    \| \Delta_2 \| & \leqslant 2 \kappa \varepsilon_1^2 + \frac{1}{4}
    \kappa^4 K \varepsilon_1^4 + 2 \kappa^2 \varepsilon_1^3 + \frac{1}{4}
    \kappa^4 \varepsilon_1^5 + \kappa^3 \varepsilon_1^4 \nonumber\\
    & \leqslant q_1 \varepsilon_1^2 \quad \tmop{with} \quad q_1 = 2 \kappa
    + 2 \kappa^2 \varepsilon_1 + \frac{5}{4} \kappa^4 K \varepsilon_1^2 +
    \frac{1}{4} \kappa^4 \varepsilon_1^3  \label{bnd-delta2-q1}
  \end{align}
  Since $\varepsilon_1 \leqslant \dfrac{\varepsilon}{\kappa^{4 / 3} K^{1 /
  3}}$ it follows $q_1 \varepsilon_1 \leqslant \tau_1 \varepsilon$ with
  $\tau_1 = 2 + 2 \varepsilon  + \frac{5}{4} \varepsilon^2 + \frac{1}{4}
  \varepsilon^3$. Hence we have obtained \ $\| \Delta_2 \| \leqslant \tau_1
  \dfrac{\varepsilon^2}{\kappa^{4 / 3} K^{1 / 3}}$.
  
  From definition $\Theta_2 = c_2 (X_1 + X_2)$. Hence we can write $\Theta_2 
  = \Theta_1  + X_2 + A_2$ with
  \begin{align*}
    A_2 \assign A_2  (X_1, X_2) & = c_2 \left( X_1 {+ X_2}  \right) - c_2
    (X_1) - X_2\\
    & = \frac{1}{2} ((X_1 + X_2)^2 - X_1^2)\\
    & = \frac{1}{2} (X_2^2 + X_1 X_2 + X_2 X_1)
  \end{align*}
  In the same way $\Psi_2 = \Psi_1 + Y_2 + B_2$ where $B_2 = A_2 (Y_1, Y_2)$.
  Expanding $(I_{\ell} + \Theta_2^{\ast})  (\Delta_1 + \Sigma)  (I_q +
  \Psi_2)$ we get
  \begin{align*}
    \Delta_3 & = (I_{\ell} + \Theta_2^{\ast})  (\Delta_1 + \Sigma)  (I_q +
    \Psi_2) - \Sigma - S_1 - S_2\\
    & = (I_{\ell} + \Theta_1^{\ast} - X_2 + A_2)  (\Delta_1 + \Sigma)  (I_q
    + \Psi_1 + Y_2 + B_2) - \Sigma - S_1 - S_2\\
    & = (I_{\ell} + \Theta_1^{\ast})  (\Delta_1 + \Sigma)  (I_q + \Psi_1) -
    \Sigma - S_1 - S_2 + (I_{\ell} + \Theta_1^{\ast})  (\Delta_1 + \Sigma) 
    (Y_2 + B_2)\\
    &  \qquad + (- X_2 + A_2)  (\Delta_1 + \Sigma)  (I_q + \Psi_1) + (- X_2 + A_2)
    (\Delta_1 + \Sigma)  (Y_2 + B_2)
  \end{align*}
  We know that
  \[ (I_{\ell} + \Theta_1^{\ast})  (\Delta_1 + \Sigma)  (I_q + \Psi_1) -
     \Sigma - S_1 - S_2 = \Delta_2 - S_2 - X_2 \Sigma + \Sigma Y_2 = 0. \]
  Expanding more $\Delta_3$, we then can write by grouping the terms
  appropriately :
  \begin{align}
    \Delta_3 & = - X_2 \Delta_1 Y_2 + \Delta_1 B_2 + A_2 \Delta_1 - X_2
    \Delta_1 B_2 + A_2 \Delta_1 Y_2 + A_2 \Delta_1 B_2  \label{L1DS}\\
    &  \qquad + \Theta_1^{\ast} \Delta_1 Y_2 - X_2 \Delta_1 \Psi_1 +
    \Theta_1^{\ast} \Delta_1 B_2 + A_2 \Delta_1 \Psi_1  \label{L2DS}\\
    &  \qquad + G, \nonumber
  \end{align}
  where $G = - X_2 \Delta_1 + \Delta_1 Y_2 - X_2 \Sigma Y_2 + \Sigma B_2 + A_2
  \Sigma + \Theta_1^{\ast} \Sigma Y_2 - X_2 \Sigma \Psi_1 + \Theta_1^{\ast}
  \Sigma B_2 + A_2 \Sigma \Psi_1 - X_2 \Sigma B_2 + A_2 \Sigma Y_2 + A_2
  \Sigma B_2$. The Lemma~\ref{lem-Gp2} modifies the quantity as sum of the
  following $G_i$'s :
  \begin{align}
    G_1 & = \frac{1}{2} X_2 (\Delta_2 - S_2) + \frac{1}{2} (S_2 - \Delta_2)
    Y_2  \label{LG1DS}\\
    G_2 & = \dfrac{1}{2} (X_1 (\Delta_2 - S_2) + (S_2 - \Delta_2) Y_1) +
    \frac{1}{2} (X_2 (- \Delta_1 - S_1) + (S_1 + \Delta_1) Y_2) 
    \label{LG2DS}\\
    G_3 & = \frac{1}{2} \left( X_1 (\Delta_2 - S_2) Y_1 + X_2 (\Delta_1 -
    S_1) Y_2 + X_1 \left( {\Delta_2}  - S_2 \right) Y_2 \right) 
    \label{LG3DS}\\
    &  \qquad + \frac{1}{2} (X_2 (\Delta_1 - S_1) Y_1 + X_1 (\Delta_1 - S_1) Y_2 +
    X_2 (\Delta_2 - S_2) Y_1) \nonumber\\
    G_4 & = \frac{1}{2} X_2 (S_2 - \Delta_2) Y_2  \label{LG4DS}\\
    G_5 & = e_2 (X_1) \Sigma R_{2, 1} + Q_{2, 1} \Sigma e_2 (Y_1) + e_2
    (X_1) \Sigma e_2 (Y_2) + e_2 (X_2) \Sigma e_2 (Y_1)  \label{LG5DS}
  \end{align}
  where $Q_{2, 1} = \frac{1}{2} (X_1 X_2 + X_2 X_1)$ and $R_{2, 1} =
  \frac{1}{2} (Y_1 Y_2 + Y_2 Y_1)$. We are going to prove $\| \Delta_3 \|
  \leqslant q_1 q_2 \varepsilon_1^3$ where $q_2$ is defined below in $\left(
  \ref{tau3DS} \right)$. To do that we will use the bounds
  \begin{enumerate}
    \item $\| X_1 \|, \| Y_1 \| \leqslant \kappa \varepsilon_1,$ $\| \Delta_2
    \| \leqslant q_1 \varepsilon_1^2$ and
    \begin{align}
      \| X_2 \|, \| Y_2 \| & \leqslant \kappa q_1 \varepsilon_1^2 . 
      \label{X2Y2DS}
    \end{align}
    \item $\| \Theta_1 \|, \| \Psi_1 \| \leqslant \left( 1 + \dfrac{1}{2}
    \kappa \varepsilon_1 \right) \kappa \varepsilon_1$.
    
    \item $\| Q_{2, 1} \|, \| R_{2, 1} \| \leqslant q_1 \kappa^2
    \varepsilon_1^3$.
    
    \item $\| A_2 \|, \| B_2 \| \leqslant \dfrac{1}{2} (q_1^2 \kappa^2
    \varepsilon_1^4 + 2 q_1 \kappa^2 \varepsilon_1^3) = \dfrac{1}{2} (q_1
    \varepsilon_1 + 2) q_1 \kappa^2 \varepsilon_1^3 $.
  \end{enumerate}
  Considering the bounds of the norms of \ matrices given in
  $(\nobracket$\ref{L1DS}-\ref{LG5DS}$\nobracket)$, we get
  \begin{align*}
&    \frac{1}{q_1 \varepsilon_1^3} \| \Delta_3 \| \\
    & \leqslant \frac{1}{4}
    q_1^3 \kappa^4 \varepsilon_1^6 + q_1^2 \kappa^4 \varepsilon_1^5 + (\kappa
    + q_1) q_1 \kappa^3 \varepsilon_1^4 + 2 \kappa^3 q_1 \varepsilon_1^3 + 2
    \kappa^2 q_1 \varepsilon_1^2 + 2 \kappa^2 \varepsilon_1 \quad \tmop{from}
    \left( \ref{L1DS} \right)\\
    &\qquad   + \frac{1}{2} \kappa^4 q_1 \varepsilon_1^4 + \kappa^3 (\kappa + q_1)
    \varepsilon_1^3 + 3 \kappa^3 \varepsilon_1^2 + 2 \kappa^2 \varepsilon_1
    \quad \tmop{from} \left( \ref{L2DS} \right)\\
    & \qquad  + \kappa q_1 \varepsilon_1 + 3 \kappa + \frac{3}{2} \kappa^2 q_1
    \varepsilon_1^2 + \frac{3}{2} \kappa^2 \varepsilon_1 + \frac{1}{2}
    \kappa^2 q_1^2 \varepsilon_1^3
    \quad \tmop{from} \left(
    \ref{LG1DS} \textrm{-} \ref{LG4DS} \right) \\
    &\qquad + \frac{1}{2} \kappa^4 K q_1
    \varepsilon_1^3 + \kappa^4 K \varepsilon_1^2 . \quad \tmop{from} \left(
     \ref{LG5DS} \right)
  \end{align*}
  Collecting the previous bound we get $| | \Delta_3 | | \leqslant q_2 q_1
  \varepsilon_1^3$ where
  \begin{align}
    q_2 & = 3 \kappa + \frac{1}{2} (11 \kappa + 2 q_1) \kappa \varepsilon_1
    + \frac{1}{2} (2 \kappa^2 K + 6 \kappa + 7 q_1) \kappa^2 \varepsilon_1^2 
    \label{tau3DS}\\
    &\qquad+\frac{1}{2} (q_1 \kappa^2 K + 2 \kappa^2 + 6 \kappa q_1 + q_1^2) \kappa^2
    \varepsilon_1^3 
     + \frac{1}{2} (3 \kappa + 2 q_1) q_1 \kappa^3 \varepsilon_1^4  \nonumber\\
    &\qquad  +
    q_1^2 \kappa^4 \varepsilon_1^5 + \frac{1}{4} q_1^3 \kappa^4
    \varepsilon_1^6 .  \nonumber
  \end{align}
  Now we are bounding $q_2 \varepsilon_1$. We remark that the monomials which
  appears in $q_2 \varepsilon_1$ are of the form $q_1^i \kappa^j K^k  
  \varepsilon_1^{i + l}$ for some $(i, j, k, l) \in \mathbb{N}^4$ such that $i
  \geqslant 0$, $3 j \leqslant 4 l$ and $3 k \leqslant l$. Since
  $\varepsilon_1 \leqslant \dfrac{\varepsilon}{\kappa^{4 / 3} K^{1 / 3}}$ and
  $q_1 \varepsilon_1 \leqslant \tau_1 \varepsilon$ the we have :
  \begin{align*}
    q_1^i \kappa^j K^k   \varepsilon_1^{i + l} & \leqslant (\tau_1
    \varepsilon)^i \kappa^{j - 4 l / 3} K^{k - l / 3} \varepsilon^l\\
    & \leqslant \tau_1^i \varepsilon^{i + l} \qquad \tmop{since} \kappa, K
    \geqslant 1.
  \end{align*}
  From the expression of $q_2$ it follows after straightforward calculation
  that $q_2 \varepsilon_1 \leqslant \tau_2 \varepsilon$ where
  \begin{align*}
    \tau_2 & = 3 + \frac{1}{2} (11 + 2 \tau_1) \varepsilon  + \frac{1}{2} (8
    + 7 \tau_1) \varepsilon^2 + \frac{1}{2} (\tau_1^2 + 7 \tau_1 + 2)
    \varepsilon^3\\
    &  \qquad + \frac{1}{2} (3 + 2 \tau_1) \tau_1 \varepsilon^4 + \tau_1^2
    \varepsilon^5 + \frac{1}{4} \tau_1^3 \varepsilon^6 .
  \end{align*}
  Since we also have $q_1 \varepsilon_1 \leqslant \tau_1 \varepsilon$ it
  follows
  \begin{align}
    | | \Delta_3 | | & \leqslant \tau_1 \tau_2 \varepsilon^2 \varepsilon_1
    \leqslant \dfrac{1}{\kappa^{4 / 3} K^{1 / 3}} \tau_2 \tau_1 \varepsilon^3
    .  \label{Delta3-bound-DS}
  \end{align}
  The Proposition is proved.
\end{proof}

\begin{lemma}
  \label{lem-Gp2}Let us consider
  \begin{align*}
    G & = - X_2 \Delta_1 + \Delta_1 Y_2 - X_2 \Sigma Y_2 + A_2 \Sigma +
    \Sigma B_2 + \Theta_1^{\ast} \Sigma Y_2 - X_2 \Sigma \Psi_1\\
    &  \qquad + \Theta_1^{\ast} \Sigma B_2 + A_2 \Sigma \Psi_1 - X_2 \Sigma B_2 +
    A_2 \Sigma Y_2 .
  \end{align*}
  Then $G = G_1 + \cdots + G_5$ with
  \begin{align*}
    G_1 & = \frac{1}{2} X_2 (\Delta_2 - S_2) + \frac{1}{2} (S_2 - \Delta_2)
    Y_2\\
    G_2 & = \dfrac{1}{2} (X_1 (\Delta_2 - S_2) + (S_2 - \Delta_2) Y_1) +
    \frac{1}{2} (X_2 (- \Delta_1 - S_1) + (S_1 + \Delta_1) Y_2)\\
    G_3 & = \frac{1}{2} \left( X_1 (\Delta_2 - S_2) Y_1 + X_2 (\Delta_1 -
    S_1) Y_2 + X_1 \left( {\Delta_2}  - S_2 \right) Y_2 \right)\\
    &  \qquad + \frac{1}{2} (X_2 (\Delta_1 - S_1) Y_1 + X_1 (\Delta_1 - S_1) Y_2 +
    X_2 (\Delta_2 - S_2) Y_1)\\
    G_4 & = \frac{1}{2} X_2 (S_2 - \Delta_2) Y_2\\
    G_5 & = e_2 (X_1) \Sigma R_{2, 1} + Q_{2, 1} \Sigma e_2 (Y_1) + e_2
    (X_1) \Sigma e_2 (Y_2) + e_2 (X_2) \Sigma e_2 (Y_1)
  \end{align*}
  where $Q_{2, 1} = \frac{1}{2} (X_1 X_2 + X_2 X_1)$ and $R_{2, 1} =
  \frac{1}{2} (Y_1 Y_2 + Y_2 Y_1)$.
\end{lemma}

\begin{proof}
  Let $e_2 (X) = X^2 / 2$. We have $A_2 = e_2 (X_2) + Q_{2, 1}$ with
  \begin{align*}
    Q_{2, 1} & = \frac{1}{2} (X_1 X_2 + X_2 X_1) .
  \end{align*}
  Moreover $\Theta_1 = X_1 + e_2 (X_1)$. In the same way $B_2 = e_2 (Y_2) +
  R_{2, 1}$ with $R_{2, 1} = \frac{1}{2} (Y_1 Y_2 + Y_2 Y_1)$and $\Psi_1 = Y_1
  + e_2 (Y_1)$. We also remark $e_2 (X_2) = \dfrac{1}{2} X_2^2$. Expanding $G$
  we can write $G$ as the sum of the following quantities :
  \begin{align*}
    G_1 & = - X_2 \Sigma Y_2 + \dfrac{1}{2} X_2^2 \Sigma + \frac{1}{2}
    \Sigma Y_2^2\\
    G_2 & = - X_2 \Delta_1 + \Delta_1 Y_2 + Q_{2, 1} \Sigma + \Sigma R_{2,
    1} - X_1 \Sigma Y_2 - X_2 \Sigma Y_1\\
    G_3 & = - X_1 \Sigma R_{2, 1} + Q_{2, 1} \Sigma Y_1 - X_2 \Sigma R_{2,
    1} + Q_{2, 1} \Sigma Y_2\\
    &  \qquad - X_1 \Sigma e_2 (Y_2) + e_2 (X_2) \Sigma Y_1 + e_2 (X_1) \Sigma Y_2
    - X_2 \Sigma e_2 (Y_1)\\
    G_4 & = - X_2 \Sigma e_2 (Y_2) + e_2 (X_2) \Sigma Y_2\\
    G_5 & = e_2 (X_1) \Sigma R_{2, 1} + Q_{2, 1} \Sigma e_2 (Y_1) + e_2
    (X_1) \Sigma e_2 (Y_2) + e_2 (X_2) \Sigma e_2 (Y_1)
  \end{align*}
  We are going to transform the quantities $G_i$'s. We first remark using
  $\Delta_2 - S_2 - X_2 \Sigma + \Sigma Y_2 = 0$ that
  \begin{align*}
    - X_2 \Sigma Y_2 + \dfrac{1}{2} X_2^2 \Sigma + \frac{1}{2} \Sigma Y_2^2 &
    =  \dfrac{1}{2} X_2 (- \Sigma Y_2 + X_2 \Sigma) + \dfrac{1}{2} (- X_2
    \Sigma + \Sigma Y_2) Y_2\\
    & = \frac{1}{2} X_2 (\Delta_2 - S_2) + \frac{1}{2} (S_2 - \Delta_2) Y_2
    .
  \end{align*}
  Hence
  \begin{align*}
    G_1 & = \frac{1}{2} X_2 (\Delta_2 - S_2) + \frac{1}{2} (S_2 - \Delta_2)
    Y_2 .
  \end{align*}
  Next we remember that $Q_{2, 1} = \dfrac{1}{2} (X_1 X_2 + X_2 X_1)$ and
  $R_{2, 1} = \dfrac{1}{2} (Y_1 Y_2 + Y_2 Y_1)$. On the other hand we have :
  $\Delta_i - S_i - X_i \Sigma + \Sigma Y_i = 0$ for $i = 1, 2$. Hence we can
  write $G_2$ as
  \begin{align*}
    G_2 & = - X_2 \Delta_1 + \Delta_1 Y_2 + Q_{2, 1} \Sigma + \Sigma R_{2,
    1} - X_1 \Sigma Y_2 - X_2 \Sigma Y_1\\
    & = - X_2 \Delta_1 + \Delta_1 Y_2 + \dfrac{1}{2} (X_1 (X_2 \Sigma -
    \Sigma Y_2) + (- X_2 \Sigma + \Sigma Y_2) Y_1)\\
    &  \qquad + \frac{1}{2} (X_2 (- \Sigma Y_1 + X_1 \Sigma) + (- X_1 \Sigma +
    \Sigma Y_1) Y_2)\\
    & = - X_2 \Delta_1 + \Delta_1 Y_2 + \dfrac{1}{2} (X_1 (\Delta_2 - S_2)
    + (S_2 - \Delta_2) Y_1)
    \\&\qquad + \frac{1}{2} (X_2 (\Delta_1 - S_1) + (S_1 -
    \Delta_1) Y_2)\\
    & = \dfrac{1}{2} (X_1 (\Delta_2 - S_2) + (S_2 - \Delta_2) Y_1) +
    \frac{1}{2} (X_2 (- \Delta_1 - S_1) + (S_1 + \Delta_1) Y_2)
  \end{align*}
  Next, by proceeding as above we see that
  \begin{align*}
    G_3 & = - X_1 \Sigma R_{2, 1} + Q_{2, 1} \Sigma Y_1 - X_2 \Sigma R_{2,
    1} + Q_{2, 1} \Sigma Y_2\\
    &  \qquad - X_1 \Sigma e_2 (Y_2) + e_2 (X_2) \Sigma Y_1 + e_2 (X_1) \Sigma Y_2
    - X_2 \Sigma e_2 (Y_1)\\
    & = \frac{1}{2} (- X_1 \Sigma Y_2 Y_1 + X_1 X_2 \Sigma Y_1 - X_2 \Sigma
    Y_1 Y_2 + X_2 X_1 \Sigma Y_2)\\
    &  \qquad + \frac{1}{2} (X_1 X_2 \Sigma Y_2 + X_2 X_1 \Sigma Y_1 - X_1 \Sigma
    Y_1 Y_2 - X_2 \Sigma Y_2 Y_1)\\
    &  \qquad + \frac{1}{2} (- X_1 \Sigma Y_2^2 - X_2 \Sigma Y_1^2 + X_1^2 \Sigma
    Y_2 + X_2^2 \Sigma Y_1)\\
    & = \frac{1}{2} \left( X_1 (\Delta_2 - S_2) Y_1 + X_2 (\Delta_1 - S_1)
    Y_2 + X_1 \left( {\Delta_2}  - S_2 \right) Y_2 \right)\\
    &  \qquad + \frac{1}{2} (X_2 (\Delta_1 - S_1) Y_1 + X_1 (\Delta_1 - S_1) Y_2 +
    X_2 (\Delta_2 - S_2) Y_1)
  \end{align*}
  We now see that
  \begin{align*}
    G_4 & = - X_2 \Sigma e_2 (Y_2) + e_2 (X_2) \Sigma Y_2\\
    & = \frac{1}{2} (- X_2 \Sigma Y_2^2 + X_2^2 \Sigma Y_2)\\
    & = \frac{1}{2} X_2 (S_2 - \Delta_2) Y_2 .
  \end{align*}
  Finally
  \begin{align*}
    G_5 & = e_2 (X_1) \Sigma R_{2, 1} + Q_{2, 1} \Sigma e_2 (Y_1) + e_2
    (X_1) \Sigma e_2 (Y_2) + e_2 (X_2) \Sigma e_2 (Y_1) .
  \end{align*}
  
\end{proof}
\section{Proof of Theorem \ref{th-svd-main} : case $p \geqslant
3$}\label{proof-p=3}

\subsection{Notations}

Let us \ introduce some quantities to simplify the reading of expressions. We
introduce the constants
\begin{align}
  \begin{array}{ll}
    \theta = 0.354, \quad & \eta = \dfrac{1}{1 - \theta},
  \end{array} & a = \frac{4}{3}, \quad b = \frac{1}{3}, \quad u_0 = 0.0297. & 
  \label{cstes}
\end{align}
and the quantities :
\begin{align}
  & \begin{array}{cl}
    w = \frac{1}{\varepsilon} (- 1 + (1 - \varepsilon)^{- 1 / 2}), & s = (1 +
    w \varepsilon)^2 + 2 w + w^2 \varepsilon = 2 (1 - \varepsilon)^{- 1},
    \\ a_1 (\varepsilon) = (1 + \sqrt{1 - \varepsilon^2})^{-
    1}, &\dis a_2 (\varepsilon) = \frac{1}{\varepsilon^2}  (a_1
    (\varepsilon) - 1 / 2)\\
    \dis b_1 (\varepsilon) = \frac{\varepsilon^2 a_1 (\varepsilon)^2}{\sqrt{1 -
    \varepsilon^2}} + 2 a_1 (\varepsilon), & \dis b_2 (\varepsilon) =
    \frac{a_1 (\varepsilon)^2}{\sqrt{1 - \varepsilon^2}} + 2 a_2 (\varepsilon)
    \\
   \hspace{-2cm} \alpha = \eta s, &
  \end{array} &  \label{cstes-th-p+1}
\end{align}
For $i = 1, 2$ we introduce
\begin{align*}
  & \begin{array}{ccccc}
    & {x_i}  = a_i (\eta \varepsilon), & y_i = b_i (\eta \varepsilon), \quad
    {z_i}  = a_i (\theta \varepsilon), & {{r_1} }  ~ = \theta^2 z_1 {+ \eta
    y_1} , & {t_1}  {= 1 + \eta x_1}  \varepsilon .
  \end{array} & 
\end{align*}
and
\begin{align}
  \tau (\varepsilon) & = 2 (1  + \eta) + \left( 2 r_1 + \theta^2  {+ 2 t_1} 
  \eta + \frac{3}{2} \eta^2 + \frac{1}{2} \eta \theta^2 + \frac{1}{2} \theta^4
  \right) \varepsilon_1  \label{def-tau}\\
  &  \qquad + \left( (z_1^2 + 2 z_2 ) \theta^6 + 2 y_1 z_1 \theta^4 + \left( 2 r_1
  + 2 x_1 {z_1}  \eta^2 + \eta^2 {y_1^2}  \right) \theta^2 \right)\varepsilon_1^2
  \nonumber\\
  & \qquad + \left( 2 \left( {y_2}  +
  x_1  y_1  \right) \eta^3 + 2 \eta r_1 t_1 \right) \varepsilon_1^2\nonumber\\
  &  \qquad + (2 z_2 \theta^8 + 2 z_2 \eta \theta^6 + (2 y_2 \eta^3 + r_1^2)
  \theta^2 + 2 (x_2 + y_2) \eta^4 ) \varepsilon_1^3 . \nonumber
\end{align}
The following lemma justifies these notations and will be use in the sequel.

\begin{lemma}
  \label{theta-u0} We have $\tau (s \varepsilon) s \varepsilon - \theta
  \leqslant 0$ and $2 \dfrac{(1 + \alpha \varepsilon)^{b / 3}}{(1 - 2 \alpha
  \varepsilon)^{a / 3}} s^{4 / 3} \tau (s \varepsilon ) \leqslant 1$ and for
  all $\varepsilon \in [0, u_0] .$
\end{lemma}

\begin{proof}
  From straighforward computations.
\end{proof}
\subsection{Proof}

It consists to verify the assumptions of Theorem \ref{general-result}. \
Remember that
\begin{align*}
  \max (\kappa^a K^{b + 1} \| E_{\ell} (U) \|, \kappa^a K^{b + 1} \| E_q (V) |
  |, \kappa^a K^b \| \Delta | | ) & \leqslant \varepsilon
\end{align*}
where $U$, $V$, $\Delta$ stand for $U_0$, $V_0$, $\Delta_0$ respectively. The
item $\left( \ref{gen-H1} \right)$ follows of Proposition \ref{unit-proj-prop}
since $\Omega  = s_p (E_{\ell} (U ))$ and $\Lambda  = s_p (E_q (V ))$. Let us
prove the item $\left( \ref{gen-H2} \right)$. To do that we denote $\Delta_{0,
1} = (I_{\ell} + \Omega ) (\Delta  + \Sigma ) (I_q + \Lambda ) - \Sigma $ and
$\varepsilon_{0, 1} = \| \Delta_{0, 1} \|$. From Proposition
\ref{unit-proj-prop} and assumption $| | E_{\ell} (U ) | |, \| E_q (V ) \|
\leqslant \dfrac{\varepsilon }{\kappa^a K^{b + 1} }$ we know that $\| \Omega 
\|, \| \Lambda  \| \leqslant \cfrac{w}{\kappa^a K^{b + 1} } \varepsilon $ . We
then apply Proposition \ref{part1} to get
\begin{align}
  \varepsilon_{0, 1} & \leqslant ((1 + w \varepsilon )^2 + 2 w + w^2
  \varepsilon )  \frac{\varepsilon}{\kappa^a K^b} \nonumber\\
  & \leqslant \frac{s \varepsilon}{\kappa^a K^b}  \qquad \tmop{from} \left(
  \ref{cstes-th-p+1} \right) .  \label{sei}
\end{align}
In view to use the Propositon \ref{Deltai1}, let us prove that $\tau 
(\varepsilon_{0, 1})  \varepsilon_{0, 1} \le \theta$. Using \ Lemma
\ref{theta-u0} we have
\begin{align*}
  \tau  (\varepsilon_{0, 1})  \varepsilon_{0, 1} & \leqslant \tau (s
  \varepsilon ) s \varepsilon  \qquad \textrm{since $\varepsilon_{0, 1}
  \leqslant s \varepsilon $ }\\
  & \leqslant \theta \qquad \textrm{ from Lemma \ref{theta-u0}} \quad
  \tmop{since} \varepsilon \leqslant u_0 .
\end{align*}
From formulas $\left( \ref{def-Deltai} \right)$ we have
\begin{align*}
  \Delta_1 = \Delta_{0, p + 1} & = (I_{\ell} + \Theta_p^{\ast}) (\Delta_{0,
  1} + \Sigma) (I_q + \Psi_p) - \Sigma - \sum_{k = 1}^p S_k .
\end{align*}
The quantity $\tau$ which appears in $\left( \ref{tau-zetai-1} \right)$ is
equal to $\tau (s \varepsilon)^p s^{p + 1}$. Using Propositon~\ref{Deltai1}
with $\tau \assign \tau  (s \varepsilon)^p s^{p + 1}$, we then get
\begin{align*}
  \| \Delta_1 \| & =  \| \Delta_{0, p + 1} \|\\
  & \leqslant \frac{1}{\kappa^a K^b} (\tau  (s \varepsilon) s^{\frac{p +
  1}{p}})^p \varepsilon^{p + 1} \qquad \tmop{since} \quad \varepsilon_{0, 1}
  \leqslant s \varepsilon .
\end{align*}
On the other hand from definition $S  = S_1 + \cdots + S_p$ where $S_k =
\tmop{diag} (\Delta_{0, k})$. It follows $\| S_i \| \leqslant \varepsilon_{0,
k} = \| \Delta_{0, k} \|$. From Proposition \ref{Deltai1} one has
\begin{align*}
  \varepsilon_{0, k} & \leqslant \tau (s \varepsilon) ^{k - 1}
  \varepsilon_{0, 1}^k\\
  & \leqslant \theta^{k - 1} \varepsilon_{0, 1} \qquad \tmop{since} \quad
  \tau (s \varepsilon) s \varepsilon  \leqslant \theta \infixand
  \varepsilon_{0, 1} \leqslant \frac{s \varepsilon}{\kappa^a K^b}
\end{align*}
We deduce
\begin{align}
  \| S \| & \leqslant \sum_{k = 1}^p \varepsilon_{0, k} \leqslant \frac{1}{1
  - \theta} \varepsilon_{0, 1} \leqslant \frac{\alpha \varepsilon}{\kappa^a
  K^b} .  \label{Sinorm}
\end{align}
The assumption $\left( \ref{tau-zetai-1} \right)$ is satisfied. In fact we
have
\begin{align}
  (2 \varepsilon)^p  \frac{(1 + \alpha \varepsilon)^b}{(1 - 2 \alpha
  \varepsilon)^a} \tau  (s \varepsilon)^p s^{p + 1} & \leqslant \left( 2
  \frac{(1 + \alpha \varepsilon)^{b / 3}}{(1 - 2 \alpha \varepsilon)^{a / 3}}
  \tau (s \varepsilon) s^{4 / 3} \varepsilon \right)^p \qquad \tmop{since} p
  \geqslant 3 \infixand s \geqslant 1 \nonumber\\
  & \leqslant 1 \qquad \textrm{ from Lemma \ref{theta-u0}} \quad
  \tmop{since} \varepsilon \leqslant u_0 .  \label{2eptau-p+1}
\end{align}
We now prove the item $(\nobracket$\ref{gen-H3}$\nobracket)$. We have
\begin{align*}
  \| I_{\ell} + \Theta  \|^2 & \leqslant (1 + | | c_p (X )  | |)^2\\
  \| (I_{\ell} + \Theta^{\ast}) (I_{\ell} + \Theta ) - I_{\ell} \| & \leqslant
  (1 + c_p (- \| X  \|)) (1 + c_p (\| X \|)) - 1
\end{align*}
Using Lemma \ref{epCp1} and $\dis\varepsilon_{0, 1} \leqslant s
\frac{\varepsilon}{\kappa^a K^b}$ we know that 
$\dis\| X \| \leqslant \eta \kappa
\varepsilon_{0, 1} \leqslant \dfrac{x}{\kappa^{a - 1} K^b} =
\dfrac{x}{\kappa^b K^b}$ with $x = \alpha \varepsilon$. We deduce both from
Lemma \ref{epCp1} that
\begin{align}
  (\nobracket 1 + | | c_p (X)  |)^2 & \leqslant (1 + x + x^2 a_1 (x) )^2 =
  \zeta_1  \label{ineq-cpXi}
\end{align}
and from Lemma \ref{cpu} that
\begin{align}
  (1 + c_p (- \| X \|)) &(1 + c_p (\| X \|)) - 1\label{ineq-cpX}\\
   & \leqslant \left( 2 \sqrt{1
  - x^2} + a_1 (x) x^{p + 1} \right) a_1 (x) \left( \frac{1}{\kappa^{a - 1}
  K^b} \alpha \varepsilon \right)^{p + \delta} \nonumber\\
  & \leqslant \left( 2 \sqrt{1 - x^2} + a_1 (x) x^3 \right) a_1 (x)
  \alpha^{p + \delta}  \left( \frac{1}{\kappa^b K^b} \right)^{p + \delta}
  \varepsilon^{p + 1} \nonumber\\
  & \leqslant \frac{\zeta_2}{\kappa^a K^{b + 1}} \varepsilon^{p + 1}
  \tmop{since} \enspace p \geqslant 3 \tmop{implies} \enspace (p + \delta) b
  \geqslant b + 1  \nonumber
\end{align}
where \ $\delta = 1$ if $p$ is odd and $\delta = 2$ if $p$ is even from Lemma
\ref{cpu}. We then remark that
\begin{align}
  (2 \epsilon)^p \alpha^{p + \delta} \varepsilon^{\delta - 1} & \leqslant (2
  \alpha^{5 / 3} \varepsilon )^p \quad \tmop{since} \quad \frac{p + \delta}{p}
  \leqslant \frac{5}{3}  \label{2epk}
\end{align}
This allows to prove the assumption $\left( \ref{tau-zetai-2} \right)$ that
is $(2 \varepsilon)^p \dfrac{(1 + \alpha \varepsilon)^{b + 1}}{(1 - 2 \alpha
\varepsilon)^a} (\zeta_1 + \zeta_2 \varepsilon^{\delta - 1}) \leqslant 1$. We
first have since $b + 1 = a$
\begin{align*}
  (2 \varepsilon)^p \left( \frac{1 + \alpha \varepsilon}{1 - 2 \alpha
  \varepsilon} \right)^a \zeta_1 & \leqslant \left( 2 \left( \frac{1 +
  \alpha \varepsilon}{1 - 2 \alpha \varepsilon} \right)^{a / 3} (1 + x + x^2
  a_1 (x))^{2 / 3} \varepsilon \right)^p\\
  & \leqslant (0.037)^p \leqslant 0.00005 \quad \tmop{since} \quad
  \varepsilon \leqslant u_0  \infixand p \geqslant 3.
\end{align*}
We now remark that
\begin{align*}
  \zeta_2 = \left( 2 \sqrt{1 - x^2} + a_1 (x) x^3 \right) a_1 (x) & \leqslant
  & 0.998 \quad \tmop{since} \varepsilon \leqslant u_0 \tmop{implies} x
  \leqslant 0.098.
\end{align*}
Taking in account $\left( \ref{ineq-cpX} \right.$-$\left. \ref{2epk} \right)$
we get :
\begin{align*}
  (2 \varepsilon)^p \left( \frac{1 + \alpha \varepsilon}{1 - 2 \alpha
  \varepsilon} \right)^a \zeta_2 \varepsilon^{\delta - 1} & \leqslant \left(
  2 \left( \frac{1 + \alpha \varepsilon}{1 - 2 \alpha \varepsilon} \right)^{a
  / 3}  \alpha^{5 / 3} \varepsilon \right)^p\\
  & \leqslant (0.24)^p \leqslant 0.013 \quad \quad \tmop{since} \quad
  \varepsilon \leqslant u_0  \infixand p \geqslant 3.
\end{align*}
Consequently $(2 \varepsilon)^p \dfrac{(1 + \alpha \varepsilon)^a}{(1 - 2
\alpha \varepsilon)^a} (\zeta_1 + \zeta_2 \varepsilon^{\delta - 1}) \leqslant
0.015 < 1$. This proves the item $\left( \ref{tau-zetai-2} \right)$. The
assumption $\left( \ref{8ae} \right)$ holds since $1 - 8 \alpha \varepsilon
\geqslant 0.25 > 0$ when $\varepsilon < u_0$.

We now verify the assumption $(\nobracket$\ref{gen-OiLiTiPi}$\nobracket)$.
From above we know that $\| \Omega  \|, \| \Lambda  \| \leqslant
\cfrac{w}{\kappa^a K ^{b + 1}} \varepsilon $ with $w = \dfrac{1}{\varepsilon}
(- 1 + (1 - \varepsilon)^{- 1 / 2})$. We can take $w \leqslant \alpha_1 =
0.52$ since $\varepsilon \leqslant u_0$.

On the other hand one has $\Theta  = c_p (X)$ and $\Psi  = c_p (Y)$. From
above \ we know that

\begin{align*}
  \|c_p (X )\|, \|c_p (Y )\| & \leqslant (1 + x a_1 (x))  \hspace{0.17em} x
  \quad \tmop{with} \quad x = \alpha \varepsilon\\
  & \leqslant \alpha_2 \varepsilon \qquad \tmop{with} \tmxspace \alpha_2 =
  3.35 \quad \tmop{since} \varepsilon \leqslant u_0 .
\end{align*}

Since $\gamma u_0 = 2 (\alpha_1 + \alpha_2 + \alpha_1 \alpha_2 u_0) u_0 <
0.233 < 1$ then the bounds $\left( \ref{gen-Ui} \right.$-$\left.
\ref{gen-Sigmai} \right)$ of Theorem \ref{general-result} hold with
\begin{align*}
  \gamma & = 7.82\\
  \dfrac{\gamma}{1 - \gamma u_0} & \leqslant 10.2\\
  \sigma = 0.82 \alpha & \leqslant 2.62
\end{align*}
The Theorem \ref{th-svd-main} is proved for $p \geqslant 3$. $\Box$

\begin{proposition}
  \label{Deltai1}Let $p > 2$, \ $\varepsilon \geqslant 0$. Let us consider
  $\Delta_1 = U_1^{\ast} M V_1 - \Sigma$ such that $\| \Delta_1 \| =
  \varepsilon_1 \leqslant \dfrac{\varepsilon}{\kappa^{4 / 3} K^{1 / 3}}$ where
  $\kappa = \kappa (\Sigma)$ and $K = K (\Sigma)$. Let us consider $\tau
  \assign \tau (\varepsilon)$ as in $\left( \ref{def-tau} \right)$ and suppose
  $\tau \varepsilon \le \theta$. Then we have
  \[ \tau_{p + 1} \assign \| \Delta_{p + 1} \| \leqslant \frac{1}{\kappa^{4 /
     3} K^{1 / 3}} \tau (\varepsilon)^p \varepsilon^{p + 1} \]
  where $\Delta_{p + 1} = \dis (I_{\ell} + \Theta_p^{\ast})  (\Delta_1 + \Sigma) 
  (I_q + \Psi_p) - \Sigma - \dis\sum_{l = 1}^p S_l$, with $\Theta_p$ and $\Psi_p$
  are defined by the formulas~$(\nobracket$\ref{def-Deltai}$\nobracket)$.
\end{proposition}
\begin{proof}
  Since the $X_k$'s and $Y_k$'s are skew Hermitian matrices, we have $\Theta_p
  = \Theta_{p - 1}  + X_p + A_p$ with
  \[ A_p \assign A_p  (X_1 + \ldots + X_{p - 1}, X_p) = c_p (X_1 + \cdots +
     X_p) - c_p (X_1 + \cdots + X_{p - 1}) - X_p \]
  In the same way $\Psi_p = \Psi_{p - 1} + Y_p + B_p$ where $B_p = A_p (Y_1 +
  \cdots + Y_{p - 1}, Y_p)$. We remark that $A_p$ and $B_p$ are Hermitian
  matrices. \ Expanding $(I_{\ell} + \Theta_p^{\ast})  (\Delta_1 + \Sigma) 
  (I_q + \Psi_p)$ we get
  \begin{align*}
    \Delta_{p + 1} & = (I_{\ell} + \Theta_p^{\ast})  (\Delta_1 + \Sigma)  (I_q
    + \Psi_p) - \Sigma - \sum_{l = 1}^p S_l\\
    & = (I_{\ell} + \Theta_{p - 1}^{\ast} - X_p + A_p)  (\Delta_1 + \Sigma) 
    (I_q + \Psi_{p - 1} + Y_p + B_p) - \Sigma - \sum_{l = 1}^p S_l\\
    & = (I_{\ell} + \Theta_{p - 1}^{\ast})  (\Delta_1 + \Sigma)  (I_q +
    \Psi_{p - 1}) - \Sigma - \sum_{l = 1}^{p - 1} S_l - S_p - X_p \Sigma +
    \Sigma Y_p\\
    & \qquad + (I_{\ell} + \Theta_{p - 1}^{\ast})  (\Delta_1 + \Sigma)  (Y_p
    + B_p) + (- X_p + A_p)  (\Delta_1 + \Sigma)  (I_q + \Psi_{p - 1})\\
    & \\
    & \qquad + (- X_p + A_p)  (\Delta_1 + \Sigma)  (Y_p + B_p) + X_p \Sigma -
    \Sigma Y_p .
  \end{align*}
  From definition we know that
  \[ (I_{\ell} + \Theta_{p - 1}^{\ast})  (\Delta_1 + \Sigma)  (I_q + \Psi_{p -
     1}) - \Sigma - \sum_{l = 1}^{p - 1} S_l - S_p - X_p \Sigma + \Sigma Y_p =
     \Delta_p - S_p - X_p \Sigma + \Sigma Y_p = 0. \]
  Expanding more $\Delta_{p + 1}$, we then can write by grouping the terms
  appropriately :
  \begin{align}
    \Delta_{p + 1} & = - X_p \Delta_1 + \Delta_1 Y_p - X_p \Delta_1 Y_p +
    \Delta_1 B_p + A_p \Delta_1 - X_p \Delta_1 B_p + A_p \Delta_1 Y_p 
    \label{L1}\\
    &  \qquad + A_p \Delta_1 B_p + \Theta_{p - 1}^{\ast} \Delta_1 Y_p - X_p
    \Delta_1 \Psi_{p - 1} + \Theta_{p - 1}^{\ast}\Delta_1 B_p + A_p
    \Delta_1 \Psi_{p - 1}  \label{L2}\\
    &  \qquad + G, \nonumber
  \end{align}
  where \ $G = - X_p \Sigma Y_p + \Sigma B_p + A_p \Sigma + \Theta_{p -
  1}^{\ast} \Sigma Y_p - X_p \Sigma \Psi_{p - 1} + \Theta_{p - 1}^{\ast}
  \Sigma B_p + A_p \Sigma \Psi_{p - 1} - X_p \Sigma B_p + A_p \Sigma Y_p + A_p
  \Sigma B_p$. From the Lemma \ref{lem-Gp} the quantity $G$ is sum of the
  following $G_i$'s :
  \begin{align}
    G_1 & = d_p (X_p) \Sigma + \Sigma d_p (Y_p)  \label{LG1}\\
    G_2 & = Q_{p, 2} \Sigma + \Sigma R_{p, 2} + \dfrac{1}{2} C_{p - 1}
    (\Delta_p - S_p) - \frac{1}{2} (\Delta_p - S_p) D_{p - 1}  \label{LG2}\\
    &  \qquad + \frac{1}{2} X_p \sum_{k = 1}^p (\Delta_k - S_k) + \frac{1}{2}
    \sum_{k = 1}^p (S_k - \Delta_k) Y_p \nonumber\\
    G_3 & = \frac{1}{2} C_{p - 1} (\Delta_p - S_p) D_{p - 1} - \frac{1}{2}
    X_p \sum_{k = 1}^{p - 1} (\Delta_k - S_k) Y_p  \label{LG3}\\
    &  \qquad + \frac{1}{2} X_p \sum_{k = 1}^p (\Delta_k - S_k) D_{p - 1} +
    \frac{1}{2} C_{p - 1} \sum_{k = 1}^p (S_k - \Delta_k) Y_p \nonumber\\
    G_4 & = \frac{1}{2} X_p (S_p - \Delta_p) Y_p - X_p \Sigma d_p (Y_p) +
    d_p (X_p) \Sigma Y_p .  \label{LG4}\\
    G_5 & = e_p (C_{p - 1}) \Sigma R_{p, 1} + Q_{p, 1} \Sigma e_p (D_{p -
    1}) + e_p (C_{p - 1}) \Sigma e_p (Y_p)     \label{LG51}\\
    &\qquad + e_p (X_p) \Sigma e_p (D_{p - 1}) 
   + Q_{p, 1} \Sigma R_{p, 1} + Q_{p, 1} \Sigma e_p (Y_p) \label{LG52}
   \\&\qquad + e_p (X_p)
    \Sigma R_{p, 1} + e_p (X_p) \Sigma e_p (Y_p) .  \label{LG53}
    \\
    G_6 & = - C_{p - 1} \Sigma R_{p, 2} + Q_{p, 2} \Sigma D_{p - 1} - X_p
    \Sigma R_{p, 2} + Q_{p, 2} \Sigma Y_p  \label{LG6}\\
    &  \qquad - C_{p - 1} \Sigma d_p (Y_p) + d_p (X_p) \Sigma D_{p - 1}\nonumber
    \\&\qquad + d_p
    (C_{p - 1}) \Sigma Y_p - X_p \Sigma d_p (D_{p - 1}) . \nonumber
  \end{align}
  where the quantities $Q_{p, i}$ and $R_{p, i}$ are defined at Lemma
  \ref{lem-G}. We now can bound $\| \Delta_{p + 1} \|$. To do that introduce
  the quantities where \ $i = 1, 2$ :
  \begin{align*}
    & \begin{array}{ccccc}
      & {x_i}  = a_i (\eta \varepsilon ), & y_i = b_i (\eta \varepsilon),
      \quad {z_i}  = a_i (\theta \varepsilon), & {r_1}  ~ = \theta^2 z_1 ~ + ~
      \eta y_1, & t_1 = 1 + x_1 \eta \varepsilon 
    \end{array} & 
  \end{align*}
  and the polynomial $q \assign q (\kappa, K, \varepsilon_1)$
  \begin{align*}
    q & = 2 (1  + \eta) \kappa + \left( 2 r_1 + \theta^2  {+ 2 t_1}  \eta +
    \frac{3}{2} \eta^2 + \frac{1}{2} \eta \theta^2 + \frac{1}{2} \theta^4
    \right) \kappa^2 \varepsilon_1\\
    &  \qquad + \left( (z_1^2 + 2 z_2 ) \theta^6 + 2 \eta x_1 z_1 \theta^4\right )
      K    \kappa^4 \varepsilon_1^2
\\&\qquad+   \left( \left( 2 r_1 + 2 x_1 {z_1}  \eta^2 + \eta^2 {y_1^2}  \right) \theta^2 + 2
    \left( {y_2}  + x_1  y_1  \right) \eta^3 + 2 \eta r_1 t_1 \right) K
    \kappa^4 \varepsilon_1^2\\
    &  \qquad + (2 z_2 \theta^8 + 2 z_2 \eta \theta^6 + (2 y_2 \eta^3 + r_1^2)
    \theta^2 + 2 (x_2 + y_2) \eta^4 ) K \kappa^5 \varepsilon_1^3 .
  \end{align*}
  The inequality $\tau (\varepsilon) \varepsilon \leqslant \theta$ implies $q
  \varepsilon_1 \leqslant \theta$. In fact it is easy to see that the
  assumption \ $\varepsilon_1 \leqslant \dfrac{\varepsilon}{\kappa^{4 / 3}
  K^{1 / 3}}$ implies $q \varepsilon_1 \leqslant \tau (\varepsilon)
  \varepsilon$ since we simultaneously have $\kappa \varepsilon_1 \leqslant
  \varepsilon$, $\kappa^2 \varepsilon_1^2 \leqslant \varepsilon^2$, $K
  \kappa^4 \varepsilon_1^3 \leqslant \varepsilon^3$ and $K \kappa^5
  \varepsilon_1^4 \leqslant \varepsilon^4$. We know that $\| \Delta_1 \|
  \leqslant \varepsilon_1$. Let us suppose $\| \Delta_k \| \leqslant q^{k - 1}
  \varepsilon_1^k$ for $1 \leqslant k \leqslant p$ and, prove that $\|
  \Delta_{p + 1} \| \leqslant q^p \varepsilon_1^{p + 1}$. We remark $q
  \geqslant 2 (\theta + \eta)$ in order that the Lemmas \ref{epCp1}-\ref{dpXp}
  apply. To bound $\| \Delta_{p + 1} \|$ we use \ the following bounds :
  \begin{enumerate}
    \item We have \ for $i = 1, 2$, \ $a_i (\theta \kappa \varepsilon_1)
    \leqslant x_i \quad b_i (\eta \kappa \varepsilon_1) \leqslant y_i$.
    \\
    \item For $1 \leqslant k \leqslant p$, we know that $\| X_k \|$, $\| Y_k
    \| \leqslant \kappa q^{k - 1} \varepsilon_1^k$ from Proposition
    \ref{fund-pert-prop}.
    \\
    \item $\| C_k \|, \| D_k \| \leqslant \eta \kappa \varepsilon_1$ from
    Lemma \ref{epCp1} and also $\left\| \dis\sum_{k = 1}^{p - 1} \Delta_k - S_k
    \right\| \leqslant \eta \varepsilon_1$ from Lemma \ref{SDeltak}.
    \\
    \item $\| Q_{p, i} \|, \| R_{p, i} \| \leqslant \eta^{2 i - 1} y_i
    \kappa^{2 i} q^{p - 1} \varepsilon_1^{p + 2 i - 1}$ from Lemma \ref{Rp}.
    \\
    \item $\| \nobracket e_p (X_p) | |, \| \nobracket e_p (Y_p) | | {\leqslant
    z_1}  \kappa^2 q^{2 (p - 1)} \varepsilon_1^{2 p} \leqslant \theta^2 z_1
    \kappa^2 q^{p - 1} \varepsilon_1^{p + 1}$
    
    and $\| \nobracket e_p (C_{p -
    1}) | |, \| \nobracket e_p (D_{p - 1}) | | {\leqslant x_1}  \eta^2
    \kappa^2 \varepsilon_1^2$ from Lemma \ref{epCp1}, ${q \varepsilon_1} 
    \leqslant \theta$ and $p \geqslant 3$.
    \\
    \item $\| \nobracket d_p (X_p) | |, \| \nobracket d_p (Y_p) | | {\leqslant
    z_2}  \kappa^4 q^{4 (p - 1)} \varepsilon_1^{4 p}$ and $\| \nobracket d_p
    (C_{p - 1}) | |, \| \nobracket d_p (D_{p - 1}) | | {\leqslant x_2}  \eta^4
    \kappa^4 \varepsilon_1^4$ from Lemma \ref{dpCp1} and $q \varepsilon
    \leqslant \theta$.
    \\
    \item $\| A_p \|, \| B_p \| {\le r_1}  \kappa^2  \hspace{0.17em} q^{p - 1}
    \varepsilon_1^{p + 1}$\quad since $A_p = e_p (X_p) + Q_{p, 1}$ and $B_p =
    e_p (Y_p) + R_{p, 1}$.
    \\
    \item $\| \Theta_{p - 1} \|, \| \Psi_{p - 1} \| \le \hspace{0.17em} {t_1} 
    \eta \kappa \varepsilon_1$\quad from Lemma~\ref{Thetap1}.
    \\
    \item $\begin{array}{lll}
      | | - X_p \Sigma d_p (Y_p) + d_p (X_p) \Sigma Y_p | | & \leqslant 2 K
      z_2 \kappa^5 q^{5 (p - 1)} \varepsilon_1^{5 p}
    \end{array}$ from Lemma \ref{dpXp}.
  \end{enumerate}
  Using the bounds above we then get $\| \Delta_{p + 1} \| \leqslant \alpha_{p
  + 1} q^{p - 1} \varepsilon_1^{p + 1}$where
$$  \begin{array}{ll}
    \alpha_{p + 1}=&\\  
     2 \kappa + \kappa^2 q^{p - 1} \varepsilon_1^p {+ 2
    r_1}  \kappa^3 q^{p - 1} \varepsilon_1^{p + 1} {+ 2 r_1}  {\kappa^2} 
    \varepsilon_1
    & \tmop{from}  \left( \ref{L1}\right)\\
       {+ r_1^2}  \kappa^4 q^{p - 1}  \hspace{0.17em} \varepsilon_1^{p + 2}
    {+ 2 t_1}  \eta \kappa^2  \hspace{0.17em} \varepsilon_1  {+ 2 r_1}  {t_1} 
    \eta \kappa^3  \hspace{0.17em} \hspace{0.17em} \varepsilon_1^2 
    & \tmop{from}  \left( \ref{L2} \right)\\
     {+ 2 z_2}  K \kappa^4 q^{3 (p - 1)} \varepsilon_1^{3 p - 1} + 2
    \eta^3 {y_2}  K \kappa^4 \varepsilon_1^2 + 2 \eta \kappa 
    &
    \tmop{from}  \left( \ref{LG1} + \ref{LG2} \right)\\
     + \frac{3}{2} \eta^2 \kappa^2 \varepsilon_1  + \frac{1}{2} \eta
    \kappa^2 q^{p - 1} \varepsilon_1^p 
    & \tmop{from}     \left( \ref{LG3} \right)\\
   + \frac{1}{2} \kappa^2 q^{2 (p - 1)} \varepsilon_1^{2 p - 1} + 2 z_2
    K \kappa^5 q^{4 (p - 1)} \varepsilon_1^{4 p - 1} 
    & \tmop{from}     \left( \ref{LG4} \right)\\
   + K \kappa^4  (2 x_1 y_1 \eta^3 \varepsilon_1^2 + 2 z_1 x_1
    \eta^2 q^{p - 1} \varepsilon_1^{p + 1}) 
    & \tmop{from}
     \left( \ref{LG51} \right)\\
    + K \kappa^4  (y_1^2 \eta^2 q^{p - 1} \varepsilon_1^{p + 1} + 2 z_1
    y_1  \eta  q^{2 (p - 1)} \varepsilon_1^{2 p} + z_{1 }^2 q^{3 (p - 1)}
    \varepsilon^{3 p - 1}) & \tmop{from}  \left( \ref{LG52}-\ref{LG53}
    \right)\\
      + K  \kappa^5  (2 \eta^4  (x_2 + y_2) \varepsilon_1^3 + 2 z_2   \eta
    q^{3 (p - 1)} \varepsilon_1^{3 p} + 2 y_2   \eta^3 q^{p - 1}
    \varepsilon_1^{p + 2}) & \tmop{from}  \left(
    \ref{LG6} \right)
  \end{array}
$$  
  Since $p \geqslant 3$ and $\theta < 1$ it follows $(q \varepsilon_1)^{k (p -
  1)} \leqslant (q \varepsilon_1)^{2 k} \leqslant (\tau \varepsilon)^{2 k}
  \leqslant \theta^{2 k}$. Plugging this in $\alpha_{p + 1}$, we then get
\begin{align*}
    \alpha_{p + 1} & \leqslant 2 \kappa + \kappa^2 \theta^2 \varepsilon_1 
    {+ 2 r_1}  \kappa^3 \theta^2 \varepsilon_1^2 {+ 2 r_1}  {\kappa^2} 
    \hspace{0.17em} \varepsilon_1  \hspace{3em}\\
    &  \qquad {+ r_1^2}  \kappa^4 \theta^2  \hspace{0.17em} \varepsilon_1^3 {+ 2
    t_1}  \eta \kappa^2 \varepsilon_1  {+ 2 r_1}  {t_1}  \eta \kappa^3 
    \hspace{0.17em} \varepsilon_1^2 \qquad\\
    &  \qquad {+ 2 z_2}  K \kappa^4 \theta^6 \varepsilon_1^2 + 2 \eta^3 {y_2}  K
    \kappa^4 \varepsilon_1^2 + 2 \eta \kappa \hspace{4em}\\
    &  \qquad+ \frac{3}{2} \eta^2 \kappa^2 \varepsilon_1  + \frac{1}{2} \eta
    \kappa^2 \theta^2 {\varepsilon _1}  \hspace{3em}\\
    &  \qquad + \frac{1}{2} \kappa^2 \theta^4 {\varepsilon _1}  + 2 z_2 K \kappa^5
    \theta^8 \varepsilon_1^3\\
    &  \qquad + K \kappa^4  (2 x_1 y_1 \eta^3 \varepsilon_1^2 + 2 z_1 x_1 
    \eta^2 \theta^2 \varepsilon_1^2) \hspace{3em}\\
    &  \qquad + K \kappa^4  (y_1^2 \eta^2 \theta^2 \varepsilon_1^2 + 2 z_1 y_1 
    \eta  \theta^4 \varepsilon_1^2 + z_{1 }^2 \theta^6 \varepsilon_1^2)
    \qquad\\
    &  \qquad + K  \kappa^5  (2 \eta^4 (x_2 + y_2) \varepsilon_1^3 + 2 z_2   \eta
    \theta^6 \varepsilon_1^3 + 2 y_2 \eta^3 \theta^2 \varepsilon_1^3) . \qquad
  \end{align*}
  Collecting the expression above following $\varepsilon_1$ and using that
  $\kappa, K \geqslant 1$, we finally find that $\alpha_{p + 1} \leqslant q$.
  We then have proved that \ $\| \Delta_{p + 1} \| \leqslant q^p
  \varepsilon_1^{p + 1} .$ \ We finally get
  \begin{align*}
    \| \Delta_{p + 1} \| & \leqslant \tau (\varepsilon)^p \varepsilon^p
    \varepsilon_1\\
    & \leqslant \frac{1}{\kappa^{4 / 3} K^{1 / 3}} \tau (\varepsilon)^p
    \varepsilon^{p + 1} .
  \end{align*}
  The theorem is proved.
\end{proof}

\begin{lemma}
  \label{lem-Gp}Let us consider
  \begin{align*}
    G & = - X_p \Sigma Y_p + A_p \Sigma + \Sigma B_p + \Theta_{p - 1}^{\ast}
    \Sigma Y_p - X_p \Sigma \Psi_{p - 1}\\
    &  \qquad + \Theta_{p - 1}^{\ast} \Sigma B_p + A_p \Sigma \Psi_{p - 1} - X_p
    \Sigma B_p + A_p \Sigma Y_p + A_p \Sigma B_p .
  \end{align*}
  Let $C_{p - 1} = X_1 + \cdots + X_{p - 1}$ and $D_{p - 1} = Y_1 + \cdots +
  Y_{p - 1}$. Then $G = G_1 + \cdots + G_6$ with
  \begin{align*}
    G_1 & = d_p (X_p) \Sigma + \Sigma d_p (Y_p)\\
    G_2 & = Q_{p, 2} \Sigma + \Sigma R_{p, 2} + \dfrac{1}{2} C_{p - 1}
    (\Delta_p - S_p) - \frac{1}{2} (\Delta_p - S_p) D_{p - 1}\\
    &  \qquad + \frac{1}{2} X_p \sum_{k = 1}^p (\Delta_k - S_k) + \frac{1}{2}
    \sum_{k = 1}^p (S_k - \Delta_k) Y_p .\\
    G_3 & = \frac{1}{2} C_{p - 1} (\Delta_p - S_p) D_{p - 1} - \frac{1}{2}
    X_p \sum_{k = 1}^{p - 1} (\Delta_k - S_k) Y_p\\
    &  \qquad + \frac{1}{2} X_p \sum_{k = 1}^p (\Delta_k - S_k) D_{p - 1} +
    \frac{1}{2} C_{p - 1} \sum_{k = 1}^p (S_k - \Delta_k) Y_p\\
    G_4 & = \frac{1}{2} X_p (S_p - \Delta_p) Y_p - X_p \Sigma d_p (Y_p) +
    d_p (X_p) \Sigma Y_p .\\
    G_5 & = e_p (C_{p - 1}) \Sigma R_{p, 1} + Q_{p, 1} \Sigma e_p (D_{p -
    1}) + e_p (C_{p - 1}) \Sigma e_p (Y_p) + e_p (X_p) \Sigma e_p (D_{p -
    1})\\
    &  \qquad + Q_{p, 1} \Sigma R_{p, 1} + Q_{p, 1} \Sigma e_p (Y_p) + e_p (X_p)
    \Sigma R_{p, 1} + e_p (X_p) \Sigma e_p (Y_p) .\\
    G_6 & = - C_{p - 1} \Sigma R_{p, 2} + Q_{p, 2} \Sigma D_{p - 1} - X_p
    \Sigma R_{p, 2} + Q_{p, 2} \Sigma Y_p\\
    &  \qquad - C_{p - 1} \Sigma d_p (Y_p) + d_p (X_p) \Sigma D_{p - 1} + d_p
    (C_{p - 1}) \Sigma Y_p - X_p \Sigma d_p (D_{p - 1}) .s
  \end{align*}
\end{lemma}

\begin{proof}
  We have $A_p = e_p (X_p) + Q_{p, 1} = \dfrac{1}{2} X_p^2 + d_p (X_p) +
  Q_{p, 1}$ with
  \begin{align*}
    Q_{p, i} & = \sum_{k = i}^{\max (k : 2 k \leqslant p)} c_k
    \sum_{\tmscript{\begin{array}{c}
      i_1 + i_2 = 2 k\\
      i_1, i  > 0
    \end{array}}} L_{i_1, i_2} (C_{p - 1}, X_p) .
  \end{align*}
  where the coefficients $c_k$ and the polynomials $L_{i_1, i_2}$ are defined
  at the beginning of the section \ref{sec-uLP}. Moreover $\Theta_{p - 1} =
  C_{p - 1} + e_p (C_{p - 1})$. In the same way $B_p = e_p (Y_p) + R_{p, 1} =
  \dfrac{1}{2} Y_p^2 + d_p (Y_p) + R_{p, 1}$ and $\Psi_{p - 1} = D_{p - 1} +
  e_p (D_{p - 1})$. We also know that $\Theta^{\ast}_{p - 1} = - C_{p - 1} +
  e_p (C_{p - 1})$ since $C_{p - 1}$ is a skew Hermitian matrix. \ Expanding
  \begin{align*}
    G & = - X_p \Sigma Y_p + A_p \Sigma + \Sigma B_p + \Theta_{p - 1}^{\ast}
    \Sigma Y_p - X_p \Sigma \Psi_{p - 1}\\
    &  \qquad + \Theta_{p - 1}^{\ast} \Sigma B_p + A_p \Sigma \Psi_{p - 1} - X_p
    \Sigma B_p + A_p \Sigma Y_p + A_p \Sigma B_p,
  \end{align*}
  a straightforward calculation shows that we can write $G$ as the sum of the
  following quantities :
  \begin{align*}
    G_1 & = d_p (X_p) \Sigma + \Sigma d_p (Y_p) \quad\\
    G_2 & = Q_{p, 1} \Sigma + \Sigma R_{p, 1} - C_{p - 1} \Sigma Y_p - X_p
    \Sigma D_{p - 1} - X_p \Sigma Y_p + \dfrac{1}{2} X_p^2 \Sigma +
    \frac{1}{2} \Sigma Y_p^2\\
    G_3 + G_6 & = - C_{p - 1} \Sigma R_{p, 1} + Q_{p, 1} \Sigma D_{p - 1} -
    X_p \Sigma R_{p, 1} + Q_{p, 1} \Sigma Y_p\\
    &  \qquad - C_{p - 1} \Sigma e_p (Y_p) + e_p (X_p) \Sigma D_{p - 1} + e_p
    (C_{p - 1}) \Sigma Y_p - X_p \Sigma e_p (D_{p - 1})\\
    G_4 & = - X_p \Sigma e_p (Y_p) + e_p (X_p) \Sigma Y_p\\
    G_5 & = e_p (C_{p - 1}) \Sigma R_{p, 1} + Q_{p, 1} \Sigma e_p (D_{p -
    1}) + e_p (C_{p - 1}) \Sigma e_p (Y_p) + e_p (X_p) \Sigma e_p (D_{p -
    1})\\
    &  \qquad + Q_{p, 1} \Sigma R_{p, 1} + Q_{p, 1} \Sigma e_p (Y_p) + e_p (X_p)
    \Sigma R_{p, 1} + e_p (X_p) \Sigma e_p (Y_p) .
  \end{align*}
  We are going to transform some quantities $G_i$'s. We first remark using
  $\Delta_p - S_p - X_p \Sigma + \Sigma Y_p = 0$ that
  \begin{align*}
    - X_p \Sigma Y_p + \dfrac{1}{2} X_p^2 \Sigma + \frac{1}{2} \Sigma Y_p^2 &
    =  \dfrac{1}{2} X_p (- \Sigma Y_p + X_p \Sigma) + \dfrac{1}{2} (- X_p
    \Sigma + \Sigma Y_p) Y_p\\
    & = \frac{1}{2} X_p (\Delta_p - S_p) - \frac{1}{2} (\Delta_p - S_p) Y_p
    .
  \end{align*}
  Next we remark that $Q_{p, 1} = \dfrac{1}{2} (C_{p - 1} X_p + X_p C_{p - 1})
  + Q_{p, 2}$ and $R_{p, 1} = \dfrac{1}{2} (D_{p - 1} Y_p + Y_p D_{p - 1}) +
  R_{p, 2}$. On the other hand we have : $\dis \sum_{k = 1}^{p - 1} (\Delta_k -
  S_k) - C_{p - 1} \Sigma + \Sigma D_{p - 1} = 0$. Hence we can write $G_2$ as
  \begin{align*}
    G_2 & = Q_{p, 1} \Sigma + \Sigma R_{p, 1} - C_{p - 1} \Sigma Y_p - X_p
    \Sigma D_{p - 1} - X_p \Sigma Y_p + \dfrac{1}{2} X_p^2 \Sigma +
    \frac{1}{2} \Sigma Y_p^2\\
    & = Q_{p, 2} \Sigma + \Sigma R_{p, 2} + \dfrac{1}{2} C_{p - 1} (X_p
    \Sigma - \Sigma Y_p) + \frac{1}{2} (- X_p \Sigma + \Sigma Y_p) D_{p - 1}\\
    &  \qquad + \frac{1}{2} X_p (- \Sigma D_{p - 1} + C_{p - 1} \Sigma) +
    \frac{1}{2} (- C_{p - 1} \Sigma + \Sigma D_{p - 1}) Y_p\\
    &  \qquad + \frac{1}{2} X_p (\Delta_p - S_p) - \frac{1}{2} (\Delta_p - S_p)
    Y_p\\
    & = Q_{p, 2} \Sigma + \Sigma R_{p, 2} + \dfrac{1}{2} C_{p - 1}
    (\Delta_p - S_p) - \frac{1}{2} (\Delta_p - S_p) D_{p - 1}\\
    &  \qquad + \frac{1}{2} X_p \sum_{k = 1}^p (\Delta_k - S_k) + \frac{1}{2}
    \sum_{k = 1}^p (S_k - \Delta_k) Y_p .
  \end{align*}
  Next, by proceeding as above and using $e_p = \dfrac{1}{2} u^2 + d_p (u)$,
  we see that
  \begin{align*}
    G_3 + G_6 & = - C_{p - 1} \Sigma R_{p, 1} + Q_{p, 1} \Sigma D_{p - 1} -
    X_p \Sigma R_{p, 1} + Q_{p, 1} \Sigma Y_p\\
    &  \qquad - C_{p - 1} \Sigma e_p (Y_p) + e_p (X_p) \Sigma D_{p - 1} + e_p
    (C_{p - 1}) \Sigma Y_p - X_p \Sigma e_p (D_{p - 1})\\
    & = \frac{1}{2} (- C_{p - 1} \Sigma Y_p D_{p - 1} + C_{p - 1} X_p
    \Sigma D_{p - 1} - X_p \Sigma D_{p - 1} Y_p + X_p C_{p - 1} \Sigma Y_p)\\
    &  \qquad + \frac{1}{2} (C_{p - 1} X_p \Sigma Y_p + X_p C_{p - 1} \Sigma D_{p
    - 1} - C_{p - 1} \Sigma D_{p - 1} Y_p - X_p \Sigma Y_p D_{p - 1})\\
    &  \qquad + \frac{1}{2} (- C_{p - 1} \Sigma Y_p^2 - X_p \Sigma D_{p - 1}^2 +
    C_{p - 1}^2 \Sigma Y_p + X_p^2 \Sigma D_{p - 1})\\
    &  \qquad - C_{p - 1} \Sigma R_{p, 2} + Q_{p, 2} \Sigma D_{p - 1} - X_p \Sigma
    R_{p, 2} + Q_{p, 2} \Sigma Y_p\\
    &  \qquad - C_{p - 1} \Sigma d_p (Y_p) + d_p (X_p) \Sigma D_{p - 1} + d_p
    (C_{p - 1}) \Sigma Y_p - X_p \Sigma d_p (D_{p - 1}) .
  \end{align*}
  We group some terms of the previous expression :
  \begin{align*}
    - C_{p - 1} \Sigma Y_p D_{p - 1} + C_{p - 1} X_p \Sigma D_{p - 1} & =
    C_{p - 1} (\Delta_p - S_p) D_{p - 1}\\
    - X_p \Sigma D_{p - 1} Y_p + X_p C_{p - 1} \Sigma Y_p & = - X_p \sum_{k
    = 1}^{p - 1} (\Delta_k - S_k) Y_p\\
    C_{p - 1} X_p \Sigma Y_p - C_{p - 1} \Sigma Y_p^2 & = C_{p - 1}
    (\Delta_p - S_p) Y_p\\
    X_p C_{p - 1} \Sigma D_{p - 1} - X_p \Sigma D_{p - 1}^2 & = X_p \sum_{k
    = 1}^{p - 1} (\Delta_k - S_k) D_{p - 1}\\
    - C_{p - 1} \Sigma D_{p - 1} Y_p + C_{p - 1}^2 \Sigma Y_p & = C_{p - 1}
    \sum_{k = 1}^{p - 1} (\Delta_k - S_k) Y_p\\
    - X_p \Sigma Y_p D_{p - 1} + X_p^2 \Sigma D_{p - 1} & = X_p (\Delta_p -
    S_p) D_{p - 1}
  \end{align*}
  In this way we get
  \begin{align*}
    G_3 + G_6 & = \frac{1}{2} C_{p - 1} (\Delta_p - S_p) D_{p - 1} -
    \frac{1}{2} X_p \sum_{k = 1}^{p - 1} (\Delta_k - S_k) Y_p + \frac{1}{2}
    C_{p - 1} (\Delta_p - S_p) Y_p\\
    &  \qquad + \frac{1}{2} X_p \sum_{k = 1}^{p - 1} (\Delta_k - S_k) D_{p - 1} +
    \frac{1}{2} C_{p - 1} \sum_{k = 1}^{p - 1} (\Delta_k - S_k) Y_p 
    \\
    &\qquad+
    \frac{1}{2} X_p (\Delta_p - S_p) D_{p - 1} + G_6\\
    & = \frac{1}{2} C_{p - 1} (\Delta_p - S_p) D_{p - 1} - \frac{1}{2} X_p
    \sum_{k = 1}^{p - 1} (\Delta_k - S_k) Y_p\\
    &  \qquad + \frac{1}{2} X_p \sum_{k = 1}^p (\Delta_k - S_k) D_{p - 1} +
    \frac{1}{2} C_{p - 1} \sum_{k = 1}^p (S_k - \Delta_k) Y_p + G_6
  \end{align*}
  with
  \begin{align*}
    G_6 & = - C_{p - 1} \Sigma R_{p, 2} + Q_{p, 2} \Sigma D_{p - 1} - X_p
    \Sigma R_{p, 2} + Q_{p, 2} \Sigma Y_p\\
    &  \qquad - C_{p - 1} \Sigma d_p (Y_p) + d_p (X_p) \Sigma D_{p - 1} + d_p
    (C_{p - 1}) \Sigma Y_p - X_p \Sigma d_p (D_{p - 1}) .
  \end{align*}
  We now see that
  \begin{align*}
    G_4 & = - X_p \Sigma e_p (Y_p) + e_p (X_p) \Sigma Y_p\\
    & = \frac{1}{2} (- X_p \Sigma Y_p^2 + X_p^2 \Sigma Y_p) - X_p \Sigma
    d_p (Y_p) + d_p (X_p) \Sigma Y_p\\
    & = \frac{1}{2} X_p (S_p - \Delta_p) Y_p - X_p \Sigma d_p (Y_p) + d_p
    (X_p) \Sigma Y_p .
  \end{align*}
  Finally $G_5$ remains unchanged.
\end{proof}
\section{Useful Lemmas and Propositions}\label{sec-uLP}

The notations are those of the introduction and sections 6, 7 and 8. We also
denote :
\begin{enumerate}
  \item  $e_p (u) = \extend{\dis\sum_{k = 1}^{max \{k \hspace{0.17em} :
  \hspace{0.17em} 2 k \leqslant p\}}} c_k u^{2 k}$\quad where $c_k = (- 1)^{k
  + 1} \dfrac{(2 k) !}{4^k  (k!)^2 (2 k - 1)}$.
  
  \item $c_p (u) = u + e_p (u) = u + \dfrac{1}{2} u^2 + d_p (u)$ with \ $d_p
  (u) = \extend{\dis\sum_{k = 2}^{max \{k \hspace{0.17em} : \hspace{0.17em} 2 k
  \leqslant p\}}} c_k u^{2 k}$ .
  
  \item $L_{i_1, i_2} (X, Y)$ is the sum of monomials which the degree of each
  monomial with respect $X$ is $i_1$ (respectively with respect $Y$ is $i_2$
  ).
\end{enumerate}
\begin{lemma}
  \label{SDeltak}Let for $1 \leqslant k \leqslant i$, $\| \Delta_k \|
  \leqslant q^{k - 1} \varepsilon_1^k$ with $q \varepsilon_1 \leqslant \theta
  < 1$. Then $| | \dis\sum_{k = 1}^i \Delta_i | | \leqslant \eta \varepsilon_1$
  with $\eta = \dfrac{1}{1 - \theta}$.
\end{lemma}

\begin{proof}
  The proof is obvious.
\end{proof}

\begin{lemma}
  \label{epbar}Let us denote $a_1 (u) = \dfrac{1}{1 + \sqrt{1 - u^2}}$ and \
  $a_2 (u) = \dfrac{a_1 (u) - 1 / 2}{u^2}$. We have
  \begin{enumerate}
    \item $| e_p (u)  | = \extend{\dis\sum_{k = 1}^{max \{k \hspace{0.17em} :
    \hspace{0.17em} 2 k \leqslant p\}}} | c_k | u^{2 k} \leqslant u^2 a_1
    (u)$.
    
    \item $| d_p (u)  | = \dis\sum_{k = 2}^{max \{k \hspace{0.17em} :
    \hspace{0.17em} 2 k \leqslant p\}} | c_k | u^{2 k} \leqslant u^4 a_2 (u) =
    u^2 \left( a_1 (u) - \dfrac{1}{2} \right)$.
  \end{enumerate}
  
\end{lemma}

\begin{proof}
  It follows from classical Taylor series expansion.
\end{proof}

\begin{lemma}
  \label{ai} Let $b_1 (u) = \dfrac{u^2 a_1 (u)^2}{\sqrt{1 - u^2}} + 2 a_1 (u)$
  and $b_2 (u) = \dfrac{a_1 (u)^2}{\sqrt{1 - u^2}} + 2 a_2 (u)$. We have
  \begin{align*}
    (x + y)^{2 i} a_i (x + y) - x^{2 i} a_i (x) - y^{2 i} a_i (y) & \leqslant
    & b_i (x + y) x y (x + y)^{2 i - 2} .
  \end{align*}
\end{lemma}

\begin{proof}
  To prove the case $i = 1$ we write
  \begin{align*}
&    (x + y)^2 a_1 (x + y) - x^2 a_1 (x) - y^2 a_1 (y)\\
     & = x^2 (a_1 (x + y) -
    a_1 (x)) + y^2 (a_1 (x + y) - a_1 (y))
    + 2 x y a_1 (x + y)\\
    & = \left( \frac{(2 x + y) x a_1 (x)}{\sqrt{1 - x^2} + \sqrt{1 - (x +
    y)^2}} + \frac{(2 y + x) y a_1 (y)}{\sqrt{1 - y^2} + \sqrt{1 - (x +
    y)^2}} + 2 \right) x y a_1 (x + y)
  \end{align*}
  Using $y \leqslant x$, $a_1 (y) \leqslant a_1 (x)$ and $\sqrt{1 - x^2},
  \sqrt{1 - y^2} \leqslant \sqrt{1 - (x + y)^2}$ we get
  \begin{align*}
    (x + y)^2 a_1 (x + y) - x^2 a_1 (x) - y^2 a_1 (y) & \leqslant \left(
    \frac{(x + y )^2 a_1 (x + y)}{\sqrt{1 - (x + y)^2}} + 2 \right) x y a_1 (x
    + y)\\
    & = b_1 (x + y) x y.
  \end{align*}
  To prove the case $i = 2$ we write from definition of $a_2 (u)$ :
  \begin{align*}
    (x + y)^4 a_2 (x + y) - x^4 a_2 (x) - y^4 a_2 (y) & = (x + y)^2 a_1 (x +
    y) - x^2 a_1 (x) - y^2 a_1 (y) - x y\\
    & \leqslant \left( \frac{(x + y )^2 a_1 (x + y)^2}{\sqrt{1 - (x +
    y)^2}} + 2 a_1 (x + y) - 1 \right) x y\\
    & \leqslant \left( \frac{a_1 (x + y)^2}{\sqrt{1 - (x + y)^2}} + 2 a_2
    (x + y) \right) x y (x + y )^2\\
    & \leqslant b_2 (x + y) x y (x + y )^2 .
  \end{align*}
  We are done.
\end{proof}

\begin{lemma}
  \label{epCp1} Let $C_{p - 1} = X_1 + \cdots + X_{p - 1}$. Let us suppose $q
  \geqslant 2 (\theta + \eta) \kappa$, $v = q \varepsilon_1 \le \theta < 1$,
  $\eta = \dfrac{1}{1 - \theta}$ and $\| X_k \| \leqslant \dfrac{\kappa}{q}
  v^k$, $1 \le k \le p - 1$. \ Then we have
  \begin{enumerate}
    \item $\| C_{p - 1} \| \leqslant \eta \kappa \varepsilon_1$.
    \\€
    \item $\| e_p (C_{p - 1}) \| \leqslant a_1  (\eta \kappa \varepsilon_1)
    \eta^2 \kappa^2 \varepsilon_1^2$.
    \\
    \item $| | e_p (X_p) | | \leqslant a_1  (\theta \kappa \varepsilon_1)
    \kappa^2 q^{2 (p - 1)} \varepsilon_1^{2 p} .$
  \end{enumerate}
\end{lemma}

\begin{proof}
  We have
  
  \begin{align*}
    \| C_{p - 1} \| & \leqslant \sum_{k = 1}^{p - 1} \| X_k \| \leqslant
    \sum_{k = 1}^{p - 1} \kappa q^{k - 1} \varepsilon_1^k \leqslant \frac{1}{1
    - v} \kappa \varepsilon_1 \leqslant \eta \kappa \varepsilon_1 .
  \end{align*}
  
  From Lemma ~\ref{epbar} we know that $| e_p (u)  | \leqslant u^2 a_1 (u)$.
  Since $q \geqslant 2 (\theta + \eta) \kappa$ and $\varepsilon_1 \leqslant
  \dfrac{\theta}{q}$ it follows that $\eta \kappa \varepsilon_1 \leqslant
  \dfrac{\eta \theta}{2 (\eta + \theta)} = \dfrac{\theta}{2 (1 + \theta -
  \theta^2)} $, we can see the quantity $a_1  (\eta \kappa \varepsilon_1)$ is
  well defined when $\eta \kappa \varepsilon_1 \leqslant 1$. That is to say \
  $\dfrac{\theta}{2 (1 + \theta - \theta^2)} \leqslant 1$. This is the case
  since $\theta < 1$. It follows
  
  \begin{align*}
    & \| e_p (C_{p - 1}) \| \leqslant a_1  (\eta \kappa \varepsilon_1) 
    \hspace{0.17em} (\eta \kappa \varepsilon_1)^2 .
  \end{align*}
  
  We now bound $\| e_p (X_p) \|$. Always from Lemma ~\ref{epbar} we have
  \begin{align*}
    \| e_p (X_p) \| & \leqslant a_1 (\kappa q^{p - 1} \varepsilon_1^p)
    (\kappa q^{p - 1} \varepsilon_1^p)^2\\
    & \leqslant a_1  (\theta \kappa \varepsilon_1) \kappa^2 q^{2 (p - 1)}
    \varepsilon_1^{2 p} \quad \tmop{since} \quad q \varepsilon_1 \leqslant
    \theta < 1.
  \end{align*}
  We are done.
\end{proof}

\begin{lemma}
  \label{dpCp1}Let us suppose $2 (\theta + \eta) \kappa \le q$, $v = q
  \varepsilon_1 \le \theta$ and $\| X_k \| \leqslant \dfrac{\kappa}{q} v^k$,
  $1 \le k \le p - 1$. \ Then we have
  \[ \| d_p (C_{p - 1}) \| \leqslant \hspace{0.17em} a_2  (\eta \kappa
     \varepsilon_1) \eta^4 \kappa^4 \varepsilon_1^4 \]
  and
  \begin{align*}
    | | d_p (X_p) | | & \leqslant a_2  (\theta \kappa \varepsilon_1)
    \kappa^4 q^{4 (p - 1)} \varepsilon_1^{4 p} .
  \end{align*}
\end{lemma}

\begin{proof}
  The proof is like to that of Lemma \ref{epCp1}.
\end{proof}

\begin{lemma}
  \label{Thetap1}Let us suppose $2 (\theta + \eta) \kappa \le q$, $v = q
  \varepsilon_1 \le \theta$ and $\|X_k \|, \|Y_k \| \leqslant
  \dfrac{\kappa}{q} v^k$, $1 \le k \le p$. Then we have
  \[ \| \Theta_{p - 1} \| \leqslant (1 + \eta \kappa \varepsilon_1 a_1 (\eta
     \kappa \varepsilon_1)) \eta \kappa \varepsilon_1 . \]
\end{lemma}

\begin{proof}
  We have $\| \Theta_{p - 1} \| \leqslant \| C_{p - 1} \| + \| e_p (C_{p - 1})
  \|$. Using $\| C_{p - 1} \| \leqslant \le \eta \kappa \varepsilon_1$ and \
  Lemma \ref{epCp1} the conclusion follows.
\end{proof}

\begin{lemma}
  \label{Rp} Let us suppose $2 (\theta + \eta) \kappa \le q$, $v = q
  \varepsilon_1 \le \theta$ and $\| X_k \| \leqslant \dfrac{\kappa}{q} v^k$,
  $1 \le k \le p$. Let
  \begin{align*}
    Q_{p, i} & = \sum_{k = i}^{\max (k : 2 k \leqslant p)} c_k
    \sum_{\tmscript{\begin{array}{c}
      i_1 + i_2 = 2 k\\
      i_1, i  > 0
    \end{array}}} L_{i_1, i_2} (C_{p - 1}, X_p), \hspace{3em} i = 1, 2.
  \end{align*}
  We have
  \[ \| Q_{p, i} \| \leqslant b_i  (\eta \kappa \varepsilon_1) \eta^{2 i - 1}
     \kappa^{2 i} q^{p - 1} \varepsilon_1^{p + 2 i - 1} \quad i = 1, 2. \]
\end{lemma}

\begin{proof}
  Let $\| C_{p - 1} \| \leqslant x$ and $\| X_p \| \leqslant y$. We have using
  Lemma \ref{epbar} :
  \begin{align*}
    \| Q_{p, i} \| & \leqslant \sum_{k = i}^{\max (k : 2 k \leqslant p)} |
    c_k |  \sum_{\ontop{i_1 + i_2 = 2 k}{i_1 > 0, i_2 > 0}} \frac{(2 k) !}{i_1
    !i_2 !} x^{i_1} y^{i_2}\\
    & \leqslant  \sum_{k \geqslant i}  | c_k | ( (x + y)^{2 k} - x^{2 k} -
    y^{2 k})\\
    & \leqslant (x + y)^{2 i} a_i (x + y) - x^{2 i} a_i (x) - y^{2 i} a_i
    (y) .
  \end{align*}
  We apply the Lemma \ref{ai} with the bounds $y \leqslant \dfrac{\kappa}{q}
  v^p \leqslant \kappa q^{p - 1} \varepsilon_1^p$ and $x \leqslant x + y
  \leqslant \dfrac{\kappa}{q}  \dfrac{v}{1 - v} \leqslant \eta \kappa
  \varepsilon_1$ . We then get :
  \begin{align*}
    \| Q_{p, 1} \| & \leqslant b_i  (\eta \kappa \varepsilon_1) {\eta^{2 i -
    1}}  \kappa^{2 i} q^{p - 1} \varepsilon_1^{p + 2 i - 1} .\\
    &  \qquad 
  \end{align*}
  The result follows.
\end{proof}

\begin{lemma}
  \label{dpXp}Let $\| X_p \|, \| Y_p \| \leqslant \kappa q^{p - 1}
  \varepsilon_1^p$, \ $2 (\theta + \eta) \kappa \le q$ and $q \varepsilon_1
  \leqslant \theta < 1$. Then
  \begin{align*}
    | | - X_p \Sigma d_p (Y_p) + d_p (X_p) \Sigma Y_p | | & \leqslant
    \nobracket 2 K a_2 (\theta \kappa \varepsilon_1)) \kappa^5 q^{5 (p - 1)}
    \varepsilon_1^{5 p} .
  \end{align*}
\end{lemma}

\begin{proof}
  Let $Z_p \assign - X_p \Sigma d_p (Y_p) + d_p (X_p) \Sigma Y_p$. Then from
  Lemma \ref{dpCp1} we deduce
  \begin{align*}
    \| \nobracket Z_p | | & \leqslant 2 K a_2  (\theta \kappa \varepsilon_1)
    \kappa^4 q^{5 (p - 1)} \varepsilon_1^{5 p} .
  \end{align*}
  We are done.
\end{proof}

\begin{lemma}
  \label{cpu}For $| u | < 1$ we have
  \[ | (1 + c_p (- u)) (1 + c_p (u)) - 1 | \leqslant \left( 2 \sqrt{1 + u^2} +
     a_1 (u) u^{p + 1} \right) a_1 (u) u^{p + \delta} \]
  where $\delta = 1$ if p is odd and $\delta = 2$ if p is even.
\end{lemma}

\begin{proof}
  Remember that $e (u) = \sqrt{1 + u^2} + u - 1$ and $e (u) = c_p (u) + r_p
  (u)$. Since $(1 + e (u))  (1 + e (- u)) = 1$ and $r_p (u) = r_p  (- u)$ it
  follows
  
  \begin{align*}
    (1 + c_p (- u))  (1 + c_p (u)) - 1 & = (1 + e (- u) - r_p (- u))  (1 + e
    (u) - r_p (u)) - 1\\
    & = (1 + e (- u))  (1 + e (u)) - 1\\
    & \quad - (1 + e (- u)) r_p (u) - (1 + e (u)) r_p (u) + r_p (u)^2\\
    & = - (2 + e (u) + e (- u) - r_p (u))  \hspace{0.17em} r_p (u)\\
    & = - \left( 2 \sqrt{1 + u^2} - r_p (u) \right)  \hspace{0.17em} r_p (u)
  \end{align*}
  
  We have
  
  \begin{align*}
    | r_p (u) | & \leqslant \sum_{i > \max \{ k : 2 k \leqslant p \}} | c_{p,
    i} | u^{2 i} =\\
    & \leqslant \frac{1}{1 + \sqrt{1 - u^2}} u^{p + \delta} = a_1 (u) u^{p +
    \delta}
  \end{align*}
  
  where $\delta = 1$ if $p$ is odd and $\delta = 2$ if $p$ is even. We deduce
  that
  
  \begin{align*}
    | (1 + c_p (- u)) (1 + c_p (u)) - 1 | & \leqslant \left( 2 \sqrt{1 + u^2}
    + a_1 (u) u^{p + \delta} \right) a_1 (u) u^{p + \delta} .
  \end{align*}
  
  We are done.
\end{proof}

\begin{lemma}
  \label{SiS0}For $i \geqslant 0$, we have
  \begin{align*}
    s_i \assign \sum_{k = 0}^{i - 1} 2^{- (p + 1)^k + 1} & \leqslant 2 -
    2^{2 - (p + 1)^i} .
  \end{align*}
\end{lemma}

\begin{proof}
  We proceed by induction. The assertion holds for $i = 0$. By assuming for
  $i$ let us prove it for $i + 1$. We have
  \begin{align*}
    s_{i + 1} & \leqslant 2 - 2^{2 - (p + 1)^i} + 2^{- (p + 1)^i + 1}
    \leqslant 2 - 2^{2 - (p + 1)^i} (1 - 2^{- 1}) = 2 - 2^{2 - (p + 1)^i -
    1}\\
    & \leqslant 2 - 2^{2 - (p + 1)^{i + 1}} \tmop{since} \quad (p + 1)^i +
    1 \leqslant 2 (p + 1)^i \leqslant (p + 1)^{i + 1} .
  \end{align*}
  We are done.
\end{proof}
\section{Proof of Davies-Smith \ Theorem
\ref{th-DS-revisited}}\label{sec-davies}

Let us denote \ $\Delta_1 = U^{\ast} \Sigma V - \Sigma$ and $\Delta_2 =
(I_{\ell} + \Theta_1^{\ast})  (\Delta_1 + \Sigma)  (I_q + \Psi_1) - \Sigma -
S_1$ with $\Theta_1 = X_1 + X_{1 }^2 / 2$ and $\Psi_1 = Y_1 + Y_1^2 / 2$. From
the definition of the map $\tmop{DS}$ we have $U_1 = U (I_{\ell} + X_1 + X_2 +
X_1^2 / 2)$, $V_1 = V (I_q + Y_1 + Y_2 + Y_1^2 / 2),$ $\Sigma_1 = \Sigma + S_1
+ S_2$ where for $i = 1, 2,$ one has $S_i = \tmop{diag} (\Delta_i)$ and the
$X_i$'s are skew Hermitian matrices be such that $\Delta_i - S_i - X_i \Sigma
+ \Sigma Y_i = 0$. The goal is to \ bound the norm of \ $\Delta_3 \assign
U_1^{\ast} M V_1 - \Sigma_1 = (I_{\ell} + \Theta_1^{\ast} - X_2)  (\Delta_1 +
\Sigma)  (I_q + \Psi_1 + Y_2) - \Sigma - S_1 - S_2$. We first expand
$\Delta_2$ and as in the proof of Proposition \ref{prop-Deltap1-bnd-p=2} we
have $\| \Delta_2 \| \leqslant q_1 \varepsilon_1^2$ where
\begin{align}
  q_1 & = 2 \kappa + 2 \kappa^2 \varepsilon_1 + \frac{5}{4} \kappa^4 K
  \varepsilon_1^2 + \frac{1}{4} \kappa^4 \varepsilon_1^3,  \label{tau2DSp=2}
\end{align}
and $q_1 \varepsilon_1 \leqslant \tau_1 \varepsilon$ with $\tau_1 = 2 + 2
\varepsilon  + \frac{5}{4} \varepsilon^2 + \frac{1}{4} \varepsilon^3$. We now
expand $\Delta_3$ to get :
\begin{align}
  \Delta_3 & = (I_{\ell} + \Theta_1^{\ast} - X_2)  (\Delta_1 + \Sigma)  (I_n
  + \Psi_1 + Y_2) - \Sigma - S_1 - S_2 \nonumber\\
  & = (I_{\ell} + \Theta_1^{\ast})  (\Delta_1 + \Sigma)  (I_n + \Psi_1) -
  \Sigma - S_1 - S_2 \nonumber\\
  &  \qquad + (I_{\ell} + \Theta_1^{\ast})  (\Delta_1 + \Sigma) Y_2 - X_2 
  (\Delta_1 + \Sigma)  (I_n + \Psi_1) - X_2  (\Delta_1 + \Sigma) Y_2 
  \label{Delta3-expand}
\end{align}
We know that
\[ (I_{\ell} + \Theta_1^{\ast})  (\Delta_1 + \Sigma)  (I_n + \Psi_1) - \Sigma
   - S_1 - S_2 = \Delta_2 - S_2 = X_2 \Sigma - \Sigma Y_2 . \]
Plugging the previous relation in $\left( \ref{Delta3-expand} \right)$ we find
\begin{align}
  \Delta_3 & = - X_2 \Delta_1 + \Delta_1 Y_2 - X_2 \Delta_1 Y_2 +
  \Theta_1^{\ast}  (\Delta_1 + \Sigma) Y_2 - X_2 (\Delta_1 + \Sigma) \Psi_1 -
  X_2 \Sigma Y_2  \label{L1DS-davies}
\end{align}
We are going to prove $\| \Delta_3 \| \leqslant q_1 q_2 \varepsilon_1^3$ where
$q_2$ is defined below in $\left( \ref{tau3DS} \right)$. To do that we will
use the bounds
\begin{enumerate}
  \item $\| \Delta_2 \| \leqslant q_1 \varepsilon_1^2$ and $\| X_2 \|, \| Y_2
  \| \leqslant \kappa q_1 \varepsilon_1^2 .$
  
  \item $\| \Theta_1 \|, \| \Psi_1 \| \leqslant \left( 1 + \dfrac{1}{2} \kappa
  \varepsilon_1 \right) \kappa \varepsilon_1$.
\end{enumerate}
Considering the bounds of the norms of \ matrices given in
$(\nobracket$\ref{L1DS-davies}), we \ get $| | \Delta_3 | | \leqslant q_3 q_1
\varepsilon_1^3$ where
\begin{align*}
  q_3 & = 2 \kappa (K \kappa + 1)  + (K \kappa + 2 + K q_1) \kappa^2
  \varepsilon_1 + (\kappa + q_1) \kappa^2 \varepsilon_1^2 .
\end{align*}
A straighforward calculation shows that if $\varepsilon_1 \leqslant
\dfrac{\varepsilon}{\kappa^{5 / 4} {K^{2 / 5}} }$ then
\begin{align}
  | | \Delta_3 | | \leqslant q_3 q_1 \varepsilon^3_1 & \leqslant \tau_3
  \tau_1 \varepsilon^3  \label{Delta3-bound-DS-davies}
\end{align}
where
\begin{align*}
  \tau_3 & = 4 + (3 + \tau_1) \varepsilon + (1 + \tau_1) \varepsilon^2 .
\end{align*}
A straightforward computation shows that for all $\varepsilon \leqslant 0.1$
we have
\begin{align*}
  \tau_3 \tau_1 & \leqslant 8 + 18 \varepsilon + 28 \varepsilon^2 .
\end{align*}
We finally get
\[ \kappa^{5 / 4} K^{2 / 5}  \| \Delta_3 \| \leqslant (8 + 18 \varepsilon + 33
   \varepsilon^2) \varepsilon ^3 . \]
Then the part 1 of \ Theorem \ref{th-DS-revisited} is proved.

We use the proof of \ Proposition \ref{prop-Deltap1-bnd-p=2} to proof the part
2 of Theorem. We have
\begin{align*}
  \| \bar{U}_1^{\ast} M \bar{V}_1 - \bar{\Sigma}_1 \| & \leqslant q_2 q_1
  \varepsilon_1^3
\end{align*}
where $q_1$ is defined in $\left( \ref{bnd-delta2-q1} \right)$ and $q_2$ in
$\left( \ref{tau3DS} \right)$. A straightforward calculation shows that if
$\varepsilon_1 \leqslant \dfrac{\varepsilon}{\kappa^{6 / 5} {K^{3 / 10}} }$
then
\begin{align}
  \| \bar{U}_1^{\ast} M \bar{V}_1 - \bar{\Sigma}_1 \| & \leqslant q_2 q_1
  \varepsilon_1^3 \leqslant \tau_2 \tau_1 \varepsilon^3 
  \label{bnd-delta-davies-revisited}
\end{align}
where $\tau = \tau_1 \tau_2$ given in $\left( \ref{taup=2} \right)$. Moreover
$\tau_2 \tau_1 \leqslant 6 + 21 \varepsilon + 54 \varepsilon^2$ for
$\varepsilon \leqslant 0.1.$ This proves te part 2. The Theorem
holds.{\hspace{5em}}$\Box$
\section{Application in the clusters case}\label{sec-clusters}

\subsection{Definiton of Clusters and first properies}

We use the Fortran or Matlab notation for submatrices, i.e., $A_{i : j, k :
l}$ is the submatrix of $A$ with lines and columns between the subscripts $i,
j$ and $k, l$respectively. We consider $e$ integers $q_i$'s such that $\dis\sum_{i
= 1}^e q_i = q$. We also associate the integers $\ell_i$, $1 \leqslant i
\leqslant e$, defined by $\dis \ell_i = 1 + \sum_{j = 1}^{i - 1} q_j$ The first
goal is to precise the notion of cluster of singular values.

\begin{definition}
  \label{def-Bni-eps} Let $e$ integers $q_i$'s such that $\dis\sum_{i = 1}^e q_i =
  q$. We associate the integers $\ell_i$, $1 \leqslant i \leqslant e$, defined
  by $\ell_i = 1 + \dis\sum_{j = 1}^{i - 1} q_j$. From \ \ $\Delta \in
  \mathbb{C}^{\ell \times q}$ with $\ell \geqslant q$, we consider its
  sub-matrices $\Delta_i {\assign \Delta_{\ell_i : \ell_{i + 1} - 1, \ell_i :
  \ell_{i + 1} - 1}}  \in \mathbb{C}^{q_i \times q_i}$, $1 \leqslant i
  \leqslant e$. We \ define the matrix
  \begin{align*}
    \tmop{Diag}_{q_1 \cdots q_e} (\Delta) & = \left(\begin{array}{ccc}
      \Delta_1 & 0 & 0\\
      0 & \ddots & 0\\
      0 & 0 & \Delta_e\\
      & 0 & 
    \end{array}\right)
  \end{align*}
  We name by $\mathbb{D}^{\ell \times q}_{q_1, \ldots, q_e}$ the set of these
  matrices. 
\end{definition}

\begin{definition}
  \label{def-cluster}Let integers $q_i$'s and $\ell_i$'s be as in Definition
  \ref{def-Bni-eps}. Let $\delta \geqslant 0$ and define the set
  ${\mathbb{D}_{q_1 \ldots q_e}^{\ell \times q}}  (\delta)$ of the matrices
  whose diagonal $\Sigma = \tmop{diag} (\sigma_1, \cdots, \sigma_q) \in
  \mathbb{D}^{\ell \times q}$ satisfies
  \begin{align}
    & | \nobracket \sigma_k - \sigma_j | \nobracket  \leqslant \delta
    \qquad \ell_i \leqslant j, k \leqslant \ell_{i + 1} - 1, \quad 1
    \leqslant i \leqslant e   \label{def-cluster-eq1}\\
    &| \nobracket {\sigma_j}  - \sigma_l | \nobracket  >  \delta,
    \qquad \ell_i \leqslant j \leqslant \ell_{i + 1} - 1, \quad \ell_k
    \leqslant l \leqslant \ell_{k + 1} - 1, \quad 1 \leqslant i < k \leqslant
    e  \label{def-cluster-eq2}
  \end{align}
  We name ${\mathbb{D}_{q_{1 \ldots q_e}}^{\ell \times q}}  (\delta)$ the set
  of clusters of size $\delta$ relatively to integers $q_1, \cdots, q_e$. We
  also name by $\mu = (q_1, \ldots, q_e)$ the multiplicity of cluster
  associated to $\Sigma$.
\end{definition}

We have

\begin{proposition}
  \label{prop-unicity-ni}Let $\delta \geqslant 0$ and $\Delta \in
  \mathbb{D}^{\ell \times q}_{q_1 \cdots q_e} (\delta)$. The tuple 
  $(q_1, \cdots, q_e)$ where each integer $q_i \geqslant 1$ is the only one such that
  the inequalities (\ref{def-cluster-eq1}-\ref{def-cluster-eq2}) hold.
\end{proposition}

\begin{proof}
  Let us suppose there exists two tuples $(m_1, \cdots, m_d)$ and $(q_1,
  \cdots, q_e)$  such that the inequalities
  (\ref{def-cluster-eq1}-\ref{def-cluster-eq2}) hold for the diagonal matrix
  $\Sigma  = \tmop{diag} (\sigma_1, \ldots, \sigma_q)$. Let us suppose for
  instance $m_1 < q_1$. Then we first have from the inequality
  (\ref{def-cluster-eq2}) : $\dis| \nobracket \sigma_{m_1} - \sigma_{m_1 + 1} |
  \nobracket > \delta$. In the other hand, since $m_1 < q_1$ we can write from
  the inequality (\ref{def-cluster-eq1}) $\dis | \nobracket \sigma_{m_1} -
  \sigma_{m_1 + 1} | \nobracket \leqslant \delta$. This is not possible and
  the proposition holds. 
\end{proof}

\subsection{Solving $\Delta - S - X \Sigma  {+ \Sigma }  Y = 0$ in the
clusters case}

We state without proof the result that is generalizes the Proposition
\ref{xij=0}.

\begin{proposition}
  \label{prop-diag-cluster} Let $\Sigma \in \mathbb{D}_{q_1 \ldots q_e}^{\ell
  \times q} (\delta)$ and $\Delta = (\delta_{i, j}) \in \mathbb{C}^{\ell
  \times q} \underline{}$. \ Consider the matrix $S \in \mathbb{D}^{\ell
  \times q}_{q_1 \ldots q_e}$ and the two skew Hermitian matrices $\smash{X} =
  (x_{i, j}) \in \mathbb{C}^{\ell \times \ell}$ and $Y = (y_{i, j}) \in
  \mathbb{C}^{q \times q}$ that are defined by the following formulas:
  \begin{enumerate}
    \item The matrix $S$ is defined by
    \begin{align}
      S & = \tmop{Diag}_{q_1 \cdots q_e} (\Delta) \in \mathbb{D}^{\ell
      \times q}_{q_1 \ldots q_e}  \label{pert-it1-cluster}
    \end{align}
    \item 
    \begin{align}
      \tmop{Diag}_{q_1 \cdots q_e} (X) & = 0  \label{eq-DiagX}\\
      \tmop{Diag}_{q_1 \cdots q_e} (Y) & = 0  \label{eq-DiagY}
    \end{align}
    \item For $1 \leqslant i < k \leqslant e$, \ $1 \leqslant j \leqslant q_i
    - 1$, \ and \ $1 \leqslant l \leqslant q_k - 1$ \ we take
    \begin{align}
      x_{\ell_i + j, \ell_k + l} & = \frac{1}{2}  \left(\dis
      \frac{\delta_{\ell_i + j, \ell_k + l} + \overline{\delta_{\ell_k + l,
      \ell_i + j}}}{\sigma_{\ell_k + l} - \sigma_{\ell_i + j}} +
      \frac{\delta_{\ell_i + j, \ell_k + l} - \overline{\delta_{\ell_k + l,
      \ell_i + j}}}{\sigma_{\ell_k + l} + \sigma_{\ell_i + j}} \right) 
      \label{pert-it4-cluster}\\
      y_{\ell_i + j, \ell_k + l} & = \frac{1}{2}  \left(\dis
      \frac{\delta_{\ell_i + j, \ell_k + l} + \overline{\delta_{\ell_k + l,
      \ell_i + j}}}{\sigma_{\ell_k + l} - \sigma_{\ell_i + j}} -
      \frac{\delta_{\ell_i + j, \ell_k + l} - \overline{\delta_{\ell_k + l,
      \ell_i + j}}}{\sigma_{\ell_k + l} + \sigma_{\ell_i + j}} \right) 
      \label{pert-it5-cluster}
    \end{align}
    \item For $q + 1 \leqslant i \leqslant \ell$ and $1 \leqslant j \leqslant
    q$, we take
    \begin{align}
      x_{i, j} & = \frac{1}{\sigma_j} \delta_{i, j} .  \label{pert-it8-cluster}
    \end{align}
    \item For $q + 1 \leqslant i \leqslant \ell$ and $q + 1 \leqslant j
    \leqslant \ell$, we take
    \begin{align}
      x_{i, j} & = 0.  \label{pert-it9-cluster}
    \end{align}
  \end{enumerate}
  Then we have
  \begin{align}
    \Delta - S - X \Sigma  + \Sigma Y   & = 0.  \label{eq-diag-cluster}
  \end{align}
\end{proposition}

\begin{definition}
  \label{def-condition-cluster} Under the previous framework, we name condition number
  of equation $\Delta - S - X \Sigma  + \Sigma  Y = 0$ the quantity
  \begin{align}
    \kappa (\Sigma ) & = \max \left( 1, \max_{1 \leqslant i \leqslant e}
    \frac{1}{| \sigma_i |}, \max_{\tmscript{\begin{array}{c}
      1 \leqslant i < k \leqslant e\\
      | \nobracket \sigma_k - \sigma_i | \nobracket > \delta
    \end{array}}}  \left. \dfrac{1}{| \nobracket \sigma_k - \sigma_i |
    \nobracket} + \dfrac{1}{| \sigma_k + \sigma_i |} \right|  \right) 
    \label{def-kappa-cluster}
  \end{align}
\end{definition}

The analysis of error is given by the following result.

\begin{proposition}
  \label{err-cluster-diag-XY} Under the notations and assumptions of
  Proposition \ref{prop-diag-cluster}, assume that $S$, $X$ and $Y$ are
  computed using \tmtextup{(\ref{pert-it1-cluster}--\ref{pert-it9-cluster})}.
  Given $\varepsilon$ with $\| \Delta \|  \leqslant \varepsilon$, the matrices
  $X$, $Y$ and $S$ solutions of $\Delta - S - X \Sigma + \Sigma Y = 0$ satisfy
  \begin{align}
    \|S\| & \leqslant \varepsilon  \label{dSigma-bnd-cluster}\\
    \|X\|,  \| Y \| & \leqslant \kappa \varepsilon  \label{dU-bnd-cluster}
  \end{align}
\end{proposition}

\subsection{Method of order p+1 in the clusters case}

Let $p \geqslant 2$ and $\mathbb{E}^{m \times \ell, n \times q}_{q_1, \ldots,
q_e} =\mathbb{C}^{m \times \ell} \times \mathbb{C}^{n \times q} \times
\mathbb{D}^{m \times n}_{q_1, \ldots, q_e}$.We denote $E_{\ell} (U) = U^{\ast}
U - I_{\ell}$, $E_q (V) = V^{\ast} V - I_q$, $\Delta = U^{\ast} M V - \Sigma$
and we define the map $H_p$ by
\begin{align}
  (U, V, \Sigma) \in \mathbb{E}^{m \times \ell, n \times q}_{q_1, \ldots, q_e}
  & \rightarrow & H_p (U, V, \Sigma) = \left(\begin{array}{c}
    U (I_{\ell} + \Omega )  (I_{\ell} + \Theta)\\
    V (I_q + \Lambda) (I_q + \Psi)\\
    \Sigma  + S
  \end{array}\right) \in \mathbb{E}^{m \times \ell, n \times q}_{q_1, \ldots,
  q_e}  \label{map-Hpq-clusters}
\end{align}
where :
\begin{enumerate}
  \item $\Omega  = s_p (E_{\ell} (U) )$ and $\Lambda  = s_p (E_q (V) ) .$
  
  \item $S = S_1 + \cdots + S_p \in \mathbb{D}^{m \times n}_{q_1 \ldots q_l}$,
  $X = X_1 + \cdots + X_p$ and $Y = Y_1 + \cdots + Y_p$ with each $X_k$, $Y_k$
  are skew Hermitian matrices. Moreover each triplet $(S_k, X_k, Y_k)$ are
  solutions of the following linear systems :
  \begin{align*}
    \Delta _k - S_k - X_k \Sigma  + \Sigma  Y_k & = 0, \qquad 1 \leqslant k
    \leqslant p
  \end{align*}
  where the $\Delta _k$'s for $2
    \leqslant k \leqslant p + 1,$ are defined as
\begin{align}
&  \begin{array}{l}
    \Delta _1  = (I_{\ell} + \Omega) (\Delta  + \Sigma) (I_q + \Lambda)
    - \Sigma, e \quad S_1 =
    \tmop{Diag}_{q_1, \ldots, q_e} (\Delta_1) \\
    \Theta_k  = c_p (X_1 + \cdots + X_k), \quad \Psi_k = c_p (Y_1
    + \cdots + Y_k), \quad 1 \leqslant k \leqslant p, \\
    \Delta _k  = (I_{\ell} + \Theta_{k - 1}^{\ast}) (\Delta _1 + \Sigma)
    (I_q + \Psi_{k - 1}) - \Sigma -\dis \sum_{l = 1}^{k - 1} S_l,   \\
    S_k  = \tmop{Diag}_{q_1, \ldots, q_e} (\Delta_k), \quad 2
    \leqslant k \leqslant p. 
  \end{array}\label{Deltak-proof}
  \end{align}
\end{enumerate}
\subsection{Result of convergence in the clusters case}

\begin{theorem}
  \label{th-svd-cluster} If the sequence define by
  \[ (U_{i + 1}, V_{i + 1}, \Sigma_{i + 1}) = H_p (U_i, V_i, \Sigma_i), \quad
     i \geqslant 0 \]
  from $(U_0, V_0, \Sigma_0) \in \mathbb{E}^{m \times \ell, n \times q}_{q_1,
  \ldots, q_e} $verifies the asumptions of Theorem \ref{th-svd-main} then it
  converges at the order $p + 1$ to \ $(U_{\infty}, V_{\infty},
  \Sigma_{\infty}) \in \tmop{St}_{m, \ell} \times \tmop{St}_{n, q} \times
  \mathbb{D}^{m \times n}_{q_1, \ldots, q_e} $such that $U_{\infty}^{\ast} M
  V_{\infty} - \Sigma_{\infty} = 0  . $
\end{theorem}

\begin{proof}
  The proof is similar to that of Theorem \ref{th-svd-main}.
\end{proof}

\subsection{Deflation method for the SVD}

The sequence $(U_i, V_i, \Sigma_i)_{i \geqslant 0}$ of Theorem
\ref{th-svd-cluster} is not a SVD sequence since the $\Sigma_i$'s belong to
$\mathbb{D}^{m \times n}_{q_1, \ldots, q_e}$. We can use the Theorem
\ref{th-svd-main} to detect the presence of clusters of singular values.

To simplify the presentation we suppose $m = n$ in order that
\begin{align*}
  \kappa (\Sigma) & = \max \left( 1, \max_{1 \leqslant i < j \leqslant n}
  \frac{1}{| \sigma_i - \sigma_j |} + \frac{1}{| \sigma_i + \sigma_j |}
  \right) .
\end{align*}
To do that we introduce an index of deflation whose the existence is given by
the following proposition.

\begin{proposition}
  \label{prop-svd-dfl} Let us consider $(U_0, V_0, \Sigma_0) \in \mathbb{E}^{m
  \times m}_{m \times m}$ \  and $\Delta_0 = U^{\ast }_0 M V_0 - \Sigma_0$.
  Let
  \begin{align*}
    e & = \max \left( \dfrac{K^{a - 1} \| \Delta_0 \|}{u_0},
    \frac{K^a}{u_0} \| E_m (U) \|, \frac{K^a}{u_0} \| E_m (V) \| \right)^{1 /
    a}
  \end{align*}
  Let us suppose $e \leqslant 1$. Then there exists an index $q \leqslant m$
  be such that \ we can rewrite the diagonal matrix $\Sigma_0$ under the form
  $\left( \begin{array}{cc}
    \Sigma_{0, q} & \\
    & \Sigma_{0, n - q}
  \end{array} \right)$ where $\kappa (\Sigma_{0, q}) e \leqslant 1.$ Let us
  consider $U_{0, q}$ and $V_{0, q}$ the sub matrices of $U_0$ and $V_0$
  respectively corresponding to $\Sigma_{0, q} .$ Then Theorem
  \ref{th-svd-main} applies for the sequence define from $(U_{0, q}, V_{0, q},
  \Sigma_{0, q}) \in \mathbb{E}^{m \times q}_{m \times q}$ \ by $(U_{i + 1,
  q}, V_{i + 1, q}, \Sigma_{i + 1, q}) = H_p (U_{i, q}, V_{i, q}, \Sigma_{i,
  q})$, $i \geqslant 0$.
\end{proposition}

\begin{proof}
  The existence of the index $q$ is obvious since $q$ is at least equal at
  $1$. In this case $\kappa (\Sigma_{0, 1}) = 1$.
\end{proof}

\begin{definition}
  \label{def-svd-dfl} Let us consider the notations and the assumption of
  Proposition \ref{prop-svd-dfl}.We name indice of deflation of $(U_0, V_0,
  \Sigma_0)$ the maximum of indices $q$ such that $\kappa (\Sigma_{0, q}) e
  \leqslant 1$. If $q$ is the index of deflation we name $(U_{0, q}, V_{0, q},
  \Sigma_{0, q})$ a deflation of $(U_0, V_0, \Sigma_0)$
\end{definition}

To determine the index of deflation and a deflation of $(U_0, V_0, \Sigma_0)$,
we propose the following algorithm. We denote $\kappa_{i, j} = \max \left( 1,
\dfrac{1}{| \sigma_i - \sigma_j |} + \dfrac{1}{| \sigma_i + \sigma_j |}
\right)$. Following the matlab notation if $A$ is a matrix and $k$ a vector of
indices $A (:, k)$ means the matrix composed by the columns indexed by the
vector $k$. Moreover $\#k$ is the size of $k$.
 \begin{align} \textbf{Algorithm to determine the index of deflation }  
      \label{algo-dfl}
      \end{align}
      
       \tmtextbf{Input} $(U_0, V_0, \Sigma_0)$ such that $e
      \leqslant 1$
      
      \tmtextbf{Ouput} $(U_{0, q}, V_{0, q}, \Sigma_{0, q})$ a deflation of
      $(U_0, V_0, \Sigma_0)$
      \begin{enumerate}
        \item  Let $\Sigma_0 = \tmop{diag} (\sigma_{0, 1}, \ldots, \sigma_{0,
        n})$ where $\sigma_{0, 1} \geqslant \cdots \geqslant \sigma_{0, n}$
        \item $k = 1$\qquad$i = 1$
        \item while $i \leqslant m$ do
        \item {\hspace{3em}}$j = 1$
        \item {\hspace{3em}}while $i + j \leqslant n$ and $\kappa_{i, i + j} e
        > 1$ do\quad$j = j + 1$\quad end while
        \item {\hspace{3em}}if $i + j \leqslant n$ and $\kappa_{i, i + j}
        \leqslant 1$ then\quad$k = [k, i + j]$\quad end if
        \item {\hspace{3em}}$i = i + j$
        \item end while
        \item $q =\#k$
        \item $\Sigma_{0, q} = \Sigma_0 (k)$\qquad$U_{0, q} = U_0
        (k)$\qquad$V_{0, q} = V_0 (k)$
      \end{enumerate}

\begin{theorem}
  \label{th-dfl} Let $(U_0, V_0, \Sigma_0)$ that satisfies the Proposition
  \ref{prop-svd-dfl}. The algorithm \ref{algo-dfl} computes a deflation of
  $(U_0, V_0, \Sigma_0)$.
\end{theorem}

\begin{proof}
  When $k = 1$ we have $\kappa (\Sigma_0 (:, 1)) = 1$ and $\kappa (\Sigma_0
  (:, 1)) e \leqslant 1$ from assumption. The loop 3-8 of the algorithm
  consists to determine an ordered list of indices $k$ such that for all $i
  \in k$ such that $i + 1 \in k$ we have $\kappa_{i, i + 1} e \leqslant 1$.
  Hence $\kappa (\Sigma_{0, q}) e \leqslant 1$ and the Theorem follows. 
\end{proof}

\section{Numerical Experiments}\label{sec-numerical-experiments}

Our numerical experiments are done with the \tmtexttt{Julia Programming
Language} {\cite{julia19}} \ coupled with the library \tmtexttt{ArbNumerics}
of Jeffrey Sarnoff. To intialize our method we proceed in two steps
\begin{enumerate}
  \item The triplet $(U_0, V_0, \Sigma_0)$ is given by the function
  \tmtexttt{svd} of \tmtexttt{Julia} with $64$-bit of precision unless
  otherwise stated.
  
  \item From this $(U_0, V_0, \Sigma_0)$ we determine $(U_{0, q}, V_{0, q},
  \Sigma_{0, q})$ by the Algorithm \ref{algo-dfl}.
\end{enumerate}
We consider for $i \geqslant 0$ the quantities  $$\varepsilon_i = \max
((\kappa_i K_i)^a  \| E_{\ell} (U_{_i}) \|, (\kappa_i K_i)^a  \| E_q (V_i) \|,
\kappa_i^a K_i^{a - 1} \| \Delta_i \|)$$ where $a$, $u_0$ \ are defined in
Theorem \ref{th-svd-main}. All the Tables below show the behaviour of $e_i = -
\lfloor \log_2 (\varepsilon_i / u_0) \rfloor$.

The strategy of practical computations is to initialize the method with $q$
bits of precision. Next the iteration $i$ is done with $q (p + 1)^i$ bits of
precision. This setting of precision is done efficiently thanks to the library
\tmtexttt{ArbNumerics} at each iteration.

\subsection{Random matrices}

Table \ref{table1} confirms the behaviour of iterates expected by the
convergence analysis.

\begin{table}[h]
  $$\begin{array}{|c|c|c|c|c|c|c|}
    \hline
    \tmop{Iterations} / \tmop{Order} & 2 & 3 & 4 & 5 & 6 & 7\\
    \hline
    \begin{array}{c}
      0\\
      1\\
      2\\
      3
    \end{array} & \begin{array}{c}
      7\\
      18\\
      44\\
      92
    \end{array} & \begin{array}{c}
      8\\
      35\\
      112\\
      346
    \end{array} & \begin{array}{c}
      9\\
      47\\
      194\\
      787
    \end{array} & \begin{array}{c}
      8\\
      59\\
      311\\
      1571
    \end{array} & \begin{array}{c}
      8\\
      69\\
      427\\
      2580
    \end{array} & \begin{array}{c}
      8\\
      85\\
      604\\
      4353
    \end{array}\\
    \hline
  \end{array}$$
  \caption{\label{table1}}
\end{table}

\

\subsection{Cauchy matrices}
The  classical Cauchy matrix is defined by $$M = \left( \dfrac{1}{i + j}
\right)_{1 \leqslant i, j \leqslant n}.$$
 Its singular values satisfy the
inequalities $\sigma_{1 + k} \geqslant 4 \left( \exp \left( \dfrac{\pi^2}{2
\tmop{Log} (4 n)} \right)  \right)^{- 2 k} \sigma_1$ where $\sigma_1$ is the
greatest singular values {\cite{beckermann2017}}. There is a strong decrease
of singular values to $0$. The computation of a deflation by the Algorithm
\ref{algo-dfl} gives different values of $q$ for $\Sigma_{0, q}$ following the
value of $p$ . For instance with $64$-bit of precision and $n = 200$, if $p =
1$ then $q = 11$ : $\Sigma_{0, q}$ is constituted of the first ten singular
values and one among the other 190's. If $p \geqslant 2$ then $q = 15$ :
$\Sigma_{0, q}$ is constituted of the first fourteen singular values and one
among the other 185's. Table \ref{table2} gives the behaviour of iterates from
a computation of a deflation.

\begin{table}[h]
  $$\begin{array}{|c|c|c|c|c|c|c|}
    \hline
    \tmop{Iterations} / \tmop{Order} & 2 & 3 & 4 & 5 & 6 & 7\\
    \hline
    \begin{array}{c}
      0\\
      1\\
      2\\
      3
    \end{array} & \begin{array}{c}
      1\\
      9\\
      31\\
      74
    \end{array} & \begin{array}{c}
      1\\
      19\\
      67\\
      214
    \end{array} & \begin{array}{c}
      1\\
      19\\
      116\\
      503
    \end{array} & \begin{array}{c}
      1\\
      35\\
      196\\
      1003
    \end{array} & \begin{array}{c}
      1\\
      36\\
      277\\
      1724
    \end{array} & \begin{array}{c}
      1\\
      51\\
      389\\
      2757
    \end{array}\\
    \hline
  \end{array}$$
  \caption{\label{table2}}
\end{table}

Table \ref{table2bis} gives the necessary precision that we need to get the
size of Cauchy matrices as index of deflation.

\begin{table}[h]
  $$\begin{array}{|c|c|c|c|}
    \hline
    n & n \leqslant 7 & 8 \leqslant n \leqslant 14 & 15 \leqslant n\\
    \hline
    \tmop{bits} \tmop{precision} & \begin{array}{c}
      64
    \end{array} & \begin{array}{c}
      128
    \end{array} & \begin{array}{c}
      \geqslant 256
    \end{array}\\
    \hline
  \end{array}$$
  \caption{\label{table2bis}}
\end{table}
\subsection{Matrices with prescribed singular values}

Let us define $M = U \Sigma V$ where $U$ and $V$ are two unitary matrices of
size $4 n \times 4 n$ and $\Sigma = \tmop{diag} (\sigma_1, \ldots, \sigma_{4
n})$ where
\begin{align*}
  \sigma_{3 (i - 1) + j} & = 2^i \qquad 1 \leqslant i \leqslant n, \quad 1
  \leqslant j \leqslant 3,\\
  \sigma_{3 n + i} & = 2^{- i} \qquad 1 \leqslant i \leqslant n.
\end{align*}
The condition $e \leqslant 1$ of the Proposition \ref{prop-svd-dfl} holds if
$\left( \dfrac{4 \times 2^n}{3} \right)^a \varepsilon_0 \leqslant u_0$ where
$\varepsilon_0 = \max (\| \Delta_0 \|, \| E_m (U_0) \|, \| E_m (V_0) \|)$. \
Table \ref{table3} gives the quantity $- \left\lfloor \log_2  \dfrac{3^a
u_0}{4^a 2^{n a}}  \right\rfloor$ with respect $n$. For instance a C matrix of
size $100 \times 100$, \ Proposition \ref{prop-svd-dfl} applies if
$\varepsilon_0 \leqslant 2^{- 139}$ for $p \geqslant 2$ and for $p = 1$, it is
necessary to have $\varepsilon_0 \leqslant 2^{- 206}$. Hence the precision
required on $\varepsilon_0$ to get

\begin{table}[h]
  $$\begin{array}{|l|l|l|l|l|l|l|l|l|l|l|}
    \hline
    p / 4 n & 4 & 20 & 40 & 60 & 80 & 100 & 120 & 140 & 160 & 180\\
    \hline
    p = 1 & 14 & 46 & 86 & 126 & 166 & 206 & 246 & 286 & 326 & 366\\
    \hline
    p \geqslant 2 & 11 & 33 & 59 & 86 & 113 & 139 & 166 & 193 & 219 & 246\\
    \hline
  \end{array}$$
  \caption{\label{table3}}
\end{table}

a deflation is greater in the case $p = 1$ than for $p \geqslant 2$. This is
confirmed by numerical experimentation. If $p = 1$ then $n \leqslant 26$
(respectively if $p \geqslant 2$ then $n \leqslant 41$) a 64-bits precision is
enough so that Proposition \ref{prop-svd-dfl} holds. Table \ref{table4} shows
for $p = 1$ (respectively $p \geqslant 2$) the quantities $q_+ =\# \{ \sigma >
1 \}$ and $q_- \# \{ \sigma > 1 \}$ from a $\Sigma_{0, q}$ given by the
initialization. In each case of Table \ref{table4} the first number matches
for $q_+$ and the second for $q_-$. The 64-bit precision used for $p = 1$
(respectively $p \geqslant 2$) until the size $100$ (respectively $140$). For
larger sizes, \ 128-bits precision are used. The quantity $q_+$ is always
equal to $n$ which is the number of multiple singular values.

\begin{table}[h]
  $$\begin{array}{|c|c|c|c|c|c|c|c|c|c|c|}
    \hline
    q_+, q_- / 4 n & 4 & 20 & 40 & 60 & 80 & 100 & 120 & 140 & 160 
    \\
    \hline
    p = 1 & 1, 1 & 5, 5 & 10, 10 & 15, 10 & 20, 5 & 25, 1 & 30, 26 & 35, 21 &
    40, 16 
    \\
    \hline
    p \geqslant 2 & 1, 1 & 5, 5 & 10, 10 & 15, 15 & 20, 18 & 25, 13 & 30, 8 &
    35, 3 & 40, 40 
    \\
    \hline
  \end{array}$$
  \caption{\label{table4}}
\end{table}

\end{document}